\documentclass[11pt]{article}
\usepackage[margin = 1in]{geometry}
\usepackage[affil-it]{authblk}
\usepackage{lipsum}
\usepackage{amsfonts}
\usepackage{graphicx}
\graphicspath{ {images/} }
\usepackage{epstopdf}
\usepackage{dsfont}
\usepackage{mathtools}
\usepackage{empheq}
\usepackage{amsfonts}
\usepackage{amsgen,amsthm,amsmath,amstext,amsbsy,amsopn,amssymb,subcaption,stmaryrd,mathabx,mathrsfs}
\usepackage{comment}
\usepackage[font=footnotesize,labelfont=bf]{caption}

\usepackage{blkarray}

\usepackage{url}
\usepackage{longtable}
\usepackage{mathtools}
\usepackage{multirow}
\usepackage{array}

\usepackage{enumitem}
\usepackage{hyperref}
\usepackage{natbib}
\usepackage{booktabs}
\usepackage{algorithm}
\usepackage{algpseudocode}
\algrenewcommand\algorithmiccomment[1]{\hfill\texttt{//~#1}}
\usepackage{xcolor}
\usepackage[utf8]{inputenc} 
\usepackage[T1]{fontenc}

\ifpdf
  \DeclareGraphicsExtensions{.eps,.pdf,.png,.jpg}
\else
  \DeclareGraphicsExtensions{.eps}
\fi
\usepackage{template}

\newtheorem{Definition}{Definition}
\newtheorem{Theorem}{Theorem}

\newtheorem{Lemma}{Lemma}
\newtheorem{Remark}{Remark}
\newtheorem{Corollary}{Corollary}

\newtheorem{Proposition}{Proposition}

\newtheorem{Claim}{Claim}

\DeclareMathAlphabet\mathbfcal{OMS}{cmsy}{b}{n}

\usepackage{bbm}
\newcommand{\eps}{\varepsilon} % Epsilon

\hypersetup{
    colorlinks=true,
    linkcolor=blue,
    filecolor=blue,      
    urlcolor=cyan,
    citecolor=blue
}

\makeatletter
\newcommand*{\rom}[1]{\expandafter\@slowromancap\romannumeral #1@}
\makeatother

\allowdisplaybreaks

\begin{document}
	\title{Adaptive Confidence Intervals in Efron's Gaussian\\ Two-Groups Model}
    
	\author[1]{Qiaosen Wang}
    \author[2]{Shuwen Chai}
    \author[1]{Chao Gao\thanks{The research of CG is supported in part by NSF Grants ECCS-2216912 and DMS-2310769, and an Alfred Sloan fellowship.}}
    \affil[1]{Department of Statistics, University of Chicago}
    \affil[2]{Department of Computer Science, Northwestern University}
	\date{\today}
    
% 	\footnotetext[2]{Email: \texttt{chaogao@uchicago.edu}. The research of CG is supported in part by NSF
% Grants ECCS-2216912 and DMS-2310769, NSF Career Award DMS-1847590, and an Alfred Sloan fellowship. }
	\maketitle

\begin{sloppypar}	
%{\bf Keywords}: 
\begin{abstract}
Robust uncertainty quantification is increasingly important in modern data analysis and is often formalized under Huber’s model, which allows an $\varepsilon$-fraction of arbitrary corruptions. In many experimental sciences, however, the measurement protocol is well controlled: contamination is more plausibly introduced upstream (e.g., during source selection or labeling), after which all samples are measured through the same instrument noise.

Motivated by this noise-oblivious nature of adversaries, we study confidence intervals for the null location parameter $\theta$ in Efron’s Gaussian two-groups model, where an unknown fraction $\varepsilon$ of observations have arbitrarily shifted means, but all samples share the same law of additive Gaussian measurement noise with variance $\sigma^2$. We characterize the minimax-optimal length among confidence intervals with a prescribed coverage level uniformly over the unknown contamination proportion and all noise-oblivious adversaries. Although prior work has shown that the minimax point estimation rate of $\theta$ does not deteriorate when $\varepsilon$ becomes unknown, our results reveal that adaptation to an unknown $\varepsilon$ is intrinsically costly for confidence intervals. In particular, with a given $\sigma^2$, the minimax-optimal length of confidence intervals that are adaptive to unknown $\varepsilon$ is of order
$$
    \sigma \Big(n^{-1/4}+\frac{\varepsilon^{1/2}}{\sqrt{1\vee\log(en\varepsilon^2)}}\Big),
$$
which is polynomially worse than the optimal length when $\varepsilon$ is known. When the variance $\sigma^2$ is also unknown, we show a further degradation: no adaptive confidence interval can be shorter than $\Omega(\sigma n^{-1/8})$.

Algorithmically, we introduce a Fourier-based certification procedure built on Carath\'{e}odory's positive-semidefiniteness constraints. By scanning candidate points and accepting those whose residual characteristic function is certifiably consistent with a Gaussian location mixture, our algorithm attains the minimax lower bound in the known-variance setting and is computable in polynomial time.
\end{abstract}

\newpage
\tableofcontents

\section{Introduction}
The classical problem of constructing a confidence interval for the mean of a simple Gaussian distribution is by now completely understood. 
Given i.i.d. observations $X_{1:n}\iid \cN(0,\sigma^2)$, the textbook $Z$-interval
\begin{align*}
    \Big[\overline{X}-z_{1-\delta/2}\frac{\sigma}{\sqrt{n}},\ \overline{X}+z_{1-\delta/2}\frac{\sigma}{\sqrt{n}}\Big],\quad \overline{X}\define \frac{1}{n}\sum_{j=1}^n X_j
\end{align*}
and its adjustment for unknown variance, the $t$-interval
\begin{align*}
    \Big[\overline{X}-t_{1-\delta/2,\, n-1}\frac{S}{\sqrt{n}},\ \overline{X}+t_{1-\delta/2,\, n-1}\frac{S}{\sqrt{n}}\Big],\quad S^2\define \frac{1}{n-1}\sum_{j=1}^n (X_j-\overline{X})^2 
\end{align*}
are guaranteed to have coverage probability $1-\delta$ in the non-asymptotic sense. Their lengths are also information-theoretically optimal up to constant factors among all confidence intervals with the same coverage level. 

By contrast, modern large-scale inference problems depart substantially from this idealization. In applications such as single-cell genomics, web-scale product experimentation, and crowd-sourced language model evaluation, one observes a large collection of experimental outcomes, of which only an (often unknown) fraction are generated under a common null distribution. The remaining samples, affected by nonnull effects, may have arbitrarily shifted mean values but are still convolved with the same type of well-specified measurement noise. A widely adopted model for such data is Efron’s two-groups formulation \cite{efron2001empirical, efron2002empirical,efron2004large,efron2008microarrays, efron2008simultaneous}. In its most typical Gaussian incarnation, this model posits that
\begin{align}\label{eq:efron_model_intro}
    X_j\sim P_{\theta,\sigma,\varepsilon,Q}\define (1-\varepsilon)\cN(\theta,\sigma^2)+\varepsilon \int \cN(\eta,\sigma^2)\dif  Q(\eta),\qquad j=1,2,\cdots,n,
\end{align}
where $\theta\in\bbR$ is the null standard, $\varepsilon\in[0,1/2)$ is the proportion of nonnulls, and $Q$ is an unknown and arbitrary probability distribution that serves as the location prior of the nonnull fraction. Alternatively, from a robust statistics perspective, nulls are regarded as inliers, while nonnulls are treated as corrupted outliers. The Gaussian two-groups model then decomposes into an equivalent form
\begin{align*}
    X= \theta \mathbf{1}_n+\mathbf{\gamma}+Z,\qquad Z\sim \cN(\mathbf{0}_n,\sigma^2\bfI_n),\ \norm{\mathbf{\gamma}}_0\lesssim \varepsilon n,
\end{align*}
which instantiates Gaussian mean-shift/noise-oblivious contamination. From this new angle, $\theta$ is interpreted as the true mean of the clean distribution, $\varepsilon$ as the proportion of contamination, and $\mathbf{\gamma}$ (equivalently, $Q$) as a noise-oblivious adversary that acts before the convolution with Gaussian measurement noises. 

The location parameter $\theta$, either interpreted as the reference level to which nonnull effects are compared, or as the true center around which the bulk of clean data are generated, is naturally tied to a variety of statistical interests. Most of the existing theory on this model and its variants focuses on the point estimation of $\theta$ \cite{jin2007estimating, jin2008proportion, cai2010optimal,carpentier2021estimating,kotekal2025optimal, diakonikolas2025efficient, diakonikolas2026sample}. In particular, the work \cite{kotekal2025optimal} characterizes the minimax estimation rate for $\theta$ under the formulation \eqref{eq:efron_model_intro} without any additional assumptions on the nonnull/adversary $Q$. As they reveal, when $\varepsilon\leq \epsmax$ for any $\epsmax\in(0,1/2)$, the minimax rate is
\begin{align}\label{eq:estimation_rate}
    \abs{\widehat{\theta}-\theta}/\sigma\asymp \frac{1}{\sqrt{n}}+\frac{\varepsilon}{ \sqrt{1\vee\log(en\varepsilon^2)}}\asymp\begin{dcases}
        \frac{1}{\sqrt{n}},\quad &0\leq \varepsilon\leq \frac{1}{\sqrt{n}},\\
        \frac{\varepsilon}{\sqrt{\log(en\varepsilon^2)}},\quad &\frac{1}{\sqrt{n}}\leq \varepsilon\leq \epsmax,
    \end{dcases}
\end{align}
for both the nonadaptive setting where $\sigma^2$ and $\varepsilon$ are known, and the more challenging adaptive setting where $\sigma^2$ and/or $\varepsilon$ are unknown. However, the corresponding inference problem of building a confidence interval for $\theta$ remains largely unexplored. In the context of large-scale inference, constructing confidence intervals for $\theta$ is arguably more imperative than obtaining a sharp point estimator. In light of the fact that the estimate of $\theta$ may enter downstream procedures such as the calculation of p-values, what matters in practice is not merely that $\widehat{\theta}$ has a small risk, but that one can rigorously quantify the uncertainty in $\widehat{\theta}$
and determine how this uncertainty propagates through the multiple testing pipeline. This calls for confidence intervals for $\theta$ with a guaranteed coverage level whose lengths are as short as information-theoretically possible, especially in the adaptive regime where $\varepsilon$ and/or $\sigma^2$ are potentially unknown.

The goal of this paper is to build adaptive robust confidence intervals for $\theta$ under Efron's Gaussian two-groups model. Throughout this paper, the nonnull proportion $\varepsilon$ is always assumed to be unknown so that the confidence intervals must adapt to it. 
Regarding variance $\sigma^2$, we study both settings in which $\sigma^2$ is known and those in which $\sigma^2$ is unknown. Given a fixed maximum nonnull proportion $\epsmax\in(0,1/2)$ and a fixed error probability $\delta\in(0,1)$, we investigate those confidence intervals that achieve:
\begin{itemize}
    \item \textbf{high-probability coverage guarantee:}
    \begin{align}\label{eq:coverage_intro}
        \inf_{\theta\in\bbR,\, \sigma^2>0}\ \inf_{\varepsilon\in[0,\epsmax],\, Q}\ P_{\theta,\sigma,\varepsilon,Q}^n\big(\theta\in \CIhat\big)\geq 1-\delta,
    \end{align}
    \item \textbf{high-probability length guarantee:}
    \begin{align}\label{eq:length_intro}
        \inf_{\theta\in\bbR,\, \sigma^2>0}\ \inf_{\varepsilon\in[0,\epsmax],\, Q}\ P_{\theta,\sigma,\varepsilon,Q}^n\big(\abs{\CIhat}\leq r\big)\geq 1-\delta,
    \end{align}
\end{itemize}
where $\abs{\CIhat}\define \sup\CIhat-\inf\CIhat$ denotes the length, and $r$ is an upper bound of length that depends on $(n,\varepsilon,\sigma)$ and also on $\delta$ and $\epsmax$. The infimum over $\varepsilon\in[0,\epsmax]$ reflects the adaptivity to the actual contamination proportion; the infimum over $Q$ represents the robustness against arbitrary noise-oblivious adversaries. We are interested in the smallest possible length $r$ under a given coverage probability, and in the design of an algorithm that achieves this limit. 

\subsection{Two Key Motivations}
\subsubsection{Adaptivity Makes Confidence Intervals Different from Point Estimation}
Confidence intervals and point estimators are mutually transferable in nonadaptive settings. When all adaptivity requirements are neglected, and both $\sigma^2$ and $\varepsilon$ are assumed to be known, a valid confidence interval can be constructed simply based on a rate-optimal estimator $\widehat{\theta}$ (such as the one in \cite{kotekal2025optimal}) by:
\begin{align}\label{eq:CI_by_estimator}
    \Big[\widehat{\theta}-C\sigma \Big( \frac{1}{\sqrt{n}}+\frac{\varepsilon}{ \sqrt{1\vee\log(en\varepsilon^2)}}\Big),\ \widehat{\theta}+C\sigma \Big( \frac{1}{\sqrt{n}}+\frac{\varepsilon}{ \sqrt{1\vee\log(en\varepsilon^2)}}\Big)\Big].
\end{align}
This is because the high-probability deviation of $\widehat{\theta}$ can be controlled by a term that can be written out once both $\sigma^2$ and $\varepsilon$ are provided. From the lower-bound perspective, this is also the shortest confidence interval up to leading constants. However, when either $\sigma^2$ or $\varepsilon$ is unknown, the optimal estimation rate of $\theta$ is not available from the data, so we cannot perform the previous trick to build a confidence interval with a point estimator.
We may first estimate $\sigma^2$ and $\varepsilon$ and then plug these estimates into the above construction \eqref{eq:CI_by_estimator}. However, $\varepsilon$ is not identifiable from finite samples, and existing results for estimating $\varepsilon$ all require additional smoothness assumptions on the adversary $Q$ \cite{jin2007estimating,jin2008proportion,cai2010optimal}. On the other hand, for point estimation, one can apply Lepskii's method \cite{lepskii1991problem, lepskii1992asymptotically} to ensure adaptivity, as adopted in \cite{kotekal2025optimal}. In contrast, there is no such universal principle for confidence intervals that additionally requires coverage. 

Another naive solution to confidence intervals without knowing $\varepsilon$ is to apply the estimator-to-CI conversion \eqref{eq:CI_by_estimator} with the unknown $\varepsilon$ substituted by a reliable upper bound $\epsmax\in(0,1/2)$. In reality, it is impossible to precisely determine the exact contamination proportion $\varepsilon$. A crude and conservative upper bound $\epsmax\in(0,1/2)$ for $\varepsilon$, say $\epsmax=0.05$ or $\epsmax=0.49$, is practically more available than the actual $\varepsilon$. Since Efron's models are nested in the sense that $P_{\theta,\sigma,\varepsilon,Q}$ can be rewritten as $P_{\theta,\sigma,\epsmax,Q'}$ whenever $\varepsilon\leq \epsmax$, we can directly replace $\varepsilon$ with $\epsmax$ in \eqref{eq:CI_by_estimator} as if $\epsmax$ were the actual contamination proportion. However, in practice, $\epsmax$ is typically a constant, which, when we apply \eqref{eq:CI_by_estimator}, leads to a barely satisfactory length of $\sigma/\sqrt{\log(n)}$. This approach essentially applies the estimation-to-CI conversion in the most pessimistic and conservative manner and cannot achieve a faster rate when the true $\varepsilon$ is much smaller than $\epsmax$ (e.g., $\epsmax=0.49$ but $\varepsilon=0$).

In fact, we will illustrate that neither the ideal known-$\varepsilon$ rate of $\sigma(1/\sqrt{n}+\varepsilon/\sqrt{1\vee\log(en\varepsilon^2)})$ nor the conservative $\varepsilon=\epsmax$ rate of $\sigma/\sqrt{\log(n)}$ is the ultimate answer when $\varepsilon$ is unknown.
\begin{Claim}[Informal]
For adaptive confidence intervals on Efron's Gaussian two-groups model with fixed $\epsmax\in(0,1/2)$, we can adapt to an unknown $\varepsilon\leq \epsmax$ and do better than the conservative length of $\sigma/\sqrt{\log(n)}$, but we cannot fully achieve the known-$\varepsilon$ rate of $\sigma(1/\sqrt{n}+\varepsilon/\sqrt{1\vee\log(en\varepsilon^2)})$.
\end{Claim}

\subsubsection{Efron's versus Huber's: When the Analyst Is Also the Experimenter}
Efron's model can be viewed as a Gaussian distribution contaminated by a noise-oblivious adversary and is therefore a subclass of the more general Huber's contamination model \cite{huber1964robust}, which permits post-noise corruption. For Gaussian observations, Huber's contamination model is formulated by
\begin{align}\label{eq:huber_model}
    X_j\sim P_{\theta,\sigma,\varepsilon,Q}^{(\textnormal{Huber})}\define(1-\varepsilon)\cN(\theta,\sigma^2)+\varepsilon Q,\qquad j=1,2,\cdots,n. 
\end{align}
Compared to Efron's formulation \eqref{eq:efron_model_intro}, the outliers need not maintain Gaussianity; as a result, the optimal rate for point estimation and confidence intervals is notably slower. 

To assume Huber's or to assume Efron's is not merely a statistical issue but a choice between two distinct modeling philosophies. A standard data-generating process can be divided into two stages:
\begin{enumerate}
    \item \textbf{Sourcing:} In total, $n$ objects are gathered in the hope that they all share the same signal $\theta$ that we plan to infer;
    \item \textbf{Measurement:} Independent experiments are conducted to measure these objects. This adds a random measurement error to each object.
\end{enumerate}
If contamination can occur in both of the above stages, then we are forced to choose the more conservative Huber's assumption. This is the situation for pure analysts, who only see the final data and have no knowledge about the type of data corruption. 

However, this is not the case for highly unified modern experimental studies, where the analyst is typically also the experimenter. Compared with Huber's model, Efron's formulation is equivalent to the assertion that contamination occurs only in the sourcing stage, where a fraction of samples may carry shifted signals. Regarding the measurement step, Efron's model assumes that every sample's signal is convolved with a well-specified device/pipeline noise. This is a more realistic assumption for modern experimental designs, in which the measurement step is protected by top-tier devices and professional lab assembly. In contrast, in the sourcing step, an unknown fraction of erroneous objects may still slip through due to theoretical or technical limitations. For example, a trace amount of other uninterested cell types can mix with the targeted cells during a sequencing experiment, whereas the lab can guarantee proper sequencing of each sampled cell. As another example, some inadequate candidates may participate in an anonymous online survey for unknown reasons, whereas well-trained experimenters are confident in correctly evaluating each submitted response. This extra belief in noise-oblivious corruption enables the design of less conservative inference procedures that potentially achieve a faster rate.

The Huber counterpart of adaptive confidence intervals for a Gaussian location parameter has been studied in \cite{luo2024adaptive}, where the proposed procedure matches the minimax-optimal length of:
\begin{align*}
    r\asymp \sigma\Big(\frac{1}{\sqrt{\log(n)}}+\frac{1}{\sqrt{\log(1/\varepsilon)}}\Big).
\end{align*}
As Efron's model is a subclass of the more general Huber's formulation, one may ask whether the algorithm in \cite{luo2024adaptive} is already locally optimal in the noise-oblivious regime. 
The answer is no:
\begin{Claim}[Informal]
    Any procedure with guaranteed adaptive coverage over all Huber's models cannot attain the minimax-optimal length locally within Efron's subclass.
\end{Claim}
We will detail this result later in Section~\ref{sec:no_adaptation_between}. As this claim implies, we must adopt Efron's formulation before running the procedure if we believe in the noise-oblivious contamination and want to obtain a faster rate, since there is no adaptation across these two settings. Therefore, despite the existing methodologies for Huber's model in \cite{luo2024adaptive}, Efron's formulation is an essentially new and unexplored regime for the study of adaptive robust confidence intervals.

\subsection{Main Contributions}
The main contribution of this paper is two-fold:
\begin{itemize}
    \item \textbf{Fundamental hardness of adaptive inference:} We establish the minimax length of adaptive confidence intervals for $\theta$ in Efron's model when only $\sigma^2$ is known, characterized by
    \begin{align}\label{eq:known_var_length_intro}
        r\asymp \sigma\Big(\frac{1}{n^{1/4}}+\frac{\varepsilon^{1/2}}{\sqrt{1\vee\log(en\varepsilon^2)}}\Big),
    \end{align}
    which is polynomially slower than the aforementioned estimation rate~\eqref{eq:estimation_rate} from~\cite{kotekal2025optimal}. As this indicates, although nonadaptive estimation, adaptive estimation, and nonadaptive inference share the same rate \eqref{eq:estimation_rate}, adaptive inference is considerably harder than these three settings. In other words, adapting to an unknown $\varepsilon$ incurs additional sample complexity only when building confidence intervals, not when doing point estimation. When $\sigma^2$ is also unknown, we show that the rate will degrade further to no faster than $\Omega(\sigma n^{-1/8})$. Moreover, we discover an interesting phenomenon that the lower bound's exponent in $n$ undergoes two more sharp transitions as $\epsmax$ increases, one from $n^{-1/8}$ to $n^{-1/16}$ at $\epsmax=1/3$ and the other one from $n^{-1/16}$ to $n^{-1/24}$ at $\epsmax=7/15$.
    \item \textbf{Novel and rate-optimal Fourier-based technique:} Fourier-based deconvolution techniques have been proved to be successful in the estimation track of Efron's model~\cite{cai2010optimal, carpentier2019adaptive, comminges2021adaptive, kotekal2025optimal, diakonikolas2025efficient, diakonikolas2026sample}. We propose yet another Fourier-based technique, but anchored on the novel idea of certifying the nonnull/adversarial component of the empirical characteristic function using Carath\'{e}odory's positive-semidefiniteness (PSD) constraints. Scanning over each candidate value of $\theta$ in a crude pilot interval, we approve or refute it according to whether its implied nonnull part of the characteristic function is a \textit{bona fide} characteristic function. This is accomplished by checking whether the empirical plug-in estimate of the nonnull component satisfies a set of carefully designed PSD certificates up to certain slack terms. This recipe, computable in polynomial time, achieves the optimal rate~\eqref{eq:known_var_length_intro} when $\sigma^2$ is known, and a rate of $\sigma(n^{-1/8}+\varepsilon^{1/4}/\log^{1/2}(en\varepsilon^2))$ when $\sigma^2$ is unknown and $\epsmax\leq 1/3$.
\end{itemize}
\paragraph{Notation.}We use standard asymptotic notations: $O(\cdot)$ ($\lesssim$), $\Omega(\cdot)$ ($\gtrsim$), $\Theta(\cdot)$ ($\asymp$), and also $o(\cdot)$ and $\omega(\cdot)$. The subscripts in asymptotic notation are treated as constants. For example, $a\lesssim_\delta b$ (or $a=O_\delta(b)$ means that $a\lesssim b$ for each fixed $\delta$. For a complex number $z$, we use $\Re(z)$, $\Im(z)$, $\abs{z}$, $\mathrm{Arg}(z)\in(-\pi,\pi]$, and $\overline{z}$ to denote its real part, imaginary part, modulus, principal argument, and complex conjugate, respectively. The notations $\phi(t)$ and $\xi(t)$ are often reserved for characteristic functions. For a set $I\subseteq \bbR$, we use $\abs{I}=\sup I-\inf I$ to denote its diameter, which will frequently be called the \emph{length} of $I$ in the context of evaluating confidence intervals. The dependencies on $\delta$ and $\epsmax$ are often merged into the leading constants, whereas the dependencies on $n$, $\varepsilon$, and $\sigma^2$ are always clearly specified. 

\begin{table}
\begin{tabular}{lll}
\toprule
\textbf{Efron's Model} & \textbf{Estimation} \cite{kotekal2025optimal} & \textbf{Confidence Interval} (This) \\
\midrule
known $\varepsilon$, known $\sigma^2$  & $\sigma\Big(n^{-1/2}+\frac{\varepsilon}{ \sqrt{1\vee\log(n\varepsilon^2)}}\Big)\ (\star)$ & $\sigma\Big(n^{-1/2}+\frac{\varepsilon}{ \sqrt{1\vee\log(n\varepsilon^2)}}\Big)\ (\star)$ \\
unknown $\varepsilon$, known $\sigma^2$ &  $\sigma\Big(n^{-1/2}+\frac{\varepsilon}{ \sqrt{1\vee\log(n\varepsilon^2)}}\Big)\ (\star)$ & $\sigma\Big(n^{\color{red}-1/4\color{black}}+\frac{\varepsilon^{\color{red}1/2\color{black}}}{ \sqrt{1\vee \log(n\varepsilon^2)}}\Big)\ (\star)$ \\
unknown $\varepsilon$, unknown $\sigma^2$ &$\sigma\Big(n^{-1/2}+\frac{\varepsilon}{ \sqrt{1\vee\log(n\varepsilon^2)}}\Big)\ (\star)$ & $\sigma\Big(n^{\color{red}-1/8\color{black}}+\frac{\varepsilon^{\color{red}1/4\color{black}}}{ \sqrt{1\vee \log(n\varepsilon^2)}}\Big),\ \epsmax\in(0, 1/3]$\\
\bottomrule
\toprule
\textbf{Huber's Model} & \textbf{Estimation}  & \textbf{Confidence Interval} \cite{luo2024adaptive} \\
\midrule
known $\varepsilon$, known $\sigma^2$  & $\sigma\Big(n^{-1/2}+\varepsilon\Big)\ (\star)$ & $\sigma\Big(n^{-1/2}+\varepsilon\Big)\qquad\qquad\ (\star)$ \\
unknown $\varepsilon$, known $\sigma^2$ &  $\sigma\Big(n^{-1/2}+\varepsilon\Big)\ (\star)$ & $\sigma\Big(\frac{1}{\sqrt{\log(n)}}+\frac{1}{\sqrt{\log(1/\varepsilon)}}\Big)\ (\star)$ \\
unknown $\varepsilon$, unknown $\sigma^2$ &$\sigma\Big(n^{-1/2}+\varepsilon\Big)\ (\star)$ & $\sigma \qquad\qquad\qquad\qquad\quad\ \ (\star)$\\
\bottomrule
\end{tabular}
\caption{Rate upper bounds for estimation and confidence intervals in Efron's and Huber's settings. Those marked $(\star)$ are provably minimax rate-optimal. For the length of confidence intervals on Efron's model with unknown $\sigma^2$ and $\varepsilon$, we derive a partially matched lower bound of $\Omega(\sigma n^{-1/8})$ when $\epsmax\in (0,1/3]$.}
\label{tab:compare_bounds}
\end{table}

\subsection{Related Literature}
In this section, several relevant topics and lines of preceding work are discussed in order. 
\paragraph{Efron's two-groups model.} The two-groups model was systematically summarized in the seminal work of Efron \cite{efron2008microarrays} in the context of large-scale inference, with an empirical-Bayesian reflection on the meaning of the null hypothesis and how it ought to be chosen. The development of this model dates back to even earlier literature \cite{efron2001empirical, efron2002empirical, efron2004large}. The most general formulation of Efron's two-groups model states only a density decomposition
\begin{align*}
    f=(1-\varepsilon)f_\textnormal{null}+\varepsilon f_\textnormal{nonnull}.
\end{align*}
The null density $f_\textnormal{null}$ is conventionally assumed to be $\cN(0,1)$. Therefore, the p-values defined through $p_j=\Phi(X_j)$ are theoretically $U[0,1]$ under the null hypothesis. This theoretical null $\cN(0,1)$ appears to oversimplify reality, as subsequent lines of research have noted. For example, the empirical Bayesian pipeline in \cite{efron2001empirical} challenges the supposition $f_\textnormal{null}=\cN(0,1)$ by treating $f_\textnormal{null}$ as an unknown nonparametric density that needs to be estimated. Instead of going into nonparametric generality, a mild parametric correction $f_\textnormal{null}=\cN(\theta,\sigma^2)$, termed by Efron as the \emph{empirical null}, is already a reliable replacement for the theoretical null, as advocated in \cite{efron2004large} and also adopted in later works \cite{efron2008microarrays,efron2008simultaneous}. This new choice of null distribution necessitates determining the null parameters $\theta$ (and $\sigma^2$, if unknown) from the data in the presence of interference from the nonnull fraction. The nonnull density $f_\textnormal{nonnull}$ is largely unspecified in the work on Efron's models. However, a line of work has adopted the structural assumption that $f_\textnormal{nonnull}$ follows a location-scale mixture of Gaussians either due to a good match on real-world data \cite{efron2004large, jin2008proportion} or for the brevity of theoretical characterization \cite{jin2007estimating, jin2008proportion, cai2010optimal}. We especially emphasize that in all these works, the scale part of the nonnull mixing prior is always assumed to be supported on $[\sigma,+\infty)$ where $\sigma^2$ is the null variance. This assumption is consistent with the practical intuition that nonnull distributions tend to be more dispersed than the null distribution. More importantly, this restricted type of location-scale mixture is essentially a location-only mixture, as seen from the following equivalence
\begin{align*}
    f_\textnormal{nonnull}=\int \cN(\eta,\tau^2)\dif H(\eta,\tau)=\int \cN(\eta,\sigma^2)\dif Q(\eta)
\end{align*}
as long as $\mathrm{supp}(H)\subseteq \bbR\times [\sigma,+\infty)$, where 
\begin{align*}
    Q(\eta)=\int \cN(\eta,\tau^2-\sigma^2) \dif H(\tau\mid \eta).
\end{align*}
The formulation is therefore reduced to the location mixture introduced previously in \eqref{eq:efron_model_intro}.

A substantial line of work \cite{jin2007estimating, jin2008proportion, cai2010optimal, carpentier2019adaptive, carpentier2021estimating, kotekal2025optimal, diakonikolas2025efficient, diakonikolas2026sample} has contributed to the theory of Efron's model and its variants under formulation \eqref{eq:efron_model_intro} and its close variants. The work of Cai and Jin \cite{jin2007estimating, cai2010optimal} adopts the Gaussian incarnation of Efron's model and addresses the minimax rate of estimating the null location and scale parameters and the proportion of nonnull. They leverage a key observation that $(\theta,\sigma^2)$ can be approximated by the characteristic function $\phi(t)$ through
\begin{align*}
    \theta\approx \frac{\Im(\overline{\phi}(t)\phi'(t)) }{\abs{\phi(t)}^2},\quad \sigma^2\approx -\frac{\frac{\dif}{\dif t}\abs{\phi(t)}}{t\abs{\phi(t)}}.
\end{align*}
These approximations yield practical estimators under proper assumptions, in particular, the existence of a universal upper bound on the $q$-th moment of $Q$ for some fixed positive integer $q$. Such moment-boundedness ensures that: first, $\abs{\phi(t)}$ are away from $0$ when $t$ is in a sufficiently wide neighborhood of $0$, which stabilizes the denominators in the above estimators; second, $\phi(t)$ admits a certain regularity near $t= 0$, which enables a sharp convergence of $\abs{\phi_n(t)-\phi(t)}$ uniformly over all $t$ in a $0$-neighborhood.

The works of~\cite{carpentier2019adaptive} and \cite{carpentier2021estimating} also provide new insights into the theory of Efron's model. In \cite{carpentier2019adaptive}, the authors address the problem of estimating $\varepsilon$ when $\theta=0$ through a reduction to testing whether the nonnull shift vector $\gamma$ is $k_0$-sparse or $(k_0+\Delta)$-sparse but separated from $\bbB_0(k_0)$ in $\ell_2$ distance. Their procedure is a combination of three tests, two of which are heavily Fourier-motivated. Interestingly, although their problem is not directly linked to ours, they also observe a discontinuous jump in the minimax lower bound when $k_0=0$ and $\Delta/n$ is around $1/3$, which we formalize and strengthen in this paper from the perspective of moment matching. 
The work of Carpentier et al. \cite{carpentier2021estimating} considers parameter estimation of $\theta$ on a variant of Efron's model called one-side contamination, in which the corrupted distribution can only stochastically dominate the null distribution from one known side. This one-side contamination model is a stylized restriction of the standard formulation \eqref{eq:efron_model_intro}; the parameter $\theta$ is identifiable in this model even when $\varepsilon$ reaches or goes beyond $1/2$. In \cite{carpentier2021estimating}, the feature of this one-sided adversary is fully exploited, and an optimal algorithm is designed based on the interplay between moment-generating functions (Laplace transform) and Chebyshev polynomials.

A more recent work~\cite{kotekal2025optimal} addresses estimating the null parameters $\theta$ and $\sigma^2$ precisely under the Gaussian two-groups formulation~\eqref{eq:efron_model_intro}. A crucial property essentially used in their paper is that
\begin{align*}
    \abs{\xi(t)}\leq 1,\quad \forall\, t\in\bbR,
\end{align*}
for any characteristic function $\xi(t)$ corresponding to a probability measure. As we will see later in this paper, this condition happens to be the simplest one of Carath\'{e}odory's PSD constraints \cite{caratheodory1911variabilitatsbereich}. With this property, they propose to estimate $\theta$ by
\begin{align*}
    \widehat{\theta}\in\ \argmin_{\mu\in\bbR}\, \sup_{\abs{t}\leq \tau}\, \inf_{\xi\in\bbC,\,\abs{\xi}\leq 1}\, \abs*{e^{-\im \mu t+\sigma^2 t^2/2}\phi_n(t)-(1-\varepsilon)-\varepsilon \xi},
\end{align*}
when both $\sigma^2$ and $\varepsilon$ are available. Adaptation to unknown $\sigma^2$ and/or $\varepsilon$ is also derived based on pilot estimators from a holdout set and Lepskii's method, without incurring any adaptation cost. This result has been extended by subsequent works to high-dimensional \cite{diakonikolas2025efficient} and non-Gaussian \cite{diakonikolas2026sample} settings, and also serves as the starting point for our current paper. However, the plot thickens when constructing a confidence interval, an inference task that additionally demands high-probability coverage. We are compelled to devise a way to quantify the magnitude of uncertainty and explore more of the delicate structures in a genuine characteristic function.

\paragraph{Huber's contamination.}As mentioned earlier, Efron's model can be viewed as a Gaussian distribution contaminated by an oblivious mean-shift adversary, and is therefore a subclass of the more general concept of Huber's contamination \cite{huber1964robust}. For Gaussian observations, Huber's contamination model is formulated as \eqref{eq:huber_model}. It is well-known that the minimax estimation rate under Huber's contamination model is given by $\sigma(1/\sqrt{n}+\varepsilon)$, which can be achieved by the sample median without knowing $\varepsilon$. Thus, when $\varepsilon$ and $\sigma^2$ are both given, a rate-optimal confidence interval can be given by
\begin{align}
    \left[\widehat{\theta}-C\sigma\left(\frac{1}{\sqrt{n}}+\varepsilon\right),\ \widehat{\theta}+C\sigma\left(\frac{1}{\sqrt{n}}+\varepsilon\right)\right],
\end{align}
where $\widehat{\theta}$ can be any rate-optimal estimator such as the sample median. Nevertheless, this routine becomes infeasible when $\varepsilon$ is unknown, similar to what we previously noted for Efron's model. Regarding adaptive robust confidence intervals under Huber's contamination, \cite{luo2024adaptive} shows that the shortest possible length scales as $\sigma (1/\sqrt{\log(n)}+1/\sqrt{\log(1/\varepsilon)})$ when $\sigma^2$ is known but $\varepsilon$ is unknown and lies within $[0,\epsmax]$ for some $\epsmax\in(0,1/2)$. This rate reveals how difficult it is to build adaptive confidence intervals for such a model. Even when the data are truly immaculate, we have to make do with a barely satisfactory length of $\sigma/\sqrt{\log(n)}$ because we are not informed of the cleanliness before our inference procedure. This signals a significant gap between adaptively estimating $\theta$ and adaptively inferring $\theta$. Also delivered by \cite{luo2024adaptive} is that the minimax length turns into $\Theta(\sigma)$ as soon as $\sigma^2$ becomes unknown, which means that insisting on variance-adaptivity will rule out the hope of obtaining a shrinking confidence interval. These results are in violent contrast to those we obtain on Efron's model, where a rate of $\sigma(n^{-1/4}+\varepsilon^{1/2}/\sqrt{1\vee \log(en\varepsilon^2)})$ is achievable when $\sigma^2$ is known, and a rate of $\sigma(n^{-1/8}+\varepsilon^{1/4}/\sqrt{1\vee \log(en\varepsilon^2)})$ when $\sigma^2$ is unknown, both of which are no slower than $\sigma/\sqrt{\log(n)}$. See Table~\ref{tab:compare_bounds} for comparison.

\paragraph{Characteristic function as trigonometric moments.} Statistical application of characteristic functions has been mentioned multiple times in previous discussions. Here, we present a special interpretation of the characteristic function as trigonometric moments. This becomes obvious once we write that
\begin{align}
    \phi(kt_0)=\E_P[(e^{\im t_0 X})^k]&=\E_P[\cos(kt_0 X)]+\im \E_P[\sin(kt_0 X)]\nonumber\\
    &=\E_P[T_k(\cos(t_0 X))]+\im \E_P[\sin(t_0 X)U_k(\cos(t_0 X))],\nonumber
\end{align}
where $T_k$ is the degree-$k$ Chebyshev's polynomial of the first kind, and $U_k$ is the degree-$k$ Chebyshev's polynomial of the second kind (see e.g., \cite{carothers1998short}). From this perspective, $\{\phi(kt_0)\}_{k=0}^\infty$ encodes the moments of $\exp(\im t_0 X)$, or equivalently those of $\cos(t_0X)$ and $\sin(t_0X)$. Moment information and associated polynomial inequalities have been applied to robustly estimate a Gaussian distribution under both Huber's contamination \cite{hopkins2018mixture, diakonikolas2018robustly, dalalyan2022all} and the weaker mean-shift contamination \cite{diakonikolas2025efficient}. However, in this line of work, the plain algebraic moments $\E_P[X^k]$ evaluated in a subset of samples, not the characteristic function or the trigonometric moments, are actually used. The main reason is that those papers consider strong contamination, in which the corrupted fraction need not come from a true probability distribution. Therefore, there is no further structure in the corrupted part of the data beyond the fact that corrupted samples occupy at most $\varepsilon$-fraction; thus, it is natural to filter them out by identifying the majority that appear clean. By contrast, Efron's model is an example of weak contamination in which a delicate Gaussian structure is embedded within its adversarial component. This additional structure makes it indispensable to explore both the clean and the contaminated parts. Na\"{i}vely pushing the plain moments to the corrupted fraction will cause an issue: the usual algebraic moments $\E_{Q*\cN(0,\sigma^2)}[X^k]$ may not exist due to the arbitrariness of $Q$. But the characteristic function, being the moments of the bounded random variable $\exp(\im t_0 X)$, always exists regardless of how dispersive $Q$ is. We summarize this rationale as follows:

{\centering\itshape
Characteristic functions are robustified moments.\par}

Like algebraic polynomials, trigonometric polynomials also obey systems of inequalities. Recall that by Sylvester's criterion, a matrix is PSD if and only if all of its minors have nonnegative determinants. Therefore, the Carath\'{e}odory's PSD constraints that we will use in this paper (introduced later in Lemma~\ref{lem:cara_psd}) reduce to a set of trigonometric inequalities. Thus, our methodology can be regarded as the trigonometric analog of algebraic sum-of-squares proofs. To the best of our knowledge, applying these moment-based certifications to characteristic functions is relatively new in statistics. This trigonometric moments tool is also used in recent work \cite{ma2025best}, which targets the approximation rates for several density classes by Gaussian location mixtures.

\subsection{Paper Organization}
The remainder of the paper is organized as follows.
Section~\ref{sec:overview} presents the high-level idea of our certification-based construction of adaptive robust confidence intervals. Sections~\ref{sec:upper_bound_known_var} and~\ref{sec:upper_bound_unknown_var} apply this methodology to obtain upper bounds for the known-variance and unknown-variance settings, respectively.
Section~\ref{sec:lower_bound} completes the theory by providing information-theoretic lower bounds.
Section~\ref{sec:discussion} serves as an extended discussion of several key intuitions. The performance of our procedures and minimax lower bounds are proved in Appendices~\ref{sec:proof_ub} and \ref{sec:proof_lb}, respectively. Other proofs and auxiliary results are deferred to Appendices~\ref{sec:other_proof} and \ref{sec:aux}.

\section{Preliminaries and Overview of Techniques}\label{sec:overview}
In this section, we cover the preliminary setup and outline the algorithmic design at a high level.

\subsection{Certifying the Adversarial Component in the Characteristic Function}\label{sec:certify_adversary}
We illustrate how the construction of adaptive robust confidence intervals can be tamed using a Fourier-based recipe. Starting with the known-variance setting where $\sigma^2>0$ is given, and fixing a constant $\epsmax\in(0,1/2)$, we recall Efron's model:
\begin{align*}
    P_{\theta,\sigma,\varepsilon,Q}=(1-\varepsilon)\cN(\theta,\sigma^2)+\varepsilon Q*\cN(0,\sigma^2),\quad \theta\in\bbR,\ \varepsilon\in[0,\epsmax].
\end{align*}
Its characteristic function $\phi(t)=\E_{P_{\theta,\sigma,\varepsilon,Q}}[\exp(\im tX)]$ reads
\begin{align}\label{eq:efron_cf}
    \phi(t)=(1-\varepsilon)e^{\im\theta t-\sigma^2t^2/2}+\varepsilon e^{-\sigma^2t^2/2} \xi(t),
\end{align}
where $\xi(t)=\E_Q[\exp(\im tX)]$ is the characteristic function of the adversary $Q$. To recover $\theta$ from finite samples, the most natural idea is to extract an $\exp(\im \widehat{\theta} t-\sigma^2t^2/2)$ component that best fits the empirical characteristic function $\phi_n(t)=\E_{P_n}[\exp(\im tX)]$; then this $\widehat{\theta}$ should be close to the truth $\theta$. However, to construct a confidence interval, an opposite route is taken: instead of extracting the benign component, we scan over candidates $\mu\in\bbR$ to recover the adversarial component $\xi(t)$. If $\theta$ and $\varepsilon$ are given, this adversarial component can be recovered from $\phi(t)$ using \eqref{eq:efron_cf} through
\begin{align*}
    \xi(t)=e^{\im \theta t}\cdot \frac{e^{-\im \theta t+\sigma^2t^2/2}\phi(t)-(1-\varepsilon)}{\varepsilon}.
\end{align*}
When $\theta$ is unknown and we suspect that a candidate $\mu\in\bbR$ is the true $\theta$, we can define accordingly
\begin{align}\label{eq:xi_mu}
    \xi(t;\mu)\define e^{\im \mu t}\cdot \frac{e^{-\im \mu t+\sigma^2t^2/2}\phi(t)-(1-\varepsilon)}{\varepsilon},
\end{align}
with $\xi(t;\theta)\equiv \xi(t)$ being its most special case. The rationale for pursuing $\xi(t;\mu)$ is as follows: a candidate $\mu$ is the true $\theta$ if and only if $\xi(t;\mu)$, as a function of $t$, is the characteristic function of a probability measure (referred to as a \emph{genuine characteristic function} from now on). In fact, $\xi(t;\mu)$ is the characteristic function of the signed measure
\begin{align}\label{eq:xi_signed_measure}
    Q+\frac{1-\varepsilon}{\varepsilon}(\delta_\theta-\delta_\mu),
\end{align}
which becomes a probability measure only when $\mu=\theta$. Therefore, we are inspired to form a confidence interval by collecting all $\mu\in\bbR$ whose corresponding $\xi(t;\mu)$ appears to be a genuine characteristic function. However, $\varepsilon$ is always assumed to be unknown in our setting of adaptive robust confidence intervals, and therefore $\xi(t;\mu)$ defined by \eqref{eq:xi_mu} is not available. Instead, we recall that $\varepsilon\leq \epsmax<1/2$ and replace $\varepsilon$ by $1/2$ in the definition \eqref{eq:xi_mu}; we also drop the rotation factor $e^{\im \mu t}$ in \eqref{eq:xi_mu} since this does not affect whether a complex-valued function is a genuine characteristic function or not. These modifications bring us to the new definition of
\begin{align}\label{eq:known_var_upsilon}
    \Upsilon(t;\mu)\define\frac{e^{-\im \mu t+\sigma^2t^2/2}\phi(t)-1/2}{1/2}=2e^{-\im \mu t+\sigma^2t^2/2}\phi(t)-1.
\end{align}
From \eqref{eq:efron_cf}, we know that
\begin{align*}
    \Upsilon(t;\mu)=2\varepsilon e^{-\im \mu t}\xi(t)+2(1-\varepsilon)e^{\im (\theta-\mu)t}-1,
\end{align*}
which is the characteristic function of the signed measure
\begin{align}\label{eq:upsilon_signed_meausre}
      2\varepsilon \delta_{-\mu}*Q+2(1-\varepsilon)\delta_{\theta-\mu}-\delta_0.
\end{align}
Similar to \eqref{eq:xi_signed_measure}, only when $\mu=\theta$ can the signed measure in \eqref{eq:upsilon_signed_meausre} become a probability measure, as summarized by the following lemma:
\begin{Lemma}\label{lem:true_candidate_known_var}
    When the variance $\sigma^2>0$ is given and $\varepsilon\in[0,1/2)$, a candidate $\mu$ is the true $\theta$ if and only if the function
    \begin{align*}
         \Upsilon(t;\mu)\define 2e^{-\im \mu t+\sigma^2t^2/2}\phi(t)-1,
    \end{align*}
    as a function of $t$, is the characteristic function of a probability measure. 
\end{Lemma}
See Appendix~\ref{sec:proof_chf} for the proof. Because only finite samples are provided, we may approximately recover $\Upsilon(t;\mu)$ from the empirical characteristic function $\phi_n(t)$ with a plug-in estimate:
\begin{align}\label{eq:known_var_upsilon_n}
    \Upsilon_n(t;\mu)\define \frac{e^{-\im \mu t+\sigma^2t^2/2}\phi_n(t)-1/2}{1/2}=2e^{-\im \mu t+\sigma^2t^2/2}\phi_n(t)-1.
\end{align}
The following route is now open to us:
\begin{enumerate}
    \item start with $\CIhat=\emptyset$, scan over $\mu$ in a pilot region, compute the function $\Upsilon_n(\cdot\,;\mu)$ for each $\mu$;
    \item certify whether $\Upsilon_n(\cdot\,;\mu)$ is \emph{approximately} a genuine characteristic function; if true, then add this $\mu$ to $\CIhat$.
\end{enumerate}
Under the unknown-variance setting and the condition that $\epsmax$ is a constant in $(0,1/3]$, we add a new parameter $\lambda$ that serves as a guessed candidate of the variance $\sigma^2$. In analogy to the known-variance setting, we define
\begin{align}\label{eq:unknown_var_upsilon}
    \Upsilon(t;\mu,\lambda)\define \frac{e^{-\im \mu t+\lambda t^2/2}\phi(t)-2/3}{1/3}=3e^{-\im \mu t+\lambda t^2/2}\phi(t)-2
\end{align}
and its plug-in estimate 
\begin{align}\label{eq:unknown_var_upsilon_n}
    \Upsilon_n(t;\mu,\lambda)\define 3e^{-\im \mu t+\lambda t^2/2}\phi_n(t)-2.
\end{align}
The constants are different from the known-variance setting because we assume $\epsmax\in(0,1/3]$, not $\epsmax\in(0,1/2)$. In analogy to Lemma~\ref{lem:true_candidate_known_var}, we claim that:
\begin{Lemma}\label{lem:true_candidate_unknown_var}
    When the variance $\sigma^2>0$ is unknown and $\varepsilon\in[0,1/3]$, a candidate $\mu$ is the true $\theta$ if and only if there exists $\lambda>0$ such that the function
    \begin{align*}
        \Upsilon(t;\mu,\lambda)\define 3e^{-\im \mu t+\lambda t^2/2}\phi(t)-2,
    \end{align*}
    as a function of $t$, is the characteristic function of a probability measure. Furthermore, if such a $\lambda>0$ exists, it must equal $\sigma^2$.
\end{Lemma}
The proof is also deferred to Appendix~\ref{sec:proof_chf}. The following pipeline, as an adjustment of the previous known-variance case, is reconciled with the introduction of a variance candidate:
\begin{enumerate}
    \item start with $\CIhat=\emptyset$, scan over $\mu$ in a pilot region, and for each of these $\mu$, scan over $\lambda$ in another pilot region; then compute the function $\Upsilon_n(\cdot\,;\mu,\lambda)$ for each pair of $(\mu,\lambda)$;
    \item certify whether $\Upsilon_n(\cdot\,;\mu,\lambda)$ is \emph{approximately} a genuine characteristic function; if true, then add this $\mu$ to $\CIhat$.
\end{enumerate}

Two concepts in the above methodology remain unclear: how should we certify the genuineness of a candidate characteristic function, and how close can we say \textit{approximately}? 

\subsection{Carath\'{e}odory's Positive Semidefiniteness (PSD) Constraints}\label{sec:cara_psd}
We certify a candidate characteristic function using the positive semidefiniteness conditions by Carath\'{e}odory \cite{caratheodory1911variabilitatsbereich} (see also \cite{schmudgen2017moment} Theorem 11.5):
\begin{Lemma}\label{lem:cara_psd}
For any $\phi:\bbR\to \bbC$ being the characteristic function of a probability distribution, it must satisfy
\begin{enumerate}
    \item \textbf{normalization:} $\phi(0)=1$;
    \item \textbf{Hermitian:} $\phi(-t)=\overline{\phi(t)}$ for any $t\in\bbR$;
    \item \textbf{positive semidefiniteness:} for any integer $m\in\bbN^*$ and any $t\in\bbR$,
    \begin{align}
    \bfT_m[\phi](t)\define
    \begin{bmatrix}
    z_0 & z_1 & z_2 & \cdots & z_m \\
    \overline{z}_1 & z_0 & z_1 & \cdots & z_{m-1} \\
    \overline{z}_2 & \overline{z}_1 & z_0 & \cdots & z_{m-2} \\
    \vdots & \vdots & \vdots & \ddots & \vdots \\
    \overline{z}_m & \overline{z}_{m-1} & \overline{z}_{m-2} & \cdots & z_0
    \end{bmatrix}\succeq 0,
    \end{align}
    where $z_k=\phi(kt)$ for $k=0,1,2,\cdots,m$.
\end{enumerate}
\end{Lemma}
\begin{proof}
The first and second conclusions are standard. To show the third conclusion on positive semidefiniteness, we notice that for any complex vector $\boldsymbol{v}\in \bbC^{m+1}$,
\begin{align*}
    \overline{\boldsymbol{v}}^\top \bfT_m[\phi](t)\boldsymbol{v}=\E_P\Big[\abs{\sum_{k=0}^m v_k e^{\im k t X}}^2\Big]\geq 0,
\end{align*}
which concludes the proof.
\end{proof}
Normalization and Hermitian property are naturally satisfied for all our candidate characteristic functions, either $\Upsilon(t;\mu)$ or $\Upsilon(t;\mu,\lambda)$, or for their finite sample estimators. The only property we need to certify is whether these candidate characteristic functions satisfy the PSD constraints. When $m=1$, the above condition is a modulus-based condition
\begin{align}\label{eq:known_var_criterion_overview}
    \abs{\Upsilon(t;\mu)}\leq 1,
\end{align}
which is an unknown-$\varepsilon$ generalization of the condition $\abs{\xi(t;\mu)}\leq 1$ that Kotekal and Gao essentially leverage in \cite{kotekal2025optimal}. As we show later in Section~\ref{sec:upper_bound_known_var}, this single certificate suffices for the known-variance setting. For the unknown-variance setting with $\epsmax\in(0,1/3]$, we choose to set $m=2$ and exploit the PSD property of a $3\times 3$ Hermitian-Toeplitz. It will be shown later in Section~\ref{sec:upper_bound_unknown_var} that this condition is equivalent to
\begin{align}\label{eq:unknown_var_criterion_overview}
    \abs{\Upsilon(t;\mu,\lambda)}\leq 1,\quad \abs{\Upsilon^2(t;\mu,\lambda)-\Upsilon(2t;\mu,\lambda)}+\abs{\Upsilon(t;\mu,\lambda)}^2\leq 1.
\end{align}
\subsection{Finite-Sample Slackness, Frequency Gridding, and Wrapping-Unwrapping}\label{sec:additional_issues}
Despite its theoretical elegance, Carath\'{e}odory's PSD constraints are not directly applicable unless the following three issues are addressed:
\paragraph{Finite-sample slackness.}In practice, we only have finite samples at hand. Therefore, $\Upsilon_n$, instead of the population-level $\Upsilon$, is what we can directly reconstruct from the data. Due to the random deviation between $\Upsilon_n$ and $\Upsilon$, we cannot expect the previous inequalities \eqref{eq:known_var_criterion_overview} and \eqref{eq:unknown_var_criterion_overview} to hold at the truth when $\Upsilon$ is replaced by $\Upsilon_n$. To avoid over-rigorous criteria that can mistakenly reject the true $\theta$, we leave a room of tolerance by adding slack terms to the right-hand side of the inequalities. The magnitudes of such slack terms are chosen by carefully quantifying the finite-sample plug-in error $\abs{\Upsilon_n-\Upsilon}$. For example, in the known-variance setting, we propose to certify the candidate by checking whether
\begin{align*}
    \abs{\Upsilon_n(t;\mu)}\leq 1+\Delta(t)
\end{align*}
for some slack term that takes the form $\Delta(t)=C_1 e^{C_2\sigma^2t^2}/\sqrt{n}$. We defer the details to the proofs in Section~\ref{sec:upper_bound_known_var} and \ref{sec:upper_bound_unknown_var}.

\paragraph{Frequency gridding.}Another issue seems even more urgent. Ideally, when $\Upsilon$ is a genuine characteristic function, Carath\'{e}odory's PSD constraints should hold for all frequencies $t\in\bbR$. The more frequencies $t$ we examine, the less likely we are to let in a false candidate. However, increasing the number of frequency checkpoints also causes two issues. First, screening over all $t\in\bbR$ is barely tractable. As another drawback, more frequency checkpoints will incur higher prices for the uniform concentration of $\abs{\phi_n(t)-\phi(t)}$. Compared with the setup in \cite{jin2007estimating, cai2010optimal}, we drop their assumption on a universal moment upper bound for $Q$; thus, $\phi(t)$ in our setting is of poor regularity, and we cannot even hope for $\abs{\phi_n(t)-\phi(t)}$ to uniformly concentrate within a universal neighborhood of $0$. 

To find a balance, we propose checking the certification inequalities only on a particularly designed finite subset of frequencies. Taking the known-variance setting as an instance, we certify an empirical candidate $\Upsilon_n$ by checking whether
\begin{align*}
    \abs{\Upsilon_n(t;\mu)}\leq 1+\Delta(t),\quad \forall\, t\in\cT,
\end{align*}
where $\cT$ is a collection of $O(\log(n))$ frequencies specified in Section~\ref{sec:upper_bound_known_var} by \eqref{eq:known_var_grid}. The proof therein justifies the sufficiency of looking at such a finite grid. When an unknown variance arises, the construction of grid becomes data-dependent and will make use of a holdout set of samples.

\paragraph{Wrapping and unwrapping.}We end this section by remedying an inherent limitation of this Fourier-based certification approach. Consider the certificate at a single frequency $t$, for example, $\abs{\Upsilon_n(t;\mu)}\leq 1+\Delta(t)$ in the known-variance setting. From the definition, we know that $\Upsilon_n(t;\mu)$ is a periodic function in $\mu$; $2\pi/t$ is one of its periods. As this implies, if a candidate $\mu$ passes the certificate at a frequency $t\in\cT$, all points in the form $\mu+2k\pi/t:\ k\in\bbZ$ will also pass the same certificate. Consequently, all candidates $\mu$ that pass $\abs{\Upsilon_n(t;\mu)}\leq 1+\Delta(t)$ form a \emph{wrapped} set of recurring intervals, each consecutive two of which are spaced by a margin of $2\pi/t$. Even after taking the intersection over all $t\in\cT$, the resulting confidence set is likely to take a fence-like shape, spuriously tending to infinity. For no reason should we deem this spurious set as a desirable confidence region.

The Fourier-based recipe can adapt to the fine-grained structures in Efron's models, but is, in the meantime, susceptible to its inherent periodicity. On the other hand, there are coarse pilot confidence intervals for $\theta$, such as conservatively applying the estimator-to-CI conversion \eqref{eq:CI_by_estimator} with $\varepsilon=\epsmax$. These pilot confidence intervals are generally too long (e.g.,  $\Theta(\sigma/\sqrt{\log(en)})$) compared to the desired minimax length, and usually do not have adaptation to $\varepsilon$. Nevertheless, a combination of a coarse yet non-periodic pilot interval and a fine-grained yet wrapped Fourier-based confidence set can yield the best of both worlds, on which we will elaborate in Section~\ref{sec:upper_bound_known_var} and \ref{sec:upper_bound_unknown_var}. At a high level, we first form a wrapped confidence set using the Fourier-based approach, where a desired short confidence set is hidden in one of its recurring branches; then, we unwrap it by intersecting it with a pilot confidence interval to identify the correct branch. Equivalently, we can first construct a pilot confidence interval as \emph{localization} and then implement our Fourier-based filtering within this interval as \emph{refinement}.

There are abundant versions of pilot estimators, among which we are interested in a $\Theta(\sigma/\sqrt{\log(en)})$-rate mean estimator and a $\Theta(\sigma^2/\log(en))$-rate variance estimator:
\begin{Lemma}\label{lem:pilot_unknown_var}
    There exist a pilot mean estimator $\widetilde{\theta}$ and a pilot variance estimator $\widetilde{\sigma}^2$ that meet the following: fix any $\delta\in(0,1)$ and $\epsmax\in(0,1/2)$, there are known absolute constants $L$, $M$, and $N_0$ that depend only on $\delta$ and $\epsmax$, such that whenever $n\geq N_0$,
    \begin{align}
        \inf_{\theta\in\bbR,\, \sigma^2>0,\, Q}\ P_{\theta,\sigma,\epsmax,Q}^n\Big(\abs{\widetilde{\theta}-\theta}\leq \frac{M\sigma}{\sqrt{\log(en)}},\ \frac{\abs{\widetilde{\sigma}^2-\sigma^2}}{\sigma^2}\leq \frac{L}{\log(en)}\Big)\geq 1-\frac{\delta}{2}.
    \end{align}
    Furthermore, both $\widetilde{\theta}$ and $\widetilde{\sigma}^2$ can be computed within $\mathrm{poly}(n)$ time.
\end{Lemma}

These two estimators can be obtained by setting $\varepsilon=\epsmax$ in Theorems 3.4 and 3.2 of \cite{kotekal2025optimal}, respectively. For the sake of completeness, we append the proof in Appendix~\ref{sec:other_proof_sec2}. Our proof is adapted from \cite{kotekal2025optimal} by adding a discretization step that clearly establishes polynomial-time computability.

\section{Construction of Confidence Intervals When Variance Is Known}\label{sec:upper_bound_known_var}
In this section, we apply our Fourier-based recipe to the known-variance setting. Given $\sigma^2>0$ and $\epsmax\in(0,1/2)$, we illustrate the coverage and tightness of the following confidence interval
\begin{align}\label{eq:CI_known_var}
    \CIhat=\left\{\mu\in \widetilde{\CI}:\ \cos\Big(t\mu-\Arg(\phi_n(t))\Big)\geq \frac{1-(1+\Delta(t))^2+4\abs{e^{\sigma^2t^2/2}\phi_n(t)}^2}{4\abs{e^{\sigma^2t^2/2}\phi_n(t)}},\, \forall\, t\in\cT\right\},
\end{align}
where 
\begin{itemize}
    \item $\phi_n(t)=\E_{P_n}[e^{\im t X}]$ is the empirical characteristic function;
    \item $\widetilde{\CI}$ is a pilot confidence interval derived from Lemma~\ref{lem:pilot_unknown_var}, defined as
    \begin{align}\label{eq:known_var_pilot_CI}
        \widetilde{\CI}=\Big[\widetilde{\theta}-\frac{M\sigma}{\sqrt{\log(en)}},\, \widetilde{\theta}+\frac{M\sigma}{\sqrt{\log(en)}}\Big],
    \end{align}
    where $M$ (specified by Lemma~\ref{lem:pilot_unknown_var}) is an absolute constant that depends only on $\delta$ and $\epsmax$;
    \item $\Delta(t)$ is a slack term:
    \begin{align}\label{eq:known_var_slack}
        \Delta(t)=2\kappa e^{\sigma^2 t^2}\sqrt{8\log(10/\delta)/n};
    \end{align}
    \item $\cT$ is a finite frequency set:
    \begin{align}\label{eq:known_var_grid}
        \cT=\big\{\sqrt{k}(\kappa \sigma)^{-1}:\ k=1,2,\cdots,\ceil{\log(en)}\big\},
    \end{align}
    where $\kappa$ is an absolute constant that depends only on $\delta$ and $\epsmax$.
\end{itemize}
The main aim is to show the following theorem.
\begin{Theorem}\label{thm:known_var}
    Assume $\sigma^2>0$ is known and fix any $\delta\in(0,1)$ and $\epsmax\in(0,1/2)$. Let $M$ be the $(\delta,\epsmax)$-dependent constant as in Lemma~\ref{lem:pilot_unknown_var}. For the $\CIhat$ defined by \eqref{eq:CI_known_var} with $\kappa\geq \sqrt{2}\vee M/\pi$, there exist absolute constants $C$ and $N$ that depend only on $\delta$, $\epsmax$, and $\kappa$, such that
    \begin{align*}
\inf_{\theta\in\bbR}\ \inf_{\varepsilon\in[0,\epsmax],\, Q}\ P^n_{\theta,\sigma,\varepsilon,Q}\Big(\theta\in\CIhat,\, \abs{\CIhat}\leq \overline{r}\Big)\geq 1-\delta
    \end{align*}
    whenever $n\geq N$, where
    \begin{align*}
        \overline{r}(n,\sigma,\varepsilon)=C\sigma\left(\frac{1}{n^{1/4}}+\frac{\varepsilon^{1/2}}{\sqrt{1\vee\log(en\varepsilon^2)}}\right).
    \end{align*}
\end{Theorem}

\subsection{Order-1 Certificates Based on Modulus}
We begin by showing how the confidence interval in \eqref{eq:CI_known_var} is derived from the certification-based idea in Section~\ref{sec:overview}. Recall Carath\'{e}odory's PSD constraints in Section~\ref{sec:cara_psd}. We apply it with $m=1$, which is to check whether $\abs{\Upsilon(t;\mu)}\leq 1$. Based on this modulus certificate, we derive the following fact as a direct corollary of Lemma~\ref{lem:true_candidate_known_var}.
\begin{Claim}\label{clm:true_candidate_order_1}
    When the variance $\sigma^2>0$ is given and $\epsmax\in(0,1/2)$, we have $\abs{\Upsilon(t;\theta)}\leq 1$ for all $t\in\bbR$.
\end{Claim}
With this modulus-based criterion, we apply the certification procedure introduced in Section~\ref{sec:certify_adversary} and complement it with the three technical adjustments presented in Section~\ref{sec:additional_issues}. Combining these ingredients leads to checking whether
\begin{align*}
    \abs{\Upsilon_n(t;\mu)}\leq 1+\Delta(t),\quad \forall\, t\in\cT,
\end{align*}
for each $\mu\in\widetilde{\CI}$. Here, the slack term $\Delta(t)$, the frequency set $\cT$, and the pilot confidence interval $\widetilde{\CI}$ are particularly chosen as in \eqref{eq:known_var_slack}, \eqref{eq:known_var_grid}, and \eqref{eq:known_var_pilot_CI}, respectively.

We now illustrate why, with the above specification, the output of this certification procedure can be transcribed into the closed form \eqref{eq:CI_known_var}. Denote by $\CIhat_t$ the set of candidates that pass the certification at $t\in\cT$:
\begin{align}\label{eq:CI_known_var_piece}
    \CIhat_t=\Big\{\mu\in \widetilde{\CI}:\ \abs{\Upsilon_n(t;\mu)}\leq 1+\Delta(t)\Big\}.
\end{align}
Recall that $\Upsilon_n(t;\mu)=2e^{-\im\mu t+\sigma^2t^2/2 }\phi_n(t)-1$. Rearranging the terms in the certification inequality, we discover that
\begin{align}\label{eq:upsilon_to_cosine}
    \abs{\Upsilon_n(t;\mu)}\leq 1+\Delta(t)\ \Longleftrightarrow\ \cos(t\mu-\Arg(\phi_n(t)))\geq \frac{1-(1+\Delta(t))^2+4\abs{e^{\sigma^2t^2/2}\phi_n(t)}^2}{4\abs{e^{\sigma^2t^2/2}\phi_n(t)}}.
\end{align}
Finally, according to our selection rule detailed in Algorithm~\ref{alg:known_var}, only those candidates that pass the certifications for all $t\in\cT$ are included in our confidence interval, which means,
\begin{align*}
    \CIhat=\bigcap_{t\in\cT}\CIhat_t.
\end{align*}
The closed form \eqref{eq:CI_known_var} then follows.
\begin{algorithm}[ht]
  \caption{Certification-based construction of $\CIhat$ (known-variance)}
  \label{alg:known_var}
  \begin{algorithmic}[1]
    \Require Max contamination level $\epsmax\in(0,1/2)$; Known variance $\sigma^2$; Data $X_1,\dots,X_n$
    \Ensure Confidence set $\CIhat$ for $\theta$
    \State Construct the pilot estimator $\widetilde{\theta}$ as in Lemma~\ref{lem:pilot_unknown_var}
    \State $\CIhat \gets \varnothing$
    \For{$\text{each candidate }\mu \in [\widetilde{\theta}\pm M\sigma/\sqrt{\log(en)}]$}\Comment{crude pilot interval}
    \State $\textsc{Cet}\gets \textsc{Pass}$
    \For{$\text{each frequency }t\in\cT$ }
        \State Compute $\Upsilon_n(t;\mu)=2e^{-\im \mu t+\sigma^2t^2/2}\phi_n(t)-1$
        \If{$\abs{\Upsilon_n(t;\mu)}>1+\Delta(t)$}\Comment{certify by modulus at each frequency}
            \State $\textsc{Cet}\gets \textsc{Fail}$
            \State \textbf{break}\Comment{end the loop over $t\in\cT$}
        \EndIf
    \EndFor
    \If{$\textsc{Cet}=\textsc{Pass}$}\Comment{all certificates are passed}
        \State $\CIhat \gets \CIhat \bigcup \{\mu\}$
    \EndIf
    \EndFor
    \State $\CIhat\gets [\inf\CIhat,\ \sup \CIhat]$\Comment{(optional) make $\CIhat$ a true interval}
  \end{algorithmic}
\end{algorithm}

We present the pseudo-code in Algorithm~\ref{alg:known_var}. Two remarks are in order:
\begin{Remark}[Whether the output $\CIhat$ is truly an interval]\label{rmk:whether_interval_known_var}
    Merely with the certification steps, we may not obtain a true interval. We can always take its convex hull, which is precisely the final step in Algorithm~\ref{alg:known_var}. Taking the convex hull neither decreases the coverage probability nor increases the length (referring to the diameter). 

    In fact, even without the convexification step, the output $\CIhat$ is a true interval as long as we set $\kappa$ as a sufficiently large constant (as allowed by Theorem~\ref{thm:known_var}).  See Appendix~\ref{sec:explain_remark} for detailed explanations.
\end{Remark}
\begin{Remark}[Scanning over $\widetilde{\CI}$ made efficient]\label{rmk:known_var_discretize}
    Even without the closed-form expression \eqref{eq:CI_known_var}, scanning $\mu$ over the pilot range can be implemented efficiently via a further discretization step. Since we know that the rate is slower than $\Theta(\sigma n^{-1/4})$ for whatever $\varepsilon\leq \epsmax$, we may partition the pilot interval $\widetilde{\CI}$ equally into $J=\Theta(n^{1/4}/\sqrt{\log(en)})$ pieces, resulting in $J+1$ pivots $\mu_0=\widehat{\theta}_\mathrm{med}-M\sigma/\sqrt{\log(en)},\ \mu_1,\ \mu_2,\cdots,\  \mu_J=\widehat{\theta}_\mathrm{med}+M\sigma/\sqrt{\log(en)}$. Let $j_L$ and $j_R$ be the minimum and maximum of the indices corresponding to the passed candidates:
    \begin{align*}
        j_L=\min \, (\textnormal{resp.}\,j_R=\max)\ \Big\{j\in \{0,1,2,\cdots,J\}:\ \abs{\Upsilon_n(t;\mu_j)}\leq 1+\Delta(t),\ \forall\, t\in\cT \Big\}.
    \end{align*}
    We can take $\CIhat'=[\mu_{j_L-1},\ \mu_{j_R+1}]$ as our final confidence interval, with the convention that $\mu_{-1}=\mu_0$ and $\mu_{J+1}= \mu_J$. Note that this $\CIhat'$ covers the original $\CIhat$. Moreover, the discretization step incurs a length increment of no more than $\Theta(\sigma n^{-1/4})$, and thus it will not change the order of the original length $\sigma\cdot O(n^{-1/4}+\varepsilon^{1/2}/\sqrt{1\vee\log(en\varepsilon^2)})$. With this, the entire algorithm is limited to at most $O(n^{1/4}\sqrt{\log(n)})$ certification steps and is thus computable within $\mathrm{poly}(n)$ time.
\end{Remark}

\subsection{Coverage and Length Guarantees from the Modulus Certificates}\label{sec:known_var_derivation}
The preceding analysis illustrates the certification-based principle of constructing a confidence interval, but is insufficient to explain why this specific confidence interval can achieve the coverage and length guarantees claimed in Theorem~\ref{thm:known_var}. In the remainder of this section, we will pin down the geometry induced by the modulus-based certificates to deliver a sketch view of why Theorem~\ref{thm:known_var} holds.

As a direct consequence of Lemma~\ref{lem:pilot_unknown_var}, the pilot coverage event
\begin{align}\label{eq:pilot_event_known_var}
    \widetilde{\cE}\define \Big\{\theta\in\widetilde{\CI}\Big\}=\Big\{\abs{\widetilde{\theta}-\theta}\leq \frac{M\sigma}{\sqrt{\log(en)}}\Big\}
\end{align}
holds with probability at least $1-\delta/2$. Meanwhile, it is straightforward from Hoeffding's inequality that $\abs{\phi_n(t)-\phi(t)}\lesssim 1 /\sqrt{n}$ with high probability for each $t\in\cT$, indicating that
\begin{align*}
    \abs{\Upsilon_n(t;\mu)-\Upsilon(t;\mu)}=2e^{\sigma^2t^2/2}\abs{\phi_n(t)-\phi(t)}\lesssim  e^{\sigma^2t^2/2}/\sqrt{n}
\end{align*} for each $t\in\cT$ separately.  Note that this concentration bound does not depend on the candidate $\mu\in\bbR$; this is because $\mu$ only appears in a rotation factor $e^{-\im \mu t}$ that does not affect the modulus. To obtain uniform concentration over $t$, we use the union bound for all $t\in\cT$, which will incur an additional cost in the convergence rate. Hopefully, as Lemma~\ref{lem:chf_concentration} ensures, only a multiplicative $\kappa \sigma t$ factor emerges, which can be absorbed into an additional $\kappa e^{\sigma^2 t^2/2}$ term. Therefore, we conclude that the uniform concentration event
\begin{align*}
    \cE\define \Big\{\abs{\phi_n(t)-\phi(t)}\lesssim \kappa e^{\sigma^2t^2/2}\cdot O_\delta (1/\sqrt{n}),\ \forall\, t\in\cT\Big\}
\end{align*}
holds with at least $1-\delta/2$ probability, on which we have
\begin{align*}
    \abs{\Upsilon_n(t;\mu)- \Upsilon(t;\mu)}\leq \Delta(t),\quad \forall\, t\in\cT,\ \mu\in\bbR.
\end{align*}
In the following derivation, we will always assume the joint event $\widetilde{\cE}\medcap \cE$ that holds with probability at least $1-\delta$.

\paragraph{Deriving the coverage.}We start by demonstrating that $\CIhat$ covers $\theta$ with high probability. From Claim~\ref{clm:true_candidate_order_1}, we know  $\abs{\Upsilon(t;\theta)}\leq 1$. Thus, by triangle inequality, $\abs{\Upsilon_n(t;\theta)}\leq\abs{\Upsilon(t;\theta)}+\Delta(t)\leq 1+\Delta(t)$ simultaneously for all $t\in\cT$ on $\widetilde{\cE}\medcap \cE$. As a consequence, $\theta$ passes the modulus-based certificates for all $t\in\cT$ on this high-probability joint event, which concludes the coverage.

\paragraph{Deriving the length.}The proof of the $\overline{r}\asymp\sigma(n^{-1/4}+\varepsilon^{1/2}/\sqrt{1\vee\log(en\varepsilon^2)})$ length is more involved and requires a careful inspection of the characteristic function's geometry on the complex plane $\bbC$. It suffices to show that simultaneously for any $\mu$ with $\abs{\mu-\theta}\geq \overline{r}$, there exists a witness $t\in\cT$ depending on $\mu$ such that $\abs{\Upsilon_n(t;\mu)}\leq 1+\Delta(t)$ fails. Since we have assumed the joint event $\widetilde{\cE}\medcap \cE$ that implies the pilot coverage event $\widetilde{\cE}$ \eqref{eq:pilot_event_known_var}, those $\mu$ with $\abs{\mu-\theta}>M/\sqrt{\log(en)}$ are already rejected by the pilot confidence interval $\widetilde{\CI
}$. Therefore, we only need to find witnesses for those $\mu$ with $\abs{\mu-\theta}\in[\overline{r},M\sigma/\sqrt{\log(en)}]$. Again, by the property of the uniform concentration event $\cE$, we have $\abs{\Upsilon_n(t;\mu)- \Upsilon(t;\mu)}\leq \Delta(t)$ for all $t\in\cT$ and $\mu\in\bbR$. By triangle inequality, a sufficient condition for the exclusion of $\mu$ is that $\abs{\Upsilon(t;\mu)}> 1+2\Delta(t)$ for some $t\in\cT$. Compare this with the proof of coverage that elicits $\abs{\Upsilon(t;\theta)}\leq 1$ from Claim~\ref{clm:true_candidate_order_1}. We now see that the width of $\CIhat$ is fully characterized by the rate at which $\abs{\Upsilon(t;\mu)}$ inflates above $1$ when $\mu$ moves away from the truth $\theta$. 

Let $r=\mu-\theta$ be the deviation of the candidate from the truth. Using the fact that the modulus of any characteristic function is no larger than $1$, we conclude that
\begin{align*}
    \Upsilon(t;\mu)=2(1-\varepsilon)e^{-\im rt}-1+\underbrace{2\varepsilon e^{-\im \mu t}\xi(t)}_{\abs{\cdot}\leq 2\varepsilon}\in \overline{\bbD}(-1+2(1-\varepsilon)e^{-\im rt},\, 2\varepsilon),
\end{align*}
where $\overline{\bbD}(w,r)$ denotes the closed disk $\{z\in\bbC:\, \abs{z-w}\leq r\}$ on the complex plane. The radius of the disk is fixed as $2\varepsilon$, while its center $-1+2(1-\varepsilon)e^{-\im rt}$ rotates around $-1+0\im$. When $r=0$ ($\mu=\theta$), the disk becomes $\overline{\bbD}(1-2\varepsilon,\, 2\varepsilon)$, which is completely contained within the unit disk $\overline{\bbD}(0,1)$ (modulus $\leq 1$); this is why high-probability coverage holds, as we have discussed previously. As the rotating angle $\abs{rt}$ increases from $0$ to $\pi$, the disk $\overline{\bbD}(-1+2(1-\varepsilon)e^{-\im rt},\, 2\varepsilon)$, where $\Upsilon(t;\mu)$ is located, will eventually be separated from the unit disk $\overline{\bbD}(0,1)$ with a positive margin for all large enough $\abs{rt}\in[0,\pi]$. In these cases, we will have $\abs{\Upsilon(t;\mu)}>1$. If, furthermore, the distance between $\overline{\bbD}(-1+2(1-\varepsilon)e^{-\im rt},\, 2\varepsilon)$ and $0$ overwhelms one plus the plug-in error bound $\Delta(t)\gtrsim\abs{\Upsilon_n(t;\mu)-\Upsilon(t;\mu)}$ for some $t\in\cT$, we claim that, with high probability, $\abs{\Upsilon_n(t;\mu)}>1+\Delta(t)$ and thus the certification fails. This can be seen from the triangular inequality that on the uniform concentration event $\cE$,
\begin{align*}
    \abs{\Upsilon_n(t;\mu)}\geq \abs{\Upsilon(t;\mu)}-\Delta(t)\geq \underbrace{\mathrm{dist}\Big(\overline{\bbD}(-1+2(1-\varepsilon)e^{-\im rt},\, 0\Big)}_{\textnormal{if this }>1+2\Delta(t) }-\Delta(t)>1+\Delta(t).
\end{align*}
To relate this condition to $\abs{r}$, we need the following analytic result.
\begin{Lemma}\label{lem:quadratic_gap}
    For $b\in[-\pi,\pi]$ and $\varepsilon\in[0,1/2)$, it holds that
    \begin{align*}
        \mathrm{dist}\Big(\overline{\bbD}(-1+2(1-\varepsilon)e^{-\im b},\, 0\Big)\geq 1+2(b/\pi)^2-4\varepsilon.
    \end{align*}
\end{Lemma}
Therefore, to witness $\Upsilon_n(t;\mu)>1+\Delta(t)$ on $\cE$ for a bad candidate $\mu=\theta+r$, it suffices to find $t\in\cT$ such that $\abs{rt}\leq \pi$ and
\begin{align}\label{eq:known_var_critical}
     (rt/\pi)^2- 2\varepsilon-2\Delta(t)>0.
\end{align}
Let us temporarily neglect the constraint that $t\in\cT$ and $\abs{rt}\leq \pi$, in order to develop a conceptual grasp of why the above condition leads to the $\sigma(n^{-1/4}+\varepsilon^{1/2}/\sqrt{1\vee\log(en\varepsilon^2)})$ rate. Rearranging the terms in \eqref{eq:known_var_critical}, we rewrite the witness condition as
\begin{align}\label{eq:known_var_critical_rearranged}
    r^2> O(1)\cdot\frac{\varepsilon}{t^2}+O(1)\cdot \frac{e^{O(1)\cdot \sigma t^2}}{t^2\sqrt{n}}\define G(t).
\end{align}
To make this condition hold for an $\abs{r}>0$ as small as possible, we minimize the right-hand side $G(t)$ with respect to $t>0$. Note that the first term $\varepsilon/t^2$ is polynomially decreasing in $t$, whereas the second term is exponentially increasing for $t\gtrsim 1/\sigma$. After we balance these two terms, the entire right-hand side is seen to be minimized at $t_*\asymp \sigma^{-1} \sqrt{1\vee \log(en\varepsilon^2)}$, with $r\asymp \sigma (n^{-1/4}+\varepsilon^{1/2}/\sqrt{1\vee\log(en\varepsilon^2)})$ being the square root of the minimized value $G(t_*)$. In general, this optimal $t_*$ that varies with the unknown $\varepsilon$ cannot be precisely realized by any finite frequency set. What we justify is that our choice $\cT$ is dense enough to approximate $t_*\asymp\sigma^{-1}\sqrt{1\vee \log(en\varepsilon^2)}$ for every unknown contamination strength $\varepsilon\in[0,\epsmax]$. In fact, for the right-hand side function $G(t)$, we have $G(t)\lesssim G(t_*)$ as long as $\abs{t^2-(t_*)^2}\lesssim 1/\sigma$. This motivates our choice of setting $\cT$ as the square root of a dense enough arithmetic sequence as prescribed in \eqref{eq:known_var_grid}. The condition that $\abs{rt}\leq \pi$ can also be satisfied by using $\abs{r}\leq M/\sqrt{\log(en)}$ and setting the constant $\kappa$ large enough. See Appendix~\ref{sec:proof_known_var_ub} for detailed proofs of correctness.

\section{Construction of Confidence Intervals When Variance Is Unknown}\label{sec:upper_bound_unknown_var}
We move on to the setting where not only $\varepsilon$ but also the variance $\sigma^2$ are unknown. In this section, we always assume $\epsmax\leq 1/3$ as opposed to $\epsmax<1/2$. We remark that the threshold $1/3$ is not arbitrary. As we shall see later in the proof of lower bounds, the optimal rate will undergo a sharp transition when $\epsmax$ exceeds $1/3$, as will be discussed in Section~\ref{sec:unknown_var_lb}.

\subsection{Pilot Confidence Intervals for Mean and Variance}
In the setting of unknown variance, the desired confidence interval is required to adapt to both $\varepsilon$ and $\sigma^2$. In light of this compound complexity, we first prepare the readers with some pilot inference strategies. These strategies may have very crude rates that do not adapt to the contamination level $\varepsilon$. However, they are adapted to variance $\sigma^2$ and thus can decompose the double-adaptation requirement, facilitating the design of our subsequent methodology. We begin by constructing a pilot confidence interval for $\sigma^2$ and a variance-adaptive pilot confidence interval for $\theta$ based on Lemma~\ref{lem:pilot_unknown_var}. Consider a sample splitting strategy in which we split $n$ iid samples into $\{\widetilde{X}_{1:\floor{n/2}},\, X_{1:\ceil{n/2}}\}$ and feed the first half $\{\widetilde{X}_{1:\floor{n/2}}\}$ into the pilot estimation process. According to Lemma~\ref{lem:pilot_unknown_var}, there are absolute $(\delta,\epsmax)$-dependent constants $(L,M,N_0)$ such that the pilot event
\begin{align}\label{eq:pilot_event_unknown_var}
    \widetilde{\cE}(\widetilde{X}_{1:\floor{n/2}})=\Big\{\abs{\widetilde{\theta}-\theta}\leq \frac{M\sigma}{\sqrt{\log(en)}},\ \frac{\abs{\widetilde{\sigma}^2-\sigma^2}}{\sigma^2}\leq \frac{L}{\log(en)}\Big\},
\end{align}
holds with probability at least $1-\delta/2$ (w.r.t. the randomness of $\widetilde{X}_{1:\floor{n/2}}$) when $n\geq 2N_0$. Here, $\widetilde{\theta}$ and $\widetilde{\sigma}^2$ are the pilot estimators introduced in Lemma~\ref{lem:pilot_unknown_var}. We then construct a pilot confidence interval for $\sigma^2$ by
\begin{align}
     \midb{\widetilde{\sigma}^2(1+\frac{L}{\log(en)})^{-1},\ \ \widetilde{\sigma}^2(1+\frac{L}{\log(en)})}\define[\widetilde{\sigma}^2_-,\widetilde{\sigma}^2_+],
\end{align}
and a pilot confidence interval for $\theta$ by
\begin{align}
    \widetilde{\CI}(\widetilde{X}_{1:\floor{n/2}})\define\Big[\widetilde{\theta}- \frac{M \widetilde{\sigma}_+}{\sqrt{\log(en)}},\ \widetilde{\theta}+\frac{M \widetilde{\sigma}_+}{\sqrt{\log(en)}}\Big].
\end{align}
With these setups, we claim by directly applying Lemma~\ref{lem:pilot_unknown_var} that:
\begin{Corollary}\label{cor:unknown_var_pilot}
Fix any $\delta\in(0,1)$ and $\epsmax\in(0,1/2)$, conditioned on $\widetilde{\cE}$ and assume $n\geq 2N_0$. We have
\begin{align*}
    \theta\in\widetilde{\CI},\quad \abs{\widetilde{\CI}}\leq \frac{M\sigma}{\sqrt{\log(en)}}\Big(1+\frac{L}{\log(en)}\Big),
\end{align*}
and also,
\begin{align*}
     \sigma^2\in[\widetilde{\sigma}^2_-,\widetilde{\sigma}^2_+],\quad \frac{\widetilde{\sigma}^2_+}{\sigma^2}\vee \frac{\sigma^2}{\widetilde{\sigma}^2_-} \leq \Big(1+\frac{L}{\log(en)}\Big)^2.
\end{align*}
\end{Corollary}
The reason for sample-splitting is that the frequency set becomes data-dependent when $\sigma^2$ is unknown. As a consequence, we need to construct the frequency set using the holdout samples to avoid confounding the uniform-concentration event of the empirical characteristic function evaluated on the training set. Mimicking the known-variance setting, we also construct a discrete frequency set 
\begin{align}\label{eq:unknown_var_grid}
    \widetilde{\cT}(\widetilde{X}_{1:\floor{n/2}})=\Big\{\sqrt{k}(\kappa \widetilde{\sigma}_+)^{-1}:\ k=1,2,\cdots,\ceil{\log(en)} \Big\},
\end{align}
where $\kappa$ is an absolute constant to be determined later that depends only on $\delta$ and $\epsmax$. We also define a discretized version of $[\widetilde{\sigma}_-^2,\widetilde{\sigma}_+^2]$ by
\begin{align}
    \widetilde{\cV}(\widetilde{X}_{1:\floor{n/2}})=\Big\{\widetilde{\sigma}_-^2+k\cdot \frac{(\widetilde{\sigma}_+^2-\widetilde{\sigma}_-^2)}{\ceil{4\kappa^{-2}\sqrt{n}\log(en)}}:\ k=0,1,2,\cdots,\ceil{4\kappa^{-2}\sqrt{n}\log(en)}\Big\},
\end{align}
where $A$ is an absolute constant that depends only on $\delta$, $\epsmax$, and $\kappa$.

\subsection{Order-2 Certificates from Carath\'{e}odory's Constraints}
Recall that, as discussed in Section~\ref{sec:certify_adversary}, in this unknown-variance setting, we switch to certify
\begin{align*}
    \Upsilon_n(t;\mu,\lambda)=3e^{-\im \mu t+\lambda t^2/2}\phi_n(t)-2,
\end{align*}
where the new parameter $\lambda$ functions as a guess for the true variance $\sigma^2$, and the different coefficients $3$ and $-2$ (previously $2$ and $-1$) reflect the modified assumption that $\epsmax\leq 1/3$. Here, $\phi_n(t)$ is the empirical characteristic function evaluated on the training set $X_{1:\ceil{n/2}}$. With a little abuse of notation, we still use $\phi_n(t)$ to denote it.

Following the principled routine in Section~\ref{sec:certify_adversary}, we scan over $(\mu,\lambda)\in \bbR\times (0,+\infty)$ and check whether $\Upsilon_n(\cdot\,;\mu,\lambda)$ is approximately a genuine characteristic function, and then collect all $\mu$ for which there exists a $\lambda$ such that $\Upsilon_n(\cdot\,;\mu,\lambda)$ passes the certification. In light of the pilot results in Corollary~\ref{cor:unknown_var_pilot}, the scanning range can be reduced to $\mu\in \widetilde{\CI}$ and $\lambda\in[\widetilde{\sigma}_-^2,\widetilde{\sigma}_+^2]$ only. The scanning range of $\lambda$ can be further reduced to the finite set $\widetilde{\cV}$, if we can show that $\widetilde{\cV}$ is a sufficiently fine-grained discretization of $[\widetilde{\sigma}_-^2,\widetilde{\sigma}_+^2]$ and will only incur a tenable approximation error.

As the final step, we specify the certificates to be applied to $\Upsilon_n(\cdot\,;\mu,\lambda)$. When variance is known, we use only a modulus check. When the unknown variance introduces an additional degree of freedom, we also add new certificates to capture the more delicate geometric properties. Specifically, given a candidate parameter $\mu$, we check whether both
\begin{align}\label{eq:unknown_var_certificates_empirical}
    \begin{dcases}
         \abs{\Upsilon_n(t;\mu,\lambda)}\leq 1+\Delta_1(t),\\
         \abs{\Upsilon_n^2(t;\mu,\lambda)-\Upsilon_n(2t;\mu,\lambda)}+\abs{\Upsilon_n(t;\mu,\lambda)}^2\leq 1+\Delta_2(t),
    \end{dcases}
\end{align}
hold for specific slack terms $\Delta_1(t)$ and $\Delta_2(t)$, for some $\lambda\in\widetilde{\cV}$ and all $t\in\widetilde{\cT}$. The rationale behind the certificates is the PSD constraints attributed to Carath\'{e}odory as per Lemma~\ref{lem:cara_psd}.
Take $m=2$ in Lemma~\ref{lem:cara_psd}, since $\Upsilon(t;\mu,\lambda)$ is already Hermitian, we obtain the following certificate:
\begin{align*}
\begin{bmatrix}
    1 & \Upsilon(t;\mu,\lambda) & \Upsilon(2t;\mu,\lambda)\\
    \overline{\Upsilon}(t,\mu,\lambda) & 1 & \Upsilon(t;\mu,\lambda)\\
    \overline{\Upsilon}(2t;\mu,\lambda) & \overline{\Upsilon}(t;\mu,\lambda) & 1
\end{bmatrix}\succeq 0.
\end{align*}
By Sylvester's criterion, a Hermitian matrix is PSD if and only if all of its principal minors are nonnegative. Therefore, the above certificate is equivalent to (omitting the trivial minor $1>0$):
\begin{align*}
   1-\abs{\Upsilon(t;\mu,\lambda)}^2\vee \abs{\Upsilon(2t;\mu,\lambda)}^2\geq 0,\quad (1-\abs{\Upsilon(t;\mu,\lambda)}^2)^2-\abs{\Upsilon^2(t;\mu,\lambda)-\Upsilon(2t;\mu,\lambda)}^2\geq 0.
\end{align*}
Note that the first part of the first inequality is equivalent to $\abs{\Upsilon(t;\mu,\lambda)}\leq 1$, and rearranging the second inequality yields that $\abs{\Upsilon(2t;\mu,\lambda)-\Upsilon(t;\mu,\lambda)^2}\leq 1-\abs{\Upsilon(t;\mu,\lambda)}^2$. 
By triangular inequality,
\begin{align*}
    \abs{\Upsilon(2t;\mu,\lambda)}\leq \abs{\Upsilon(t;\mu,\lambda)^2}+1-\abs{\Upsilon(t;\mu,\lambda)}^2=1,
\end{align*}
which covers the second part of the first inequality.
Therefore, it suffices to certify the following
\begin{align}\label{eq:unknown_var_certificates_population}
    \begin{dcases}
        \abs{\Upsilon(t;\mu,\lambda)}\leq 1,\\
        \abs{\Upsilon^2(t;\mu,\lambda)-\Upsilon(2t;\mu,\lambda)}+\abs{\Upsilon(t;\mu,\lambda)}^2\leq 1.
    \end{dcases}
\end{align}
With finite samples, we replace $\Upsilon(\cdot\,;\mu,\lambda)$ with $\Upsilon_n(\cdot\,;\mu,\lambda)$  and add slack terms $\Delta_1(t)$ and $\Delta_2(t)$ to the right-hand side of \eqref{eq:unknown_var_certificates_population}, which leads to the usable empirical version \eqref{eq:unknown_var_certificates_empirical}.

Having specified all the ingredients for certification, we can formulate the confidence interval by
\begin{align}\label{eq:CI_unknown_var}
    \CIhat=\bigcup_{\lambda\in\widetilde{\cV}}\,\bigcap_{t\in\widetilde{\cT}}\CIhat_{\lambda,t},
\end{align}
where 
\begin{align}
    \CIhat_{\lambda,t}=\Big\{\mu\in\widetilde{\CI}:\
    &\abs{\Upsilon_n(t;\mu,\lambda)}\leq 1+\Delta_1(t),\,\nonumber\\
    &\abs{\Upsilon_n^2(t;\mu,\lambda)-\Upsilon_n(2t;\mu,\lambda)}+\abs{\Upsilon_n(t;\mu,\lambda)}^2\leq 1+\Delta_2(t)  \Big\},
\end{align}
consists of all candidates $\mu$ that pass at $t$ with the company of $\lambda$, and
\begin{align}
    \Delta_1(t)=6\kappa e^{\widetilde{\sigma}_+^2t^2}\sqrt{\frac{16\log(10/\delta)}{n}},\quad \Delta_2(t)=81\kappa e^{5\widetilde{\sigma}_+^2t^2/2}\sqrt{\frac{16\log(10/\delta)}{n}}
\end{align}
denote carefully chosen slack terms. 
See Algorithm~\ref{alg:unknown_var} in Appendix~\ref{sec:alg_description} for an algorithmic description.

We present below the main result of this section, which guarantees the performance of the confidence interval $\CIhat$ defined above by \eqref{eq:CI_unknown_var}.
\begin{Theorem}\label{thm:unknown_var}
    Assume $\sigma^2>0$ is unknown and fix any $\delta\in(0,1)$ and $\epsmax\in(0,1/3]$. Let $M$ be the $(\delta,\epsmax)$-dependent constant as in Lemma~\ref{lem:pilot_unknown_var}. For the $\CIhat$ defined by \eqref{eq:CI_unknown_var} with $\kappa\geq \sqrt{5}\vee (2M/\pi)$, there exist absolute constants $C$ and $N$ that depend only on $\delta$, $\epsmax$, and $\kappa$, such that
    \begin{align*}
\inf_{\theta\in\bbR,\, \sigma^2>0}\ \inf_{\varepsilon\in[0,\epsmax],\, Q}\ P^n_{\theta,\sigma,\varepsilon,Q}\Big(\theta\in\CIhat,\, \abs{\CIhat}\leq \overline{r}\Big)\geq 1-\delta
    \end{align*}
    whenever $n\geq N$, where
    \begin{align*}
        \overline{r}(n,\sigma,\varepsilon)=C\sigma\left(\frac{1}{n^{1/8}}+\frac{\varepsilon^{1/4}}{\sqrt{1\vee\log(en\varepsilon^2)}}\right).
    \end{align*}
\end{Theorem}
The proof of the above theorem is deferred to Appendix~\ref{sec:proof_unknown_var_ub}.

\subsection{Coverage and Length Guarantees from Order-2 Certificates}
We outline the proof of Theorem~\ref{thm:unknown_var} by first briefing the coverage property and then the length bound. We will always condition on the intersection between the good pilot event \eqref{eq:pilot_event_unknown_var} and the uniform concentration event of the empirical characteristic function.

\paragraph{Deriving the coverage.} We shall mimic the previous derivation in Section~\ref{sec:known_var_derivation}, using the fact that $\Upsilon(t;\mu=\theta,\lambda=\sigma^2)$ is a genuine characteristic function to show that $\mu=\theta$ can pass all certificates. It is possible that the true $\sigma^2$ may not be attained by any $\lambda\in\widetilde{\cV}$ since the variance candidate set $\widetilde{\cV}$ has been discretized. Nevertheless, the discretization is dense enough such that there exists a $\lambda_*\in\widetilde{\cV}$ very close to $\sigma^2$. As a result, this deterministic approximation error, together with the random plug-in error, can be absorbed into our slack terms $\Delta_1(t)$ and $\Delta_2(t)$, and therefore the mean-variance pair $(\theta,\lambda_*)$ passes all certificates.

\paragraph{Deriving the length.} The derivation of high-probability length also borrows intuitions from an approximation viewpoint. Given a bad location candidate $\mu$ with $r=\mu-\theta$ large in the absolute sense, we need to verify that there is no $\lambda\in \widetilde{\cV}$ such that the pair $(\mu,\lambda)$ can pass the certificates for all $t\in\widetilde{\cT}$. Recall that
\begin{align*}
    \phi(t)=[(1-\varepsilon)e^{\im \theta t}+\varepsilon \xi(t)]\cdot e^{\sigma^2t^2/2},
\end{align*}
where $\xi(t)=\E_Q[\exp(\im tX)]$. Since $\abs{\xi(t)}\leq 1$, we have
\begin{align*}
    \abs{\phi(t)-\phi_0(t)}\lesssim O(\varepsilon)e^{-\sigma^2t^2/2},
\end{align*}
where $\phi_0(t)=e^{-\im \theta t +\sigma^2t^2/2}$ is the characteristic function of the clean fraction $\cN(\theta,\sigma^2)$. As this implies, $\phi$ is $O(\varepsilon e^{-\sigma^2t^2/2})$-close to $\phi_0$, and thus their certificate values should also be very close to each other. In fact, through rigorous control of the approximation error, we can show that the population certificate values of $\phi$ are also $O(\varepsilon)$-close to those of $\phi_0$. Furthermore, by change-of-variables in the location family, the certificate values of $\phi_0$ against the candidate $\mu$ are equal to the certificate values of $e^{-\sigma^2t^2/2}=\E_{\cN(0,\sigma^2)}[\exp(\im tX)]$ against the candidate $r=\mu-\theta$. Therefore, to quantify the population-level gap between the certificate values and $1$, it remains to determine:

\emph{For $r\neq 0$, how well can order-2 certificates detect that the distribution $\cN(0,\sigma^2)$ is not an Efron's model with location $r$? }

In contrast to the quadratic gap in Lemma~\ref{lem:quadratic_gap}, we have a quartic gap in the unknown-variance setting.

\begin{Lemma}\label{lem:quartic_gap}
Let $\Upsilon(t;r,\lambda)$ be as in \eqref{eq:unknown_var_upsilon} with specifically $\phi(t)=e^{-\sigma^2t^2/2}$. For any $\sigma^2>0$, $\lambda>0$, and $rt\in[-\pi,\pi]$, it holds that
\begin{align*}
    \max\left\{\abs{\Upsilon(t;r,\lambda)},\ \abs{\Upsilon^2(t;r,\lambda)-\Upsilon(2t;r,\lambda)}+\abs{\Upsilon(t;r,\lambda)}^2\right\}\geq 1+(rt/\pi)^4.
\end{align*}
\end{Lemma}
Note how this bound is entirely independent of the true variance $\sigma^2$ and its guessed proxy $\lambda$. With this bound established, the failure of certification can be reduced to
\begin{align}\label{eq:unknown_var_critical}
    \underbrace{(rt/\pi)^4}_{\textnormal{geometric gap}} \quad - \underbrace{O(1)\cdot \varepsilon}_{\textnormal{approximation error}} - \qquad \underbrace{O(1)\cdot e^{O(1)\cdot \sigma^2 t^2}/\sqrt{n}}_{\textnormal{plug-in error}}>0
\end{align}
for some $t\in\widetilde{\cT}$ such that $\abs{rt}\leq \pi$. Rearranging the terms, we obtain:
\begin{align}
    r^4> O(1)\cdot \frac{\varepsilon}{t^4}+O(1)\cdot \frac{e^{O(1)\cdot \sigma^2 t^2}}{t^4\sqrt{n}}\define H(t).
\end{align}
Minimizing the right-hand side with respect to $t\geq 0$ yields exactly the desired rate $r\gtrsim \sigma(n^{-1/8}+\varepsilon^{1/4}/\sqrt{1\vee \log(en\varepsilon^2)})$, with the minimum attained at $t_*\asymp \sigma^{-1}\sqrt{1\vee\log(en\varepsilon^2)}$. Similar to the known-variance setting, we have $H(t)\lesssim H(t_*)$ as long as $\abs{t^2-t_*^2}\lesssim 1/\sigma^2$, which motivates our choice of the discrete frequency set $\widetilde{\cT}$. See Appendix~\ref{sec:proof_unknown_var_ub} for detailed proofs of correctness.

\section{Minimax Lower Bounds}\label{sec:lower_bound}

In this section, we develop information-theoretic limits to reveal the tightness of the upper bounds shown in previous sections. Following \cite{luo2024adaptive}, we reduce the inference task to a composite hypothesis testing problem. Recall the notation $P_{\theta,\sigma,\varepsilon,Q}=(1-\varepsilon)\cN(\theta,\sigma^2)+\varepsilon Q*\cN(0,\sigma^2)$. We start with the known-variance setting in which $\epsmax\in(0,1/2)$ is a fixed constant and $\sigma^2>0$ is given. For $r>0$ and $\varepsilon\in[0,\epsmax]$, we consider testing between a pair of composite hypotheses:
\begin{align}\label{eq:known_var_test}
    H_0(\sigma,\epsmax):\ X_{1:n}\iid P\in\{P_{0,\sigma,\epsmax,Q}:\, Q\}\quad\textnormal{v.s.}\quad H_1(r,\sigma,\varepsilon):\ X_{1:n}\iid P\in\{P_{r,\sigma,\varepsilon,Q}:\, Q\}.
\end{align}
Note that the differences between $H_0(\sigma,\epsmax)$ and $H_1(r,\sigma,\varepsilon)$ differ not only in location parameters but also in their contamination strength. For any $\CIhat$ that adaptively satisfies high-probability coverage and a length rate $r/2=r(n,\sigma,\varepsilon)/2$, it naturally induces a test for $H_0(\sigma,\epsmax)$ versus $H_1(r,\sigma,\varepsilon)$ that has uniformly small type-I and type-II errors over all $\varepsilon\in[0,\epsmax]$. Indeed, define the test
\begin{align}\label{eq:test_by_CI}
    T(X_{1:n})=\indi\{0\notin \CIhat\}.
\end{align}
By modifying the proof of Proposition 2 in \cite{luo2024adaptive}, it can be shown that 
\begin{align*}
    \sup_{P\in H_0(\sigma,\epsmax)} P^n(T)+\sup_{P\in H_1(r,\sigma,\varepsilon)}P^n(1-T)\leq 3\delta,
\end{align*}
as long as $\CIhat$ (knowing $\sigma^2$) satisfies both the adaptive robust coverage and $r/2$-length guarantees. We will detail this fact in Proposition~\ref{prop:CI_to_test_known_var}. This $\CIhat$-to-test reduction is used in our proof of the lower bound via contradiction. To show a rate $\underline{r}=\underline{r}(n,\sigma,\varepsilon)$ is a lower bound for the length of all adaptive confidence intervals, it suffices to verify the indistinguishability between $H_0(\sigma,\epsmax)$ and $H_1(\underline{r},\sigma,\varepsilon)$ for all $\varepsilon\in[0,\epsmax]$.

Slightly generalizing the known-variance setting, we consider in the unknown-variance setting a modified pair of hypotheses:
\begin{align}\label{eq:unknown_var_test}
    H_0(\epsmax):\ X_{1:n}\iid P\in\{P_{0,\sigma,\epsmax,Q}:\, \sigma,Q\}\quad\textnormal{v.s.}\quad H_1(r,\sigma,\varepsilon):\ X_{1:n}\iid P\in\{P_{r,\sigma,\varepsilon,Q}:\, Q\}
\end{align}
for every $\sigma^2>0$ and $\varepsilon\in[0,\epsmax]$ and for some function $r=r(n,\sigma,\varepsilon)>0$. Compared to \eqref{eq:known_var_test}, the only change is on the null hypothesis $H_0$ whose variance is now arbitrary and no longer needs to be the same as $H_1$, to capture the hardness of adapting to an unknown variance. Again, consider the test $T(X_{1:n})$ defined by \eqref{eq:test_by_CI}. We can show that the sum of type-I and type-II errors of $T$ in \eqref{eq:unknown_var_test} is within $3\delta$ as long as  $\CIhat$ satisfies both adaptive robust coverage and $r/2$-length without knowing the variance. Therefore, the lower bound $\underline{r}$ can again be derived from the indistinguishability between $H_0(\sigma,\epsmax)$ and $H_1(\underline{r},\sigma,\varepsilon)$.

With the above reduction established, it suffices to find a pair of Gaussian location mixtures from $H_0(\sigma,\epsmax)$ (resp. $H_0(\epsmax)$) and $H_1(r,\sigma,\varepsilon)$ when the variance $\sigma^2$ is known (resp. unknown) that are hard to distinguish. To achieve this, we first construct a pair of priors that match several leading moments. Then, the chi-square divergence between their Gaussian posteriors can be controlled by the differences between unmatched prior moments using Hermite polynomials, as justified by the following lemma from \cite{wu2020polynomial}.
\begin{Lemma}[\cite{wu2020polynomial} Theorem 3.3.3.]\label{lem:chi2_subgaussian} Let $\nu_0$ and $\nu_1$ be two probability distributions on $\bbR$. Suppose that they match the first $K$ moments, and both of them are sub-Gaussian with variance proxy $\gamma^2<1$. Then,
\begin{align*}
    \chi^2\big(\nu_0*\cN(0,1),\, \nu_1*\cN(0,1)\big)\leq \frac{16 \gamma^{2K+2}}{\sqrt{K}(1-\gamma^2)}.
\end{align*}
\end{Lemma}
By slightly modifying their proof, we obtain a more handy version in the current context.
\begin{Lemma}[Modified from \cite{wu2020polynomial} Theorem 3.3.3.]\label{lem:chi2_by_moments} Let $\nu_0$ and $\nu_1$ be two probability distributions on $\bbR$. Suppose that $\nu_0(\{0\})\geq 1/2$, then,
\begin{align*}
    \chi^2\big(\nu_0*\cN(0,1),\, \nu_1*\cN(0,1)\big)\leq 2\sum_{k=1}^{+\infty}\frac{\big(\E_{\nu_0}[X^k]-\E_{\nu_1}[X^k]\big)^2}{k!}.
\end{align*}
\end{Lemma}

\subsection{Known-Variance Lower Bounds}\label{sec:known_var_lb}
When the variance $\sigma^2$ is known, we derive the following lower bound that matches the upper bound in Theorem~\ref{thm:known_var}.
\begin{Theorem}\label{thm:known_var_lower}
Fix any $\delta\in(0,1/4)$, $\epsmax\in(0,1/2)$, and assume that $\sigma^2>0$ is known. For any confidence interval $\CIhat$ that satisfies
\begin{align*}
    \inf_{\theta}\, \inf_{\varepsilon\in[0,\epsmax],\, Q}\, P_{\theta,\sigma,\varepsilon,Q}^n(\theta\in\CIhat)\geq 1-\delta,\quad \inf_{\theta}\, \inf_{\varepsilon\in[0,\epsmax],\, Q}\, P_{\theta,\sigma,\varepsilon,Q}^n(\abs{\CIhat}\leq r(n,\sigma,\varepsilon))\geq 1-\delta,
\end{align*}
it must hold that
\begin{align*}
    r(n,\sigma,\varepsilon)\geq c\sigma \Big(\frac{1}{n^{1/4}}+\frac{\varepsilon^{1/2}}{1\vee \sqrt{\log(en\varepsilon^2)}}\Big),
\end{align*}
for all $(n,\sigma,\varepsilon)$ and an absolute constant $c$ that depends only on $\delta$ and $\epsmax$.
\end{Theorem}
\begin{proof}
By scale-equivariance, it suffices to prove the conclusion for $\sigma^2=1$. When $\varepsilon\leq n^{-1/2}$, we shall prove the first part of the rate, $n^{-1/4}$, which can be seen from the following pair of priors:
\begin{align*}
    \nu_0=(1-\epsmax)\delta_0+\epsmax\delta_{c_\delta n^{-1/4}}, \quad  \nu_1=\delta_{c_\delta\epsmax n^{-1/4}},
\end{align*}
for some $c_\delta\in(0,1)$. Note that they match in the first moment. Define $P_i=\nu_i*\cN(0,1)$. We notice that $P_0\in H_0(1,\epsmax)$ and $P_1\in H_1(\Theta_{\delta,\epsmax}(n^{-1/4}),1,\epsmax)$. By Lemma~\ref{lem:chi2_by_moments}, we have
\begin{align*}
    \chi^2(P_1,P_0)\leq 2\sum_{k\geq 2}\frac{c_\delta^{2k}(\epsmax n^{-k/4}-\epsmax^k n^{-k/4})^2}{k!} \leq 2c_\delta^4/n\cdot \sum_{k\geq 2}1/k!\leq 2ec_\delta^4 /n.
\end{align*}
Thus,
\begin{align*}
    \inf_T\ P_0^n(T)+P_1^n(1-T)=1-\TV(P_0^n, P_1^n)\geq 1-\frac{1}{2}\sqrt{(1+\chi^2(P_1,P_0))^n-1}\geq 4\delta,
\end{align*}
as long as $c_\delta$ is small enough, which completes the proof of the first rate $c_{\delta,\epsmax}n^{-1/4}$. To show the second part of the rate, we need to find new priors $\nu_0$ and $\nu_1$ that match a higher number of moments.
\begin{Proposition}\label{prop:known_var_matching}
For any $\epsmax\in(0,1/2)$, $\varepsilon\in(0,\epsmax]$, $K$ a positive multiple of $4$, and $\tau>0$, there exist two finitely supported probability distributions $\nu_0$, $\nu_1$, and an absolute constant $c$ that depends only on $\epsmax$, such that
\begin{enumerate}
    \item \textbf{point concentration:}
    \begin{align}
        \nu_0(\{0\})\geq 1-\epsmax,\quad \nu_1(\{c\varepsilon^{1/2}/\tau\})\geq 1-\varepsilon;
    \end{align}
    \item \textbf{moment matching:}
    \begin{align}
        \E_{\nu_0}[X^k]=\E_{\nu_1}[X^k],\quad k=0,1,2,\cdots, K+1;
    \end{align}
    \item \textbf{magnitude of moments:} if, in addition, $K/\tau\geq 1$, then
    \begin{align}
    \abs*{\E_{\nu_0}[X^k]}\vee \abs*{\E_{\nu_1}[X^k]}\leq 2\varepsilon (K/\tau)^k,\quad k\geq 2.
    \end{align}  
\end{enumerate}
\end{Proposition}
The proof of this result is deferred to Appendix~\ref{sec:proof_known_var_matching}. With the above proposition established, we continue to derive a lower bound for the $\varepsilon\geq n^{-1/2}$ case, which is the second part of the rate in Theorem~\ref{thm:known_var_lower}. Let $\nu_0$ and $\nu_1$ be the two priors given by Proposition~\ref{prop:known_var_matching}. We set 
\begin{align*}
    K=16\ceil{\log(en\varepsilon^2/c_\delta)},\quad \tau=eK^{1/2},
\end{align*}
where $c_\delta\in(0,1)$ is a $\delta$-dependent constant to be specified later. Check that 
\begin{align*}
    K/\tau=(4/e)\sqrt{\ceil{\log(en\varepsilon^2/c_\delta)}}\geq 1
\end{align*}
when $\varepsilon\geq n^{-1/2}$, as long as $c_\delta\in(0,1)$. 
Denote $P_i=\nu_i*\cN(0,1)$ for $i=0,1$. Then, by Lemma~\ref{lem:chi2_by_moments},
\begin{align*}
    \chi^2(P_1,P_0)&\leq 2\sum_{k\geq K+2}\frac{(4\varepsilon (K/\tau)^k)^2}{k!}\\
    &=32 \varepsilon^2 e^{K^2/\tau^2} \Prob(\mathrm{Poi}(K^2/\tau^2)\geq K+2)\\
    &\leq 32 \varepsilon^2e^{K^2/\tau^2}\cdot e^{K^2/\tau^2}\Big(\frac{eK^2/\tau^2}{K+2}\Big)^{K+2}\\
    &=32 \varepsilon^2 \Big(\frac{eK^2/\tau^2}{K+2}\Big)^{K+2},
\end{align*}
where in the second-to-last line we used the Chernoff bound for Poisson (Lemma~\ref{lem:poisson_tail}) and the fact that $ K^2/\tau^2=K/e^2< K+2$. Furthermore, we notice that $ eK^2/(\tau^2(K+2))=K/(e(K+2))\leq 1/e$. Therefore, we can continue to write:
\begin{align*}
     \chi^2(P_1,P_0)&\leq 32\varepsilon^2 e^{-(K+2)}\leq 32e^{-2} \varepsilon^2 e^{-8\log(en\varepsilon^2/c_\delta)}\leq 32e^{-2}\varepsilon^2\cdot \frac{c_\delta}{en\varepsilon^2}\leq 2c_\delta/n.
\end{align*}
Again, by setting $c_\delta$ small enough, we have
\begin{align*}
    \inf_T\ P_0^n(T)+P_1^n(1-T)=1-\TV(P_0^n, P_1^n)\geq 1-\frac{1}{2}\sqrt{(1+\chi^2(P_1,P_0))^n-1}\geq 4\delta,
\end{align*}
which proves a lower bound of $c_{\delta,\epsmax}\cdot\varepsilon^{1/2}/(1\vee \log^{1/2}(en\varepsilon^2))$ and completes the whole proof.
\end{proof}

The proof of Theorem~\ref{thm:known_var_lower} relies on the construction of $\nu_0$ and $\nu_1$ as in Proposition~\ref{prop:known_var_matching}, which we briefly describe below. Recall that for any distribution $\nu$ whose $k$-th moment exists, its $k$-th moment can be obtained from its characteristic function through the relation
\begin{align*}
    \E_\nu[X^k]=\im^{-k}\left.\frac{\dif \E_\nu[e^{\im tX}]}{\dif t}\right|_{t=0}.
\end{align*}
Therefore, two distributions match their $0$-th to $(K+1)$-th moments (if they exist in the absolute sense) if and only if the difference between their characteristic function is $o(\abs{t}^{K+1})$ as $t\to 0$. This motivates us to prove the existence of $\nu_0$ and $\nu_1$ by constructing the (unnormalized) difference between their characteristic functions: $M(t)\propto \E_{\nu_0-\nu_1}[e^{\im tX}]$. The desired difference $M(t)$ should satisfy
\begin{enumerate}
    \item $M_0(0)=0$, since every characteristic function takes value $1$ at $0$, and $M(t)$ is proportional to the difference of two characteristic functions;
    \item $M(t)=o(\abs{t}^{K+1})$, which is the moment-matching condition that we have just discussed.
\end{enumerate}
We start the trial with $\nu_0$ and $\nu_1$ being finitely supported distributions, which means that $M(t)=\sum_{j\in[J]} \alpha_j e^{\im b_j t}$ for a finite collection of coefficients $(\alpha_j,b_j)$. A good starting point is the binomial form:
\begin{align*}
    M(t)= \Big(e^{\im t/(2\tau)}-e^{-\im t/(2\tau)}\Big)^{\Theta(K)}=\sum_{k=0}^{\Theta(K)}(-1)^{\Theta(K)-k}\binom{\Theta(K)}{k}e^{\im (k-\Theta(K)/2) t/\tau },
\end{align*}
which automatically satisfy the previous constraints. Nevertheless, we also need to check whether the induced $\nu_0$ and $\nu_1$ both highly concentrate at one point, in order to guarantee that $\nu_0*\cN(0,\sigma^2)\in H_0(\sigma,\epsmax)$ and $\nu_1*\cN(0,\sigma^2)\in H_1(r,\sigma,\varepsilon)$ after location-shifting and normalization. Unfortunately, although we know that the binomial coefficient $\binom{K}{k}$ peaks at $k=K/2$ for even $K$, it is not peaky enough to include $1-\epsmax$ (resp. $1-\varepsilon$) of the total mass to form a legal $\nu_0$ (resp. $\nu_1$). To address this issue, we perform a \emph{tilt-then-integrate} trick twice. To be specific, we define
\begin{align*}
    M_0(t)=\Big(e^{\im t/(2\tau)}-e^{-\im t/(2\tau)}\Big)^{K}=\sum_{k=0}^{K}(-1)^{K-k}\binom{K}{k}e^{\im (k-K/2) t/\tau }
\end{align*}
and then multiply it by an exponential tilting factor $e^{\im at/\tau}$ for some $a\in(0,1/2)$ to become
\begin{align*}
    M_1(t)=e^{\im at/\tau}M_0(t)=\sum_{k=0}^{K}(-1)^{K-k}\binom{K}{k}e^{\im (k-K/2+a) t/\tau }.
\end{align*}
The magic emerges when we conduct integration, i.e., defining
\begin{align*}
    M_2(t)=\int_0^t M_1(s)\dif s=\sum_{k=0}^{K}(-1)^{K-k}\binom{K}{k}\frac{\tau }{k-K/2+a}e^{\im (k-K/2+a) t/\tau }.
\end{align*}
Note that for $k\neq K/2$, the absolute denominator $\abs{k-K/2+a}$ is larger than $1/2$ since we have assumed $a\in(0,1/2)$; but for $k=K/2$, the denominator becomes simply $a$, and can thus be made arbitrarily small by tuning $a\to 0^+$ to be small enough. Thus, the central coefficient ($k=K/2$) can be made significantly larger than the sum of the absolute values of its noncentral counterparts, well mimicking the hypotheses $H_0(\sigma,\epsmax)$ and $H_1(r,\sigma,\varepsilon)$ that we desire. The full construction is given by applying this tilt-then-integrate trick twice, that is, to analyze:
\begin{align*}
    M_4(t)=\int_0^t e^{\im bs/\tau}\Big(\int_0^{s}e^{\im as'/\tau}M_0(s')\dif s'\Big)\dif s,
\end{align*}
where we set $a\asymp b\asymp \varepsilon^{1/2}$. This choice of $a$ and $b$ will scale the central coefficient by $1/ab\asymp 1/\varepsilon$ so that it can dominate at least $1-O(\varepsilon)$ proportion of the total mass. Readers are directed to Appendix~\ref{sec:proof_known_var_matching} for a detailed proof of the correctness.

\subsection{Unknown-Variance Lower Bounds}\label{sec:unknown_var_lb}
In the unknown-variance setting with $\epsmax\in(0,1/3]$, we have seen a further degradation of the rate in the upper bounds, especially from $n^{-1/4}$ to $n^{-1/8}$ when $\eps=O(n^{-1/2})$. The following theorem justifies that such a worse exponent is unavoidable.
\begin{Theorem}\label{thm:unknown_var_lower}
Fix any $\delta\in(0,1/4)$, $\epsmax\in(0,1/3]$, and assume that $\sigma^2>0$ is unknown. For any confidence interval $\CIhat$ that satisfies
\begin{align*}
    \inf_{\theta,\, \sigma^2}\, \inf_{\varepsilon\in[0,\epsmax],\, Q}\, P_{\theta,\sigma,\varepsilon,Q}^n(\theta\in\CIhat)\geq 1-\delta,\quad \inf_{\theta,\, \sigma^2}\, \inf_{\varepsilon\in[0,\epsmax],\, Q}\, P_{\theta,\sigma,\varepsilon,Q}^n(\abs{\CIhat}\leq r(n,\sigma,\varepsilon))\geq 1-\delta,
\end{align*}
it must hold that
\begin{align*}
    r(n,\sigma,\varepsilon)\geq c\sigma n^{-1/8},
\end{align*}
for all $(n,\sigma,\varepsilon)$ and an absolute constant $c$ that depends only on $\delta$ and $\epsmax$.
\end{Theorem}
Although this lower bound does not fully match the scaling of the upper bound in Theorem~\ref{thm:unknown_var} for all $\varepsilon\in[0,\epsmax]$, it does capture the severe degradation of the length from $n^{-1/4}$ to $n^{-1/8}$ even when the samples are truly clean ($\varepsilon=0$). According to the $\CIhat$-to-test reduction introduced at the beginning of Section~\ref{sec:lower_bound}, it suffices to show the indistinguishability between $H_0(\epsmax)$ and $H_1(r,\sigma,0)$ for $r\asymp \sigma n^{-1/8}$, i.e., finding $Q$ and $v^2>0$ such that
\begin{align*}
    [(1-\epsmax)\delta_0+\epsmax Q]*\cN(0,v^2)\qquad \textnormal{v.s.}\qquad \delta_r *\cN(0,\sigma^2)
\end{align*}
cannot be distinguished with a high enough probability using $n$ samples. Scaling both sides by $1/v$ (which does not affect the hardness of testing), the above hypotheses become:
\begin{align*}
    [(1-\epsmax)\delta_0+\epsmax Q']*\cN(0,1)\qquad \textnormal{v.s.}\qquad \delta_{r/v} *\cN(0,\sigma^2/v^2)=\cN(r/v,\sigma^2/v^2-1)*\cN(0,1),
\end{align*}
for some new $Q'$ as long as $\sigma^2\geq v^2$, which is in a standardized format ready for Lemma~\ref{lem:chi2_subgaussian}. The choice $\sigma^2\geq v^2$ is natural through the lens of moment-matching since the variance on the left-hand side is never lower than $1$. Now, to apply Lemma~\ref{lem:chi2_subgaussian}, we need to determine the maximum number of moments that the location priors:
\begin{align*}
    \nu_0=(1-\epsmax)\delta_0+\epsmax Q'\qquad \textnormal{and}\qquad  \nu_1=\cN(r/v,\sigma^2/v^2-1)
\end{align*}
can match. Again, since scaling both sides will not affect the number of matched moments, we divide both $\nu_0$ and $\nu_1$ simultaneously by $r/v$, and further reduce this problem to the moment-matching between
\begin{align*}
    \mu_0=(1-\epsmax)\delta_0+\epsmax Q''\qquad \textnormal{and}\qquad \mu_1=\cN(1,s^2),
\end{align*}
where $s^2=(\sigma^2-v^2)/r^2$. How many moments can $\mu_0$ and $\mu_1$ match? A short answer when $\epsmax\in(0,1/3]$ is: we can find a finitely supported $Q''$ and an $s^2>0$ that both depend only on $\epsmax$ such that $\mu_0$ and $\mu_1$ (and thus $\nu_0$ and $\nu_1$) match their leading three moments; on the other hand, matching the fourth moment is impossible. Therefore, by Lemma~\ref{lem:chi2_subgaussian}, the chi-square divergence between $\nu_0*\cN(0,1)$ and $\nu_1*\cN(0,1)$ is bounded by $(r/v)^{2\times 3+2}=(r/\sqrt{\sigma^2-r^2s^2})^8=(r/\sqrt{\sigma^2-O_{\epsmax}(r^2)})^8$, which can be made $O_\delta(1/n)$ once we set $r\asymp_{\delta,\epsmax}(\sigma n^{-1/8})$.

The remaining task is to explain why the $\mu_0$ and $\mu_1$ defined above can match at most the first three but not the fourth moment. At the end of this section, we will prove Theorem~\ref{thm:unknown_var_lower} by constructing a concrete pair of $\mu_0$ and $\mu_1$ that match the first three moments. Nevertheless, there is another \emph{a priori} way to determine the maximum number of matched moments without instantiating $\mu_0$ and $\mu_1$ in arduous detail. We first prepare the readers with a standard conclusion on the truncated Hamburger moment problem.
\begin{Lemma}[Theorem 3.3 and 3.4 of \cite{lasserre2009moments}]\label{lem:truncated_hamburger}
Let $\{m_0=1,\, m_1,\, m_2,\cdots,\, m_k\}$ be a finite sequence in $\bbR$. Define the associated Hankel matrix by $\bfH_\ell=[m_{i+j}]_{i,j=0}^\ell\in\bbR^{(\ell+1)\times (\ell+1)}$ for $0\leq \ell\leq \floor{k/2}$. Then, there exists a probability measure $\mu$ on $\bbR$ such that $\E_\mu[X^j]=m_j$ for $j=0,1,2,\cdots,k$, if and only if
\begin{itemize}
    \item when $k$ is odd, $\bfH_{(k-1)/2}\succeq 0$ and $(m_{(k+1)/2}, \cdots, m_k)\in\mathrm{range}(\bfH_{(k-1)/2})$;
    \item when $k$ is even, there exist $m_{k+1}$ and $m_{k+2}$ such that $\bfH_{k/2+1}\succeq 0$ and $\mathrm{rank}(\bfH_{k/2+1})=\mathrm{rank}(\bfH_{k/2})$.
\end{itemize}
\end{Lemma}
Once $\mu_0$ and $\mu_1$ match their first $k$ moments, the moments of $Q''$ are automatically restricted to 
\begin{align}\label{eq:moments_q''}
    \E_{Q''}[X^j]=\epsmax^{-1}\Big(\E_{\cN(1,s^2)}[X^j]-(1-\epsmax)\E_{\delta_0}[X^j]\Big),\quad j=0,1,2,\cdots,k,
\end{align}
which can always be written as polynomials in $s^2$ and $\epsmax$. Consider the corresponding Hankel matrices $\{\bfH_j\}$. We notice that
\begin{align*}
    \det \bfH_0=1,\quad \det \bfH_1=\epsmax^{-2}(\epsmax s^2+\epsmax -1).
\end{align*}
By setting $s^2>(1-\epsmax)/\epsmax$, we can make both $\det H_0>0$ and $\det H_1>0$. This is equivalent to $ \bfH_2\succ 0$ (and thus $\mathrm{range}(\bfH_2)=\bbR^{3}$) by Sylvester's criterion; therefore, according to Lemma~\ref{lem:truncated_hamburger}, matching the first three moments is feasible for all $\epsmax\in(0,1/2)$. However, we also have
\begin{align*}
    \det\bfH_2=\epsmax^{-3}s^2\Big[(3\epsmax-1) s^4+\epsmax -1\Big].
\end{align*}
When $\epsmax\leq 1/3$, we have $\det \bfH_2<0$ if $s^2>0$, and $\det \bfH_1<0$ if $s^2=0$, indicating that $\bfH_3$ can never be PSD no matter what $m_5$ and $m_6$ we augment. Therefore, applying Lemma~\ref{lem:truncated_hamburger} with $k=4$, we conclude the impossibility of matching the first-to-fourth moments when $\epsmax\leq 1/3$.

Although we have shown the existence of priors that match up to the third moment, to properly apply Lemma~\ref{lem:chi2_subgaussian}, we need $\mu_0$ and $\mu_1$ to be $O_{\delta,\epsmax}(1)$-sub-Gaussian so that the unmatched moments only have a controllable contribution to the chi-square divergence. It suffices to find $\mu_0$ and $\mu_1$ with finite supports that depend on at most $\delta$ and $\epsmax$. 
This is true according to Lemma~\ref{lem:finite_support} from Appendix~\ref{sec:aux}, which also guarantees the existence of a four-point-supported moment-matching $Q''$, as long as we set $s^2>(1-\epsmax)/\epsmax$ to ensure the positive definiteness of the Hankel matrix $\bfH_1$ and thus the $0$-th to $3$-rd terms of~\eqref{eq:moments_q''} become a valid truncated moment sequence. Moreover, the support of such a $Q''$ depends only on $\epsmax$ since the truncated moment sequence \eqref{eq:moments_q''} only depends on $\epsmax$ and $s^2$ (which is set to depend on $\epsmax$ only). The number of supporting points can be further reduced from $4$ to $2$. In Appendix~\ref{sec:proof_unknown_var_lower}, a concrete example of $Q''$ is selected to prove Theorem~\ref{thm:unknown_var_lower}.

Although we do not show any upper bound for $\epsmax\in(1/3,1/2)$ in the unknown variance setting, we present the following lower bound results proved in Appendix~\ref{sec:proof_unknown_var_lower_above_third}.
\begin{Theorem}\label{thm:unknown_var_lower_above_third}
Fix any $\delta\in(0,1/4)$, $\epsmax\in(1/3,1/2)$, and assume that $\sigma^2>0$ is unknown. For any confidence interval $\CIhat$ that satisfies
\begin{align*}
    \inf_{\theta,\, \sigma^2}\, \inf_{\varepsilon\in[0,\epsmax],\, Q}\, P_{\theta,\sigma,\varepsilon,Q}(\theta\in\CIhat)\geq 1-\delta,\quad \inf_{\theta,\, \sigma^2}\, \inf_{\varepsilon\in[0,\epsmax],\, Q}\, P_{\theta,\sigma,\varepsilon,Q}(\abs{\CIhat}\leq r(n,\sigma,\varepsilon))\geq 1-\delta,
\end{align*}
it must hold that
\begin{align*}
    r(n,\sigma,\varepsilon)\geq\begin{dcases}
        c\sigma n^{-1/16},\quad &\textnormal{if}\ \epsmax\in(1/3,7/15],\\
        c\sigma n^{-1/24},\quad &\textnormal{if}\ \epsmax\in(7/15,1/2),
    \end{dcases}
\end{align*}
for all $(n,\sigma,\varepsilon)$ and an absolute constant $c$ that depends only on $\delta$ and $\epsmax$. 
\end{Theorem}
As the above theorem suggests, the rate may undergo additional sharp transitions as the maximum contamination strength $\epsmax$ increases. Therefore, the $\epsmax>1/3$ setting is substantially harder than $\epsmax\leq 1/3$, and our upper bound $\sigma(n^{-1/8}+\varepsilon^{1/4}/\sqrt{\log(en\varepsilon^2)})$ for $\epsmax\in(0,1/3]$ cannot continue to hold when $\epsmax$ is above $1/3$.

\section{Discussions}\label{sec:discussion}

\subsection{No Adaptation Between Huber's and Efron's Models}\label{sec:no_adaptation_between}
Adaptive robust confidence intervals for Huber's model have been studied in \cite{luo2024adaptive}, where a procedure is provided that achieves $\sigma(1/\sqrt{\log(1/\varepsilon)}+1/\sqrt{\log(n)})$ length. Since Efron's model is a subclass of the more general Huber's formulation, one may question whether the algorithm in \cite{luo2024adaptive} is already locally optimal in the noise-oblivious regime. The answer is negative. As we will show in what follows, any method with guaranteed coverage over all Huber's models, including the one in \cite{luo2024adaptive}, cannot attain the optimal length locally within Efron's subclass. Let
\begin{align*}
    P_{\theta,\sigma,\varepsilon,Q}^{(\textnormal{Huber})}=(1-\varepsilon)\cN(\theta,\sigma^2)+\varepsilon Q,\quad P_{\theta,\sigma,\varepsilon,Q}^{(\textnormal{Efron})}=(1-\varepsilon)\cN(\theta,\sigma^2)+\varepsilon Q*\cN(0,\sigma^2).
\end{align*}
We claim the following:
\begin{Theorem}[No adaptation between Huber's and Efron's models]\label{thm:no_adaptation_between}
Let $\delta\in(0,1/4)$ and $\epsmax\in(0,1/2)$ be fixed constants.
\begin{enumerate}
    \item If $\sigma^2>0$ is known, then for any adaptive confidence interval $\CIhat$ that satisfies
\begin{align*}
    \inf_{\theta}\,\sup_{\varepsilon\in[0,\epsmax],\, Q} \Big(P_{\theta,\sigma,\varepsilon,Q}^{(\textnormal{Huber})}\Big)^n(\theta\in\CIhat)\geq 1-\delta,\quad  \inf_{\theta}\,\sup_{\varepsilon\in[0,\epsmax],\, Q} \Big(P_{\theta,\sigma,\varepsilon,Q}^{(\textnormal{Efron})}\Big)^n(\abs{\CIhat}\leq r)\geq 1-\delta,
\end{align*}
the length function $r=r(n,\sigma,\varepsilon)$ must satisfy $r\gtrsim_{\delta,\epsmax}\sigma/\sqrt{\log(n)}$ for all $(n,\sigma,\varepsilon)$.
\item If $\sigma^2>0$ is unknown, then for any adaptive confidence interval $\CIhat$ that satisfies
\begin{align*}
    \inf_{\theta,\sigma^2}\,\sup_{\varepsilon\in[0,\epsmax],\, Q} \Big(P_{\theta,\sigma,\varepsilon,Q}^{(\textnormal{Huber})}\Big)^n(\theta\in\CIhat)\geq 1-\delta,\quad  \inf_{\theta,\sigma^2}\,\sup_{\varepsilon\in[0,\epsmax],\, Q} \Big(P_{\theta,\sigma,\varepsilon,Q}^{(\textnormal{Efron})}\Big)^n(\abs{\CIhat}\leq r)\geq 1-\delta,
\end{align*}
the length function $r=r(n,\sigma,\varepsilon)$ must satisfy $r\gtrsim_{\delta,\epsmax}\sigma$ for all $(n,\sigma,\varepsilon)$.
\end{enumerate}
\end{Theorem}
As the above result conveys, it is a strict dichotomy to choose between:
\begin{enumerate}
    \item guaranteed coverage on Huber's models;
    \item local optimal length on Efron's models for all $\varepsilon\in[0,\epsmax]$.
\end{enumerate}
Therefore, Huber's and Efron's formulations differ not only in their real-world implications; there is also a statistical barrier that prevents any inference procedure from adapting from the former to the latter. 

\subsection{Geometric Intuition About the Adaptation Cost}
In this section, we develop further intuition about the adaptation cost through geometric visualization. We will focus on the known-$\sigma^2$, unknown-$\varepsilon$ setting to illustrate the hardness of an unknown contamination/non-null proportion. Recall that, in Section~\ref{sec:lower_bound}, the lower bound for the adaptive setting is characterized by the compound testing problem \eqref{eq:known_var_test}, which we recall below:
\begin{align*}
    H_0(\sigma,\epsmax):\ X_{1:n}\iid P\in\{P_{0,\sigma,\epsmax,Q}:\, Q\}\quad\textnormal{v.s.}\quad H_1(r,\sigma,\varepsilon):\ X_{1:n}\iid P\in\{P_{r,\sigma,\varepsilon,Q}:\, Q\}.
\end{align*}
However, for the nonadaptive setting in which $\varepsilon$ is known, \eqref{eq:known_var_test} is no longer a valid lower bound construction. Instead, one should consider the following version as used in \cite{kotekal2025optimal}:
\begin{align}\label{eq:known_var_test_nonadaptive}
    H_0(\sigma,\varepsilon):\ X_{1:n}\iid P\in\{P_{0,\sigma,\varepsilon,Q}:\, Q\}\quad\textnormal{v.s.}\quad H_1(r,\sigma,\varepsilon):\ X_{1:n}\iid P\in\{P_{r,\sigma,\varepsilon,Q}:\, Q\}.
\end{align}
Compared to \eqref{eq:known_var_test}, the null hypothesis in \eqref{eq:known_var_test_nonadaptive} is a smaller class, and therefore the null and alternative in \eqref{eq:known_var_test_nonadaptive} are easier to separate than those in \eqref{eq:known_var_test} given the same $r$. However, this is just a clich\'{e} saying that adaptation is never easier. Only after we have constructed hard instances for \eqref{eq:known_var_test} and \eqref{eq:known_var_test_nonadaptive} can we conclude the existence of a concrete adaptation cost. 

In fact, there is a much simpler way for us to gain quantitative insight into why the rate degrades from $\sigma(n^{-1/2}+\varepsilon/\sqrt{\log(en\varepsilon^2)})$ to $\sigma(n^{-1/4}+\varepsilon^{1/2}/\sqrt{\log(en\varepsilon^2)})$ in the adaptive setting. This is by switching to the upper-bound perspective to investigate when \eqref{eq:known_var_test} and \eqref{eq:known_var_test_nonadaptive} can be consistently tested using any (empirical) functional of the characteristic function $\phi(t)$. Recall that in Efron's model, the characteristic function is given in \eqref{eq:efron_cf} by
\begin{align*}
    \phi(t)=(1-\varepsilon)e^{\im\theta t-\sigma^2t^2/2}+\varepsilon e^{-\sigma^2t^2/2} \xi(t),
\end{align*}
where $\xi(t)=\E_Q[\exp(\im tX)]$. We begin by locating $\phi(t)$ under the null and alternative of \eqref{eq:known_var_test} separately. Using the fact that $\abs{\xi(t)}=\abs{\E_Q[e^{\im tX}]}\leq 1$ for all $t\in\bbR$, we obtain:
\begin{align*}
    \textrm{adaptive}\ \eqref{eq:known_var_test}:\quad e^{\sigma^2t^2/2}\phi(t)\in
    \begin{dcases}
         \overline{\bbD}\Big((1-\epsmax),\epsmax\Big)\quad & \textrm{under}\ H_0(\sigma,\epsmax),\\
         \overline{\bbD}\Big((1-\varepsilon)e^{\im rt},\varepsilon\Big)\quad & \textrm{under}\ H_1(r,\sigma,\varepsilon).
    \end{dcases}
\end{align*}
Similarly, for \eqref{eq:known_var_test_nonadaptive}, we have 
\begin{align*}
    \textrm{non-adaptive}\ \eqref{eq:known_var_test_nonadaptive}:\quad e^{\sigma^2t^2/2}\phi(t)\in
    \begin{dcases}
         \overline{\bbD}\Big((1-\varepsilon),\varepsilon\Big)\quad & \textrm{under}\ H_0(\sigma,\varepsilon),\\
         \overline{\bbD}\Big((1-\varepsilon)e^{\im rt},\varepsilon\Big)\quad & \textrm{under}\ H_1(r,\sigma,\varepsilon).
    \end{dcases}
\end{align*}
See Figure~\ref{fig:two_disks} for illustrations. 
\begin{figure}[ht]
    \centering
    \includegraphics[width=1.0\linewidth]{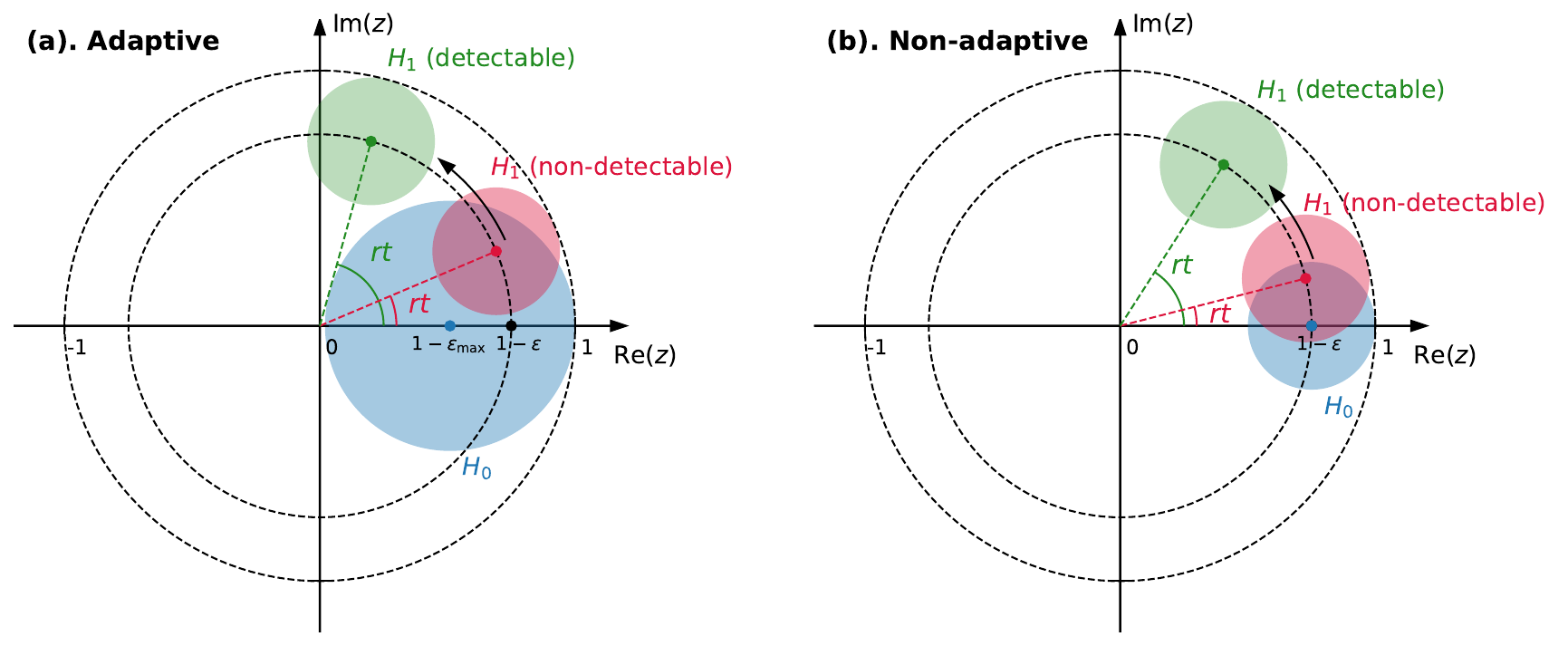}
    \caption{Possible locations of $e^{\sigma^2t^2/2}\phi(t)$ on the complex plane $\bbC$ under the null and alternative hypotheses of: (left). the adaptive testing problem \eqref{eq:known_var_test}; (right). the non-adaptive testing problem \eqref{eq:known_var_test_nonadaptive}.}
    \label{fig:two_disks}
\end{figure}
In the population regime (with infinite samples), the null and alternative are distinguishable with $\phi(t)$ only if the null disk (in blue) and the alternative disk (in red or green) have no contact with each other. As seen in Figure~\ref{fig:two_disks}, this will require the alternative disk to be rotated through a large enough angle $rt$. In the adaptive setting \eqref{eq:known_var_test}, the null disk is significantly larger than its counterparts in the non-adaptive setting \eqref{eq:known_var_test_nonadaptive}, and will therefore require a larger rotation angle. To quantitatively inspect the difference between these two settings, we write down the Euclidean distance between the null and alternative disks: 
\begin{align*}
    &\ \mathrm{dist}\left(\overline{\bbD}\Big((1-\epsmax),\epsmax\Big),\, \overline{\bbD}\Big((1-\varepsilon)e^{\im rt},\varepsilon\Big)\right)\\
    &=\left(\sqrt{(\epsmax-\varepsilon)^2+4(1-\epsmax)(1-\varepsilon)\sin^2(rt/2)}-(\epsmax+\varepsilon)\right)_+.
\end{align*}
When $\varepsilon\leq \epsmax-\Omega(1)$, we can expand the distance for small $\abs{rt}$ as:
\begin{align*}
    \textrm{adaptive}\ \eqref{eq:known_var_test}:\quad \mathrm{dist}\left(\overline{\bbD}\Big((1-\epsmax),\epsmax\Big),\, \overline{\bbD}\Big((1-\varepsilon)e^{\im rt},\varepsilon\Big)\right)\asymp_{\epsmax} (rt)^2-\varepsilon,
\end{align*}
\begin{align*}
    \textrm{non-adaptive}\ \eqref{eq:known_var_test_nonadaptive}:\quad \mathrm{dist}\left(\overline{\bbD}\Big((1-\varepsilon),\varepsilon\Big),\, \overline{\bbD}\Big((1-\varepsilon)e^{\im rt},\varepsilon\Big)\right)\asymp \abs{rt}-\varepsilon.
\end{align*}
We observe that the gap between the two disks increases quadratically with the rotation angle $\abs{rt}$ in the adaptive setting, while it increases linearly with $\abs{rt}$ in the nonadaptive setting. Thus, the population-level gap is significantly smaller in the former setting, indicating that the signal of the testing problem \eqref{eq:known_var_test} is intrinsically weaker than that of \eqref{eq:known_var_test_nonadaptive}. If these population gaps overwhelm the empirical fluctuations of $\phi(t)$, we can test with $\phi_n(t)$ using finite samples. For the adaptive setting \eqref{eq:known_var_test}, this condition becomes
\begin{align*}
    (rt)^2-\varepsilon\gtrsim e^{O(\sigma t^2)}/\sqrt{n}\ \Longleftrightarrow\  r^2\gtrsim \frac{\varepsilon}{t^2}+\frac{e^{O(\sigma t^2)}}{t^2\sqrt{n}}.
\end{align*}
Optimizing over $t>0$ produces precisely the adaptive length of $\sigma(n^{-1/4}+\varepsilon^{1/2}/\sqrt{1\vee \log(en\varepsilon^2)})$ as in Theorem~\ref{thm:known_var}. In contrast, for the nonadaptive setting, the condition becomes
\begin{align*}
    \abs{rt}-\varepsilon\gtrsim e^{O(\sigma t^2)}/\sqrt{n}\ \Longleftrightarrow\  r\gtrsim \frac{\varepsilon}{t}+\frac{e^{O(\sigma t^2)}}{t\sqrt{n}}.
\end{align*}
Optimizing over $t>0$ produces precisely the non-adaptive length (point estimation rate) of $\sigma(n^{-1/2}+\varepsilon/\sqrt{1\vee \log(en\varepsilon^2)})$ proved in \cite{kotekal2025optimal}. In summary, the $\varepsilon$-adaptivity induces a smaller order of geometric separation (from linear to quadratic), and therefore results in a slower rate for the length of confidence intervals.

\subsection{Higher-Order Certificates}
We have seen in Theorem~\ref{thm:unknown_var_lower_above_third} that in the unknown-variance setting, the rate will further degrade to $\Omega(\sigma n^{-1/16})$ when $\epsmax$ exceeds $1/3$ and to $\Omega(\sigma n^{-1/24})$ when $\epsmax$ exceeds $7/15$. Although we have not provided any upper bound on the length in these regimes, we conjecture that it can be tackled with Carath\'{e}odory's PSD certificates in higher orders. We have seen in Section~\ref{sec:upper_bound_known_var} the application of the order-1 (modulus-based) certificates to the known-variance setting, and in Section~\ref{sec:upper_bound_unknown_var} the application of order-2 certificates to the unknown-variance setting when $\epsmax\leq 1/3$. Here, we briefly discuss what higher-order PSD certificates should look like. When $m=3$, the $4\times 4$ matrix $\bfT_m[\phi](t)$ has $2^4-1=15$ principal minors, leading to $15$ polynomial inequalities in $\{1,\phi(t),\phi(2t),\phi(3t)\}$ and their complex conjugates. Many of these inequalities are trivial or redundant, and the PSD condition can eventually be narrowed down to $4$ inequalities:
\begin{align*}
    \begin{dcases}
        1-\abs{\phi(t)}^2\geq 0,\\
        (1-\abs{\phi(t)}^2)^2-\abs{\phi(t)^2-\phi(2t)}^2\geq 0,\\
        (1-\abs{\phi(t)}^2)^2-\abs{\phi(t)^2-\phi(3t)}^2\geq 0,\\
        1-3\abs{\phi(t)}^2-2\abs{\phi(2t)}^2-\abs{\phi(3t)}^2+(\abs{\phi(t)}^2-\abs{\phi(2t)}^2)^2+\abs{\phi(t)}^2\abs{\phi(3t)}^2\\
        \quad +2(\phi(t)^2\overline{\phi(2t)}+
        \overline{\phi(t)}^2\phi(2t))+2(\phi(t)\phi(2t)\overline{\phi(3t)}+\overline{\phi(t)\phi(2t)}\phi(3t))-(\phi(t)^3\overline{\phi(3t)}\\
        \quad +\overline{\phi(t)}^3\phi(3t))-(\phi(t)\overline{\phi(2t)}^2\phi(3t)+\overline{\phi(t)}\phi(2t)^2\overline{\phi(3t)})\geq 0.
    \end{dcases}
\end{align*}
The expression is already highly complicated. However, still only a finite number of certificates need to be analyzed, and the concentration of their empirical values can still be controlled in terms of $\abs{\mathrm{poly}(\phi_n(t))-\mathrm{poly}(\phi(t))}$. To obtain an adaptive robust confidence interval in these settings, we only need to change the certification inequalities in Algorithm~\ref{alg:unknown_var} and re-determine the slack terms and the frequency set.

\bibliography{ref}

\newcommand{\etalchar}[1]{$^{#1}$}
\begin{thebibliography}{DKK{\etalchar{+}}18}

\bibitem[Car07]{caratheodory1907variabilitatsbereich}
Constantin Carath{\'e}odory.
\newblock {\"U}ber den variabilit{\"a}tsbereich der koeffizienten von potenzreihen, die gegebene werte nicht annehmen.
\newblock {\em Mathematische Annalen}, 64(1):95--115, 1907.

\bibitem[Car11]{caratheodory1911variabilitatsbereich}
Constantin Carath{\'e}odory.
\newblock {\"U}ber den variabilit{\"a}tsbereich der fourier’schen konstanten von positiven harmonischen funktionen.
\newblock {\em Rendiconti Del Circolo Matematico di Palermo (1884-1940)}, 32(1):193--217, 1911.

\bibitem[Car98]{carothers1998short}
Neal~L Carothers.
\newblock A short course on approximation theory.
\newblock {\em Bowling Green State University, Bowling Green, OH}, 38, 1998.

\bibitem[CCNT21]{comminges2021adaptive}
L~Comminges, O~Collier, M~Ndaoud, and AB~Tsybakov.
\newblock Adaptive robust estimation in sparse vector model.
\newblock {\em The Annals of Statistics}, 49(3):1347--1377, 2021.

\bibitem[CDRV21]{carpentier2021estimating}
Alexandra Carpentier, Sylvain Delattre, Etienne Roquain, and Nicolas Verzelen.
\newblock Estimating minimum effect with outlier selection.
\newblock {\em The Annals of Statistics}, 49(1):272--294, 2021.

\bibitem[CJ10]{cai2010optimal}
T~Tony Cai and Jiashun Jin.
\newblock Optimal rates of convergence for estimating the null density and proportion of nonnull effects in large-scale multiple testing.
\newblock {\em The Annals of Statistics}, pages 100--145, 2010.

\bibitem[CV19]{carpentier2019adaptive}
Alexandra Carpentier and Nicolas Verzelen.
\newblock Adaptive estimation of the sparsity in the gaussian vector model.
\newblock {\em The Annals of Statistics}, 47(1):93--126, 2019.

\bibitem[DIKL26]{diakonikolas2026sample}
Ilias Diakonikolas, Giannis Iakovidis, Daniel~M Kane, and Sihan Liu.
\newblock Sample complexity bounds for robust mean estimation with mean-shift contamination.
\newblock {\em arXiv preprint arXiv:2602.22130}, 2026.

\bibitem[DIKP25]{diakonikolas2025efficient}
Ilias Diakonikolas, Giannis Iakovidis, Daniel~M Kane, and Thanasis Pittas.
\newblock Efficient multivariate robust mean estimation under mean-shift contamination.
\newblock {\em arXiv preprint arXiv:2502.14772}, 2025.

\bibitem[DKK{\etalchar{+}}18]{diakonikolas2018robustly}
Ilias Diakonikolas, Gautam Kamath, Daniel~M Kane, Jerry Li, Ankur Moitra, and Alistair Stewart.
\newblock Robustly learning a gaussian: Getting optimal error, efficiently.
\newblock In {\em Proceedings of the Twenty-Ninth Annual ACM-SIAM Symposium on Discrete Algorithms}, pages 2683--2702. SIAM, 2018.

\bibitem[DM22]{dalalyan2022all}
Arnak~S Dalalyan and Arshak Minasyan.
\newblock All-in-one robust estimator of the gaussian mean.
\newblock {\em The Annals of Statistics}, 50(2):1193--1219, 2022.

\bibitem[Efr04]{efron2004large}
Bradley Efron.
\newblock Large-scale simultaneous hypothesis testing: the choice of a null hypothesis.
\newblock {\em Journal of the American Statistical Association}, 99(465):96--104, 2004.

\bibitem[Efr08a]{efron2008microarrays}
Bradley Efron.
\newblock Microarrays, empirical bayes and the two-groups model.
\newblock {\em Statistical Science}, 23(1):1--22, 2008.

\bibitem[Efr08b]{efron2008simultaneous}
Bradley Efron.
\newblock Simultaneous inference: When should hypothesis testing problems be combined?
\newblock {\em The Annals of Applied Statistics}, pages 197--223, 2008.

\bibitem[ET02]{efron2002empirical}
Bradley Efron and Robert Tibshirani.
\newblock Empirical bayes methods and false discovery rates for microarrays.
\newblock {\em Genetic epidemiology}, 23(1):70--86, 2002.

\bibitem[ETST01]{efron2001empirical}
Bradley Efron, Robert Tibshirani, John~D Storey, and Virginia Tusher.
\newblock Empirical bayes analysis of a microarray experiment.
\newblock {\em Journal of the American statistical association}, 96(456):1151--1160, 2001.

\bibitem[HL18]{hopkins2018mixture}
Samuel~B Hopkins and Jerry Li.
\newblock Mixture models, robustness, and sum of squares proofs.
\newblock In {\em Proceedings of the 50th Annual ACM SIGACT Symposium on Theory of Computing}, pages 1021--1034, 2018.

\bibitem[Hub64]{huber1964robust}
Peter~J Huber.
\newblock Robust estimation of a location parameter.
\newblock {\em The Annals of Mathematical Statistics}, 35(1):73--101, 1964.

\bibitem[JC07]{jin2007estimating}
Jiashun Jin and T~Tony Cai.
\newblock Estimating the null and the proportion of nonnull effects in large-scale multiple comparisons.
\newblock {\em Journal of the American Statistical Association}, 102(478):495--506, 2007.

\bibitem[Jin08]{jin2008proportion}
Jiashun Jin.
\newblock Proportion of non-zero normal means: universal oracle equivalences and uniformly consistent estimators.
\newblock {\em Journal of the Royal Statistical Society Series B: Statistical Methodology}, 70(3):461--493, 2008.

\bibitem[KG25a]{kotekal2025optimal}
Subhodh Kotekal and Chao Gao.
\newblock Optimal estimation of the null distribution in large-scale inference.
\newblock {\em IEEE Transactions on Information Theory}, 2025.

\bibitem[KG25b]{kotekal2025sparsity}
Subhodh Kotekal and Chao Gao.
\newblock Sparsity meets correlation in gaussian sequence model.
\newblock {\em The Annals of Statistics}, 53(3):1095--1122, 2025.

\bibitem[Las09]{lasserre2009moments}
Jean~Bernard Lasserre.
\newblock {\em Moments, positive polynomials and their applications}, volume~1.
\newblock World Scientific, 2009.

\bibitem[Lep91]{lepskii1991problem}
OV~Lepskii.
\newblock On a problem of adaptive estimation in gaussian white noise.
\newblock {\em Theory of Probability \& Its Applications}, 35(3):454--466, 1991.

\bibitem[Lep92]{lepskii1992asymptotically}
OV~Lepskii.
\newblock Asymptotically minimax adaptive estimation. i: Upper bounds. optimally adaptive estimates.
\newblock {\em Theory of Probability \& Its Applications}, 36(4):682--697, 1992.

\bibitem[LG24]{luo2024adaptive}
Yuetian Luo and Chao Gao.
\newblock Adaptive robust confidence intervals.
\newblock {\em arXiv preprint arXiv:2410.22647}, 2024.

\bibitem[MWY25]{ma2025best}
Yun Ma, Yihong Wu, and Pengkun Yang.
\newblock On the best approximation by finite gaussian mixtures.
\newblock {\em IEEE Transactions on Information Theory}, 2025.

\bibitem[Sch17]{schmudgen2017moment}
Konrad Schm{\"u}dgen.
\newblock {\em The Moment Problem}, volume 277.
\newblock Springer, 2017.

\bibitem[SS10]{stein2010complex}
Elias~M Stein and Rami Shakarchi.
\newblock {\em Complex analysis}, volume~2.
\newblock Princeton University Press, 2010.

\bibitem[Ver18]{vershynin2018high}
Roman Vershynin.
\newblock {\em High-dimensional probability: An introduction with applications in data science}, volume~47.
\newblock Cambridge university press, 2018.

\bibitem[WY20]{wu2020polynomial}
Yihong Wu and Pengkun Yang.
\newblock Polynomial methods in statistical inference: Theory and practice.
\newblock {\em Foundations and Trends{\textregistered} in Communications and Information Theory}, 17(4):402--585, 2020.

\end{thebibliography}
\bibliographystyle{alpha}
\appendix

\section{Deferred Proofs of Performance Guarantees}\label{sec:proof_ub}

\subsection{Proofs of Performance Guarantees in the Known-Variance Setting}\label{sec:proof_known_var_ub}

\subsubsection{Proof of Lemma~\ref{lem:quadratic_gap}}
\begin{proof}[Proof of Lemma~\ref{lem:quadratic_gap}]
We have
\begin{align*}
    \mathrm{dist}\Big(\overline{\bbD}(-1+2(1-\varepsilon)e^{-\im b},2\varepsilon),0\Big)&\geq \abs{-1+2(1-\varepsilon)e^{-\im b}}-2\varepsilon\\
    &=\sqrt{\Big(-1+2(1-\varepsilon)\cos(b)\Big)^2+\Big(2(1-\varepsilon)\sin(b)\Big)^2}-2\varepsilon\\
    &=\sqrt{\Big(2-2\varepsilon-\cos(b)\Big)^2+\Big(1-\cos^2(b)\Big)}-2\varepsilon\\
    &\geq (1-\cos(b))+1-4\varepsilon\\
    &\geq 2(b/\pi)^2+1-4\varepsilon,
\end{align*}
where we apply Lemma~\ref{lem:cosine_quadratic} to the last line.
\end{proof}

\subsubsection{Proof of Theorem~\ref{thm:known_var}}
\begin{Proposition}\label{prop:known_var_coverage}
    With the setup of Theorem~\ref{thm:known_var}, for the $\CIhat$ defined by \eqref{eq:CI_known_var}, we have
    \begin{align*}
\inf_{\theta\in\bbR}\ \inf_{\varepsilon\in[0,\epsmax],\, Q}\ P^n_{\theta,\sigma,\varepsilon,Q}\Big(\theta\in\CIhat\Big)\geq 1-\delta,
    \end{align*}
whenever $n\geq N_0$, where $N_0$ is specified by Lemma~\ref{lem:pilot_unknown_var}.
\end{Proposition}
\begin{proof}
By Lemma~\ref{lem:chf_concentration}, the event
\begin{align}\label{eq:concentration_event_known_var_proof}
    \cE=\Big\{\abs{\phi_n(t)-\phi(t)}\leq \kappa\sigma t \sqrt{\frac{8\log(10/\delta)}{n}}\leq  \kappa e^{\sigma^2 t^2/2}\sqrt{\frac{8\log(10/\delta)}{n}},\quad \forall\ t\in\cT\Big\}
\end{align}
holds with probability at least $1-\delta/2$, regardless of the choice of the underlying distribution. Also, the pilot concentration event 
\begin{align}\label{eq:pilot_event_known_var_proof}
    \widetilde{\cE}=\Big\{\abs{\widetilde{\theta}-\theta}\leq M\sigma/\sqrt{\log(en)}\Big\}
\end{align}
holds with probability exceeding $1-\delta/2$ whenever $n\geq N_0$ according to Lemma~\ref{lem:pilot_unknown_var}. Therefore, on the union event $\cE\medcap \widetilde{\cE}$ that happens with at least $1-\delta$ probability, we have both
\begin{align*}
    \theta\in \widetilde{\CI}=\Big[\widetilde{\theta}-M\sigma/\sqrt{\log(en)},\ \widetilde{\theta}+M\sigma/\sqrt{\log(en)}\Big]
\end{align*}
and, by the triangle inequality, $\theta$ passes the certification since
\begin{align*}
    \abs{\Upsilon_n(t;\theta)}&=\abs{2e^{-\im \theta t+\sigma^2t^2/2}\phi_n(t)-1}\\
    &\leq \abs{2e^{-\im \theta t+\sigma^2t^2/2}\phi_n(t)-2e^{-\im \theta t+\sigma^2t^2/2}\phi(t)}+\abs{2e^{-\im \theta t+\sigma^2t^2/2}\phi(t)-1}\\
    &=2e^{\sigma^2t^2/2}\abs{\phi_n(t)-\phi(t)}+\abs{\Upsilon(t;\theta)}\\
    &\leq 2\kappa e^{\sigma^2 t^2/2}\cdot e^{\sigma^2 t^2/2}\sqrt{\frac{8\log(10/\delta)}{n}}+1\\
    &\leq 2\kappa e^{\sigma^2 t^2}\sqrt{\frac{8\log(10/\delta)}{n}}+1,\quad \forall\ t\in\cT,
\end{align*}
where in the second-to-last line we have used the uniform concentration result and the fact that $\Upsilon(t;\theta)$ is a genuine characteristic function. Therefore, $\theta\in\CIhat$ on $\cE\medcap \widetilde{\cE}$, and the desired coverage property follows.
\end{proof}

\begin{Proposition}\label{prop:known_var_length}
    With the setup of Theorem~\ref{thm:known_var}, for the $\CIhat$ defined in Section~\ref{sec:upper_bound_known_var} with $\kappa\geq \sqrt{2}\vee (M/\pi)$, there exist constants $C$ and $N$ that depend only on $\epsmax$, $\delta$, and $\kappa$, such that
    \begin{align*}
        \inf_{\theta\in\bbR}\ \inf_{\varepsilon\in[0,\epsmax],\, Q}\ P^n_{\theta,\sigma,\varepsilon,Q}\Big(\CIhat\bigcap \{\mu:\abs{\mu-\theta}> \overline{r}\}=\emptyset\Big)\geq 1-\delta,
    \end{align*}
    whenever $n\geq N$, where
    \begin{align*}
        \overline{r}(n,\varepsilon,\sigma)=C\sigma\left(\frac{1}{n^{1/4}}+\frac{\varepsilon^{1/2}}{\sqrt{1\vee \log(en\varepsilon^2)}}\right).
    \end{align*}
\end{Proposition}
\begin{proof}
Let $r=\mu-\theta$. Consider the good union event $\cE\medcap \widetilde{\cE}$ (defined in \eqref{eq:concentration_event_known_var_proof} and \eqref{eq:pilot_event_known_var_proof}) that holds with probability at least $1-\delta$. On this event (more specifically whenever $\widetilde{\cE}$ holds), we have $\abs{\widetilde{\theta}-\theta}\leq M\sigma/\sqrt{\log(en)}$, and thus
\begin{align*}
    \CIhat\subseteq \widetilde{\CI}=\Big[\widetilde{\theta}-M\sigma/\sqrt{\log(en)},\ \widetilde{\theta}+M\sigma/\sqrt{\log(en)}\Big].
\end{align*}
Therefore, all $\mu$ with $r=\abs{\mu-\theta}>M\sigma/\sqrt{\log(en)
}$ are excluded from $\CIhat$. It suffices to show that all $\mu$ with $r=\abs{\mu-\theta}\in(\overline{r},M\sigma/\sqrt{\log(en)}]$ will also be excluded. Also, on this union event, we have
\begin{align*}
    \abs{\Upsilon_n(t;\mu)}-1-\Delta(t)&\geq \abs{\Upsilon(t;\mu)}-\abs{\Upsilon_n(t;\mu)-\Upsilon(t;\mu)}-1-\Delta(t)\\
    &= \abs{-1+2(1-\varepsilon)e^{-\im r t}+2\varepsilon e^{-\im \mu}\xi(t)}-2\abs{\phi_n(t)-\phi(t)}-1-\Delta(t)\\
    &\geq \mathrm{dist}\Big(\overline{\bbD}(-1+2(1-\varepsilon)e^{-\im rt},2\varepsilon),0\Big)-1-4\kappa e^{\sigma^2 t^2}\sqrt{\frac{8\log(10/\delta)}{n}},
\end{align*}
where in the last line we use the fact that $-1+2(1-\varepsilon)e^{-\im r t}+2\varepsilon e^{-\im \mu}\xi(t)$ is within the closed disk $\overline{\bbD}(-1+2(1-\varepsilon)e^{-\im rt},2\varepsilon)$ since $\abs{\xi(t)}\leq 1$. Therefore, applying Lemma~\ref{lem:quadratic_gap}, we can continue to write
\begin{align*}
    \abs{\Upsilon_n(t;\mu)}-1-\Delta(t)&\geq 2(rt/\pi)^2-4\varepsilon-4\kappa e^{\sigma^2 t^2}\sqrt{\frac{8\log(10/\delta)}{n}}\define S(t),
\end{align*}
whenever $t\in\cT$ and $\abs{rt}\leq \pi$. We consider three different cases.
\paragraph{Case I.}When $\varepsilon\in[0,n^{-1/2}]$ and $\abs{r/\sigma}\in[C_1 n^{-1/4}, M/\sqrt{\log(en)}]$, consider $t_1=(\kappa\sigma)^{-1}\in\cT$. Since $\abs{rt_1}\leq M/\kappa$, we have $\abs{rt_1}\leq \pi$ as long as $\kappa \geq M/\pi$. Then, we have
\begin{align*}
    S(t_1)\geq [2C_1^2/(\pi^2\kappa^2)-4-4\kappa e^{1/\kappa^2}\sqrt{8\log(10/\delta)}]\cdot n^{-1/2}.
\end{align*}
As long as both $\kappa\geq M/\pi$ and $C_1^2> 2\pi^2\kappa^2+2\pi^2\kappa^3e^{1/\kappa^2}\sqrt{8\log(10/\delta)}$ hold, we will have $S(t_1)> 0$ and thus $\mu\notin \CIhat$.

\paragraph{Case II.}When $\varepsilon\in[n^{-1/2},\epsmax]$ and $\abs{r/\sigma}\in[C_2\varepsilon^{1/2}/\sqrt{\log(en\varepsilon^2)},M/\sqrt{\log(en)}]$, consider $t_2=(\kappa\sigma)^{-1}\floor{\sqrt{\log(en\varepsilon^2)}}$. Note that $t_2\in\cT$ as long as $B\geq 1$, and $\abs{rt_2}\leq \pi$ as long as $\kappa\geq M/\pi$. Given these, we have
\begin{align*}
    S(t_2)&\geq [2C_2^2/(\pi^2\kappa^2)-4]\varepsilon-4\kappa e^{\kappa^{-2}\log(en\varepsilon^2)}\sqrt{\frac{8\log(10/\delta)}{n}}\\
    &=[2C_2^2/(\pi^2\kappa^2)-4]\varepsilon-4\kappa\sqrt{8\log(10/\delta)}e^{1/\kappa^2}\varepsilon^{2/\kappa^2} n^{1/\kappa^2-1/2}\\
    &\geq [2C_2^2/(\pi^2\kappa^2)-4-4\kappa e^{1/\kappa^2}\sqrt{8\log(10/\delta)}]\varepsilon.
\end{align*}
As long as both $\kappa\geq \sqrt{2}\vee (M/\pi)$ and $C_2^2> 2\pi^2\kappa^2+2\pi^2\kappa^3e^{1/\kappa^2}\sqrt{8\log(10/\delta)}$ hold, we will have $S(t_2)> 0$ and thus $\mu\notin \CIhat$.

Gathering the above conditions, we need to ensure that the algorithmic parameters satisfy
\begin{align*}
    \kappa\geq \sqrt{2}\vee (M/\pi),\quad C_1^2\wedge C_2^2> 2\pi^2\kappa^2+2\pi^2\kappa^3e^{1/\kappa^2}\sqrt{8\log(10/\delta)}
\end{align*}
for the certification process, and also $N\geq N_0$ for the pilot event $\widetilde{\cE}$ to be true. These can be simultaneously satisfied by setting
\begin{align*}
    N=N_0,\quad \kappa=\sqrt{2}\vee (M/\pi),\quad C=C_1=C_2=2\pi\kappa+2\pi \kappa^{3/2}e^{1/(2\kappa^2)}(8\log(10/\delta))^{1/4},
\end{align*}
which concludes the proof.
\end{proof}

\begin{proof}[Proof of Theorem~\ref{thm:known_var}]
Theorem~\ref{thm:known_var} now directly follows from the combination of Proposition~\ref{prop:known_var_coverage} and Proposition~\ref{prop:known_var_length}, with a special notice that both the coverage property in Proposition~\ref{prop:known_var_coverage} and the length property in Proposition~\ref{prop:known_var_length} hold on the same joint event $\cE\medcap \widetilde{\cE}$.
\end{proof}

\subsection{Proofs of Performance Guarantees in the Unknown-Variance Setting}\label{sec:proof_unknown_var_ub}

\subsubsection{Lipschitzness of the Certificates}
With second-order certificates being more sophisticated than their simple modulus-based counterparts in Section~\ref{sec:known_var_derivation}, we introduce the following reformulation that views the population (resp. empirical) certificates as functionals of the population (resp. empirical) characteristic functions.  
\begin{Definition}
    For any function $f:\,\bbR\to \bbC$, any $\mu\in\bbR$ and any $\lambda>0$, we define
\begin{align*}
    \cJ_1[f](t;\mu,\lambda)=\abs{3e^{-\im\mu t+\lambda t^2/2}f(t)-2},
\end{align*}
and also
\begin{align*}
    \cJ_2[f](t;\mu,\lambda)&=\abs{(3e^{-\im \mu t+\lambda t^2/2}f(t)-2)^2-(3e^{-\im \mu (2t)+\lambda (2t)^2/2}f(2t)-2)}
    +\abs{3e^{-\im \mu t+\lambda t^2/2}f(t)-2}^2.
\end{align*}
\end{Definition}
These two functionals $\cJ_1$ and $\cJ_2$ are directly related to the population and empirical PSD certificates through:
\begin{align*}
\begin{dcases}
    \cJ_1[\phi](t;\mu,\lambda)=\abs{\Upsilon(t;\mu,\lambda)},\\
    \cJ_2[\phi](t;\mu,\lambda)=\abs{\Upsilon^2(t;\mu,\lambda)-\Upsilon(2t;\mu,\lambda)}+\abs{\Upsilon(t;\mu,\lambda)}^2,
\end{dcases}
\end{align*}
which always holds for $\mu\in\bbR$ and $\lambda>0$. Similar equations hold when we replace $\phi$ by $\phi_n$ and replace $\Upsilon$ by $\Upsilon_n$. Under the above formulation, we can bound the difference between population and empirical certificate values using the difference between population and empirical characteristic functions, according to the following Lipschitzness property of $\cJ_1$ and $\cJ_2$:

\begin{Lemma}\label{lem:lipschitz}
For any $f,g:\bbR\to \bbC$, we have
\begin{align*}
    \abs{\cJ_1[f](t;\mu,\lambda)-\cJ_1[g](t;\mu,\lambda)}\leq 3e^{\lambda t^2/2}\abs{f(t)-g(t)},
\end{align*}
and also
\begin{align*}
    &\ \abs{\cJ_2[f](t;\mu,\lambda)-\cJ_2[g](t;\mu,\lambda)}\nonumber\\
    &\leq 18e^{\lambda t^2}\abs{f^2(t)-g^2(t)}+ 24 e^{\lambda t^2/2}\abs{f(t)-g(t)}+ 3e^{2\lambda t^2}\abs{f(2t)-g(2t)}.
\end{align*}
\end{Lemma}
\begin{proof}
By definition, we have precisely
\begin{align*}
    \abs{\cJ_1[f](t;\mu,\lambda)-\cJ_1[g](t;\mu,\lambda)}&=3e^{\lambda t^2/2}\abs{f(t)-g(t)},
\end{align*}
since the shared multiplier $e^{-\im \mu t}$ does not affect the modulus. This already proves the first conclusion. To prove the second one, we pair the terms and use the triangle inequality to claim that
\begin{align*}
    \abs{\cJ_2[f](t;\mu,\lambda)-\cJ_2[g](t;\mu,\lambda)}&\leq 9e^{\lambda t^2}\abs{f^2(t)-g^2(t)}+12 e^{\lambda t^2/2}\abs{f(t)-g(t)}\nonumber\\
    &\ +3e^{2\lambda t^2}\abs{f(2t)-g(2t)}+9e^{\lambda t^2}\abs{f^2(t)-g^2(t)}\nonumber\\
    &\ +12e^{\lambda t^2/2}\abs{f(t)-g(t)}\\
    &=18e^{\lambda t^2}\abs{f^2(t)-g^2(t)}+24e^{\lambda t^2/2}\abs{f(t)-g(t)}\nonumber\\
    &\ +3e^{2\lambda t^2}\abs{f(2t)-g(2t)},
\end{align*}
which recovers the second conclusion.
\end{proof}

\subsubsection{Proof of Lemma~\ref{lem:quartic_gap}}
\begin{proof}
We split the problem into three cases.
\paragraph{Case I.}When $\abs{rt}\in(\pi/2,\pi]$. In this case, $\cos(rt)\leq 0$, and therefore,
\begin{align*}
    \abs{\Upsilon(t;r,\lambda)}\Big|_{\phi(t)=e^{-\sigma^2t^2/2}}&=\abs{3e^{-\im rt+(\lambda-\sigma^2)t^2/2}-2}\nonumber\\
    &\geq \abs{3e^{(\lambda-\sigma^2)t^2/2}\cos(rt)-2}\geq \abs{0-2}\geq 1+ (rt/\pi)^4.
\end{align*}
\paragraph{Case II.}When $\abs{rt}\in[0,\pi/2]$ and $\lambda> \sigma^2$. In this case, $e^{(\lambda-\sigma^2)t^2/2}\geq 1$, and therefore,
\begin{align*}
    \abs{\Upsilon(t;r,\lambda)}\Big|_{\phi(t)=e^{-\sigma^2t^2/2}}&=\abs{3e^{-\im rt+(\lambda-\sigma^2)t^2/2}-2}\\
    &=\sqrt{ \abs{3e^{-\im rt}-2}^2+(e^{(\lambda-\sigma^2)t^2/2}-1)(9(e^{(\lambda-\sigma^2)t^2/2}+1)-12)}\\
    &\geq\abs{3e^{-\im rt}-2}\\
    &=1+(\sqrt{13-12\cos(rt)}-1)\\
    &=1+\frac{12(1-\cos(rt))}{\sqrt{13-12\cos(rt)}+1}\\
    &\geq 1+\frac{12(2rt/\pi)^2}{\sqrt{25}+1}\\
    &=1+8(rt/\pi)^2\geq 1+(rt/\pi)^4.
\end{align*}
\paragraph{Case III.}When $\abs{rt}\in[0,\pi/2]$ and $\lambda\leq \sigma^2$. In this case, we have
\begin{align*}
    &\ \left(\abs{\Upsilon^2(t;r,\lambda)-\Upsilon(2t;r,\lambda)}+\abs{\Upsilon(t;r,\lambda)}^2\right)\Big|_{\phi(t)=e^{-\sigma^2t^2/2}}\\
    &\geq\Re\Big((3e^{-\im rt+\lambda t^2/2}e^{-\sigma^2t^2/2}-2)^2-(3e^{-2\im rt+2\lambda t^2}e^{-2\sigma^2t^2}-2)\Big)+\abs{3e^{-\im rt+\lambda t^2/2}e^{-\sigma^2t^2/2}-2}^2\\
    &=1+\Big(6(3e^{(\lambda-\sigma^2)t^2}-e^{2(\lambda-\sigma^2)t^2})\cos^2(rt)-24e^{(\lambda-\sigma^2)t^2/2}\cos(rt)+3e^{2(\lambda-\sigma^2)t^2}+9\Big)\\
    &\define 1+\Big(6a^2(3-a^2)x^2-24ax+3a^4+9\Big),
\end{align*}
where $a=e^{(\lambda-\sigma^2)t^2/2}\in[0,1]$ and $x=\cos(rt)\in[0,1]$. Note that
\begin{align*}
    &\ 6a^2(3-a^2)x^2-24ax+3a^4+9\\
    &=9(1-x)^2+3a^4(1-x)^2+6(a-1)^2(a^2+2a+3)x(1-x)+3(1-a)^3(a+3)x^2\\
    &\geq 9(1-x)^2,\quad \forall\, x\in[0,1],\ a\in[0,1].
\end{align*}
Therefore, in this case, we obtain a lower bound:
\begin{align*}
    \geq 1+9(1-\cos(rt))^2\geq 1+9(2rt/\pi)^4\geq 1+(rt/\pi)^4,
\end{align*}
which completes the proof.
\end{proof}

\subsubsection{Proof of Theorem~\ref{thm:unknown_var}}
\begin{Proposition}\label{prop:unknown_var_coverage}
With the setup of Theorem~\ref{thm:unknown_var}, for the $\CIhat$ defined by \eqref{eq:CI_unknown_var}, we have
    \begin{align*}
\inf_{\theta\in\bbR,\, \sigma^2>0}\ \inf_{\varepsilon\in[0,\epsmax],\, Q}\ P^n_{\theta,\sigma,\varepsilon,Q}\Big(\theta\in\CIhat\Big)\geq 1-\delta,
    \end{align*}
as long as $n\geq 2N_0$ and $\kappa\geq 1/\sqrt{8\log(10)}$.
\end{Proposition}
\begin{proof}
Suppose that both the pilot coverage event $\widetilde{\cE}$ defined by \eqref{eq:pilot_event_unknown_var} and the training-set concentration event
\begin{align}\label{eq:concentration_event_unknown_var_proof}
    \cE(X_{1:\ceil{n/2}})=\Big\{\abs{\phi_n(t)-\phi(t)}\leq \kappa e^{\widetilde{\sigma}_+^2t^2/2}\sqrt{\frac{16\log(10/\delta)}{n}},\ t\in\sqrt{\bbZ}_+\cdot(\kappa\widetilde{\sigma}_+)^{-1}\Big\}
\end{align}
hold. This joint event happens with probability exceeding $1-\delta/2-\delta/2=1-\delta$. (Here, $\phi_n(t)$ denotes the empirical characteristic function evaluated on the training set $X_{1:\ceil{n/2}}$) On this joint event $\widetilde{\cE}\medcap \cE$, we have $\theta\in\widetilde{\CI}$, and it suffices to show that $\theta$ passes the certificates for all $t\in\widetilde{\cT}$ and some $\lambda\in\widetilde{\cV}$. Then, by our construction of $\widetilde{\cV}$, there exists $\lambda_*\in\widetilde{\cV}$ such that $\sigma^2- (\widetilde{\sigma}_+^2-\widetilde{\sigma}_-^2)/\ceil{4\kappa^{-2}\sqrt{n}\log(en)}\leq \lambda_*\leq \sigma^2$ given the event $\widetilde{\cE}$. Define
\begin{align*}
    \psi(t)\define e^{(\lambda_*-\sigma^2)t^2/2}\phi(t).
\end{align*}
Note that for all $t\in\widetilde{\cT}$
\begin{align*}
    \abs{\lambda_*-\sigma^2}t^2\leq \frac{\ceil{\log(en)}(\widetilde{\sigma}_+^2-\widetilde{\sigma}_-^2)}{\ceil{4\kappa^{-2}\sqrt{n}\log(en)}\kappa^2\widetilde{\sigma}_+^2}\leq \frac{1}{4\sqrt{n}}.
\end{align*}
Therefore, we can use $|e^x-1|\leq 2\abs{x}$ when $\abs{x}\leq 1$ to claim that for all $t\in\widetilde{\cT}$
\begin{align*}
    \abs{\phi(t)-\psi(t)}=\abs{1-e^{(\lambda_*-\sigma^2)t^2/2}}\abs{\phi(t)}&\leq \abs{1-e^{(\lambda_*-\sigma^2)t^2/2}}\\
    &\leq 1/(4\sqrt{n}),
\end{align*}
and also
\begin{align*}
    \abs{\phi(2t)-\psi(2t)}\leq \abs{1-e^{2(\lambda_*-\sigma^2)t^2}}\leq 1/\sqrt{n},
\end{align*}
\begin{align*}
    \abs{\phi(t)^2-\psi(t)^2}=\abs{\phi(t)+\psi(t)}\abs{\phi(t)-\psi(t)}\leq 1/(2\sqrt{n}).
\end{align*}
Then, for all $t\in\widetilde{\cT}$, using $\cJ_1[\phi](t;\theta,\sigma^2)\leq 1$, we obtain
\begin{align*}
    \abs{\Upsilon_n(t;\theta,\lambda_*)}&=\cJ_1[\phi_n](t;\theta,\lambda_*)\\
    &\leq \cJ_1[\phi](t;\theta,\lambda_*)+3e^{\lambda_0 t^2/2}\abs{\phi_n(t)-\phi(t)}\\
    &\leq \cJ_1[\phi](t;\theta,\lambda_*)+3\kappa e^{\widetilde{\sigma}_+^2t^2}\sqrt{16\log(10/\delta)/n}\\
    &=\cJ_1[\psi](t;\theta,\sigma^2)+3\kappa e^{\widetilde{\sigma}_+^2t^2}\sqrt{16\log(10/\delta)/n}\\
    &\leq \cJ_1[\phi](t;\theta,\sigma^2)+ 3e^{\sigma^2t^2/2}\abs{\phi(t)-\psi(t)}+3\kappa e^{\widetilde{\sigma}_+^2t^2}\sqrt{16\log(10/\delta)/n}\\
    &\leq 1+6\kappa e^{\widetilde{\sigma}_+^2t^2}\sqrt{16\log(10/\delta)/n}
\end{align*}
as long as $\kappa\geq (16\log(10))^{-1/2}$. Similarly, for the second certificate, we have that for all $t\in\widetilde{\cT}$,
\begin{align*}
    &\ \abs{\Upsilon_n^2(t;\theta,\lambda_*)-\Upsilon_n(2t;\theta,\lambda_*)}+\abs{\Upsilon_n(t;\theta,\lambda_*)}^2\nonumber\\
    &=\cJ_2[\phi_n](t;\theta,\lambda_*)\\
    &\leq \cJ_2[\phi](t;\theta,\lambda_*)+18e^{\widetilde{\sigma}_+^2t^2}\abs{\phi_n^2(t)-\phi^2(t)}\nonumber\\
    &\ +24e^{\widetilde{\sigma}_+^2t^2/2}\abs{\phi_n(t)-\phi(t)}+3e^{2\widetilde{\sigma}_+^2t^2}\abs{\phi_n(2t)-\phi(2t)}\\
    &\leq \cJ_2[\phi](t;\theta,\lambda_*)+63\kappa e^{5\widetilde{\sigma}_+^2t^2/2}\sqrt{16\log(10/\delta)}\\
    &=\cJ_2[\psi](t;\theta,\sigma^2)+63\kappa e^{5\widetilde{\sigma}_+^2t^2/2}\sqrt{16\log(10/\delta)}\\
    &\leq \cJ_2[\phi](t;\theta,\sigma^2)+63\kappa e^{5\widetilde{\sigma}_+^2t^2/2}\sqrt{16\log(10/\delta)}+18e^{\widetilde{\sigma}_+^2t^2}\abs{\phi^2(t)-\psi^2(t)}\nonumber\\
    &+24e^{\widetilde{\sigma}_+^2t^2/2}\abs{\phi(t)-\psi(t)}+3e^{2\widetilde{\sigma}_+^2t^2}\abs{\phi(2t)-\psi(2t)}\\
    &\leq 1+81\kappa e^{5\widetilde{\sigma}_+^2t^2/2}\sqrt{16\log(10/\delta)/n}
\end{align*}
also as long as $\kappa\geq (16\log(10))^{-1/2}$. Therefore, under the given conditions, $\theta$ passes all certificates with $\lambda=\lambda_*$, which implies that $\theta\in\CIhat$ on $\widetilde{\cE}\medcap \cE$.
\end{proof}

\begin{Proposition}\label{prop:unknown_var_length}
    With the setup of Theorem~\ref{thm:unknown_var}, for the $\CIhat$ defined by \eqref{eq:CI_unknown_var} with $\kappa\geq \sqrt{5}\vee (2M/\pi)$, there exist absolute constants $C$ and $N$ that depend only on $\delta$, $\epsmax$, and $\kappa$, such that
    \begin{align*}
        \inf_{\theta\in\bbR,\,\sigma^2>0}\ \inf_{\varepsilon\in[0,\epsmax],\, Q}\ P^n_{\theta,\sigma,\varepsilon,Q}\Big(\CIhat\bigcap \{\mu:\abs{\mu-\theta}\geq \overline{r}\}=\emptyset\Big)\geq 1-\delta
    \end{align*}
    whenever $n\geq N$, where
    \begin{align*}
        \overline{r}(n,\varepsilon,\sigma)=C\sigma\left(\frac{1}{n^{1/8}}+\frac{\varepsilon^{1/4}}{\sqrt{1\vee \log(en\varepsilon^2)}}\right).
    \end{align*}
    Furthermore, the choices of these constants can tolerate Proposition~\ref{prop:unknown_var_coverage}.
\end{Proposition}
\begin{proof}
Again, we assume the joint event $\widetilde{\cE}\medcap \cE$ (defined in \eqref{eq:pilot_event_unknown_var} and \eqref{eq:concentration_event_unknown_var_proof}). This joint event occurs with probability exceeding $1-\delta/2-\delta/2=1-\delta$. On this event and when $\log(en)\geq L$, any $\mu$ with $\abs{\mu-\theta}> 2M\sigma/\sqrt{\log(en)}$ will be excluded from $\widetilde{\CI}$ and therefore excluded from $\CIhat$. It suffices to show that any $\mu$ with $\abs{\mu-\theta}\in[\overline{r},2M\sigma/\sqrt{\log(en)}]$ cannot pass the certificates for any $\lambda\in\widetilde{\cV}$. Define
\begin{align*}
    \phi_0(t)\define e^{\im \theta t-\sigma^2t^2/2}
\end{align*}
as the characteristic function of the clean distribution $\cN(\theta,\sigma^2)$. Then,
\begin{align*}
    \abs{\phi(t)-\phi_0(t)}&=\abs{(1-\varepsilon)e^{\im \theta t-\sigma^2t^2/2}+\varepsilon\xi_Q(t)e^{-\sigma^2t^2/2}-e^{\im \theta t-\sigma^2t^2/2}}\\
    &=\varepsilon e^{-\sigma^2t^2/2}\abs{\xi_Q(t)-e^{\im \theta t}}\\
    &\leq 2\varepsilon e^{-\sigma^2t^2/2}.
\end{align*}
Also,
\begin{align*}
    \abs{\phi^2(t)-\phi_0^2(t)}&=\abs{\phi(t)-\phi_0(t)}\abs{\phi(t)+\phi_0(t)}\leq 4\varepsilon e^{-\sigma^2t^2/2}.
\end{align*}
Denoting $r=\mu-\theta$, we have that for all $t\in\widetilde{\cT}$,
\begin{align*}
    &\ \abs{\Upsilon_n(t;\mu,\lambda)}-1-\Delta_1(t)\nonumber\\
    &= \cJ_1[\phi_n](t;\mu,\lambda)-1-\Delta_1(t)\\
    &\geq \cJ_1[\phi](t;\mu,\lambda)-3e^{\lambda t^2/2}\abs{\phi_n(t)-\phi(t)}-1-\Delta_1(t)\\
    &\geq \cJ_1[\phi_0](t;\mu,\lambda)-3e^{\lambda t^2/2}\abs{\phi(t)-\phi_0(t)}-3e^{\lambda t^2/2}\abs{\phi_n(t)-\phi(t)}-1-\Delta_1(t)\\
    &\geq \cJ_1[\phi_0](t;\mu,\lambda)-1-6\varepsilon e^{(\widetilde{\sigma}_+^2-\widetilde{\sigma}_-^2)t^2/2}-6\kappa e^{\widetilde{\sigma}^2_+t^2}\sqrt{16\log(10/\delta)/n}\\
    &\geq \cJ_1[\phi_0](t;\mu,\lambda)-1-6\varepsilon e^{(2\frac{L}{\log(en)}+\frac{L^2}{\log^2(en)})\sigma^2t^2/2}-6\kappa e^{\widetilde{\sigma}^2_+t^2}\sqrt{16\log(10/\delta)/n}\\
    &\geq \cJ_1[\phi_0](t;\mu,\lambda)-1-6\varepsilon e^{3L\sigma^2t^2/(2\log(en))}-6\kappa e^{\widetilde{\sigma}^2_+t^2}\sqrt{16\log(10/\delta)/n}\\
    &=\cJ_1[e^{-\sigma^2t^2/2}](t;r,\lambda)-1-6\varepsilon e^{3L\sigma^2t^2/(2\log(en))}-6\kappa e^{\widetilde{\sigma}^2_+t^2}\sqrt{16\log(10/\delta)/n}\\
    &\geq \cJ_1[e^{-\sigma^2t^2/2}](t;r,\lambda)-1-6\varepsilon e^{3L\sigma^2t^2/(2\log(en))}-6\kappa e^{4\sigma^2t^2}\sqrt{16\log(10/\delta)/n}
\end{align*}
as long as $\log(en)\geq L$. Similarly,
\begin{align*}
    &\ \abs{\Upsilon_n^2(t;\mu,\lambda)-\Upsilon_n(2t;\mu,\lambda)}+\abs{\Upsilon_n(t;\mu,\lambda)}^2\nonumber-1-\Delta_2(t)\\
    &=\cJ_2[\phi_n](t;\mu,\lambda)-1-\Delta_2(t)\\
    &\geq \cJ_2[\phi](t;\mu,\lambda)-18e^{2\widetilde{\sigma}_+^2t^2}\abs{\phi_n^2(t)-\phi^2(t)}\nonumber\\
    &\ -24e^{\widetilde{\sigma}_+^2t^2/2}\abs{\phi_n(t)-\phi(t)}-3e^{2\widetilde{\sigma}_+^2t^2}\abs{\phi_n(2t)-\phi(2t)}-1-\Delta_2(t)\\
    &\geq \cJ_2[\phi](t;\mu,\lambda)-(63+81)\kappa e^{5\widetilde{\sigma}_+^2t^2/2}\sqrt{16\log(10/\delta)}-1\\
    &\geq \cJ_2[\phi_0](t;\mu,\lambda)-144\kappa e^{5\widetilde{\sigma}_+^2t^2/2}\sqrt{16\log(10/\delta)}-18e^{\widetilde{\sigma}_+^2t^2}\abs{\phi^2(t)-e^{2\im\theta t}\phi_0^2(t)}\nonumber\\
    &-24e^{\widetilde{\sigma}_+^2t^2/2}\abs{\phi(t)-e^{\im \theta t}\phi_0(t)}-3e^{2\widetilde{\sigma}_+^2t^2}\abs{\phi(2t)-e^{\im\theta t}\phi_0(2t)}-1\\
    &\geq \cJ_2[\phi_0](t;\mu,\lambda)-1-126\varepsilon e^{6L\sigma^2t^2/\log(en)}-144\kappa e^{5\widetilde{\sigma}_+^2t^2/2}\sqrt{16\log(10/\delta)/n}\\
    &=\cJ_2[e^{-\sigma^2t^2/2}](t;r,\lambda)-1-126\varepsilon e^{6L\sigma^2t^2/\log(en)}-144\kappa e^{5\widetilde{\sigma}_+^2t^2/2}\sqrt{16\log(10/\delta)/n}\\
    &\geq \cJ_2[e^{-\sigma^2t^2/2}](t;r,\lambda)-1-126\varepsilon e^{6L\sigma^2t^2/\log(en)}-144\kappa e^{10\sigma^2t^2}\sqrt{16\log(10/\delta)/n}
\end{align*}
also when $\log(en)\geq L$. Note that $\cJ_1[e^{-\sigma^2t^2/2}](t;r,\lambda)$ and $\cJ_2[e^{-\sigma^2t^2/2}](t;r,\lambda)$ are exactly the two terms on the left-hand side of the inequality in Lemma~\ref{lem:quartic_gap}. Thus, we apply Lemma~\ref{lem:quartic_gap} to conclude that when $\abs{rt}\leq \pi$,
\begin{align*}
    &\ \Big(\abs{\Upsilon_n(t;\mu,\lambda)}-1-\Delta_1(t)\Big)\vee\Big(\abs{\Upsilon_n^2(t;\mu,\lambda)-\Upsilon_n(2t;\mu,\lambda)}+\abs{\Upsilon_n(t;\mu,\lambda)}^2\nonumber-1-\Delta_2(t)\Big)\nonumber\\
    &\geq (rt/\pi)^4-126\varepsilon e^{6L\sigma^2t^2/\log(en)}-144\kappa e^{10\sigma^2t^2}\sqrt{16\log(10/\delta)/n}\\
    &\define S(t).
\end{align*}
Therefore, it suffices to show that for any $\varepsilon\in[0,\epsmax]$: either $\overline{r}(n,\epsilon,\sigma)\geq M\sigma/\sqrt{\log(en)}$, or for any $r\in[\overline{r}(n,\varepsilon,\sigma),M\sigma/\sqrt{\log(en)}]$, there exists $t\in\widetilde{\cT}$ with $\abs{rt}\leq \pi$ such that $S(t)>0$. We consider the following cases.
\paragraph{\textbf{Case I.}} When $\varepsilon\in[0,n^{-1/2}]$ and $\abs{r/\sigma}\in[ C_1n^{-1/8},2M/\sqrt{\log(en)}]$. Consider $t_1=1/(\kappa\widetilde{\sigma}_+)\in\widetilde{\cT}$. Then, on $\widetilde{\cE}\medcap \cE$, we have $\abs{rt_1}\leq 2M/\kappa$. Thus, when $\kappa\geq 2M/\pi$, we will have $\abs{rt_1}\leq \pi$. When this holds, we have
\begin{align*}
    S(t_1)&\geq C_1^4n^{-1/2}(\sigma^2/\widetilde{\sigma}_+^2)^2-126\varepsilon e^{6L\kappa^{-2}(\sigma^2/\widetilde{\sigma}_+^2)/\log(en)}-144\kappa e^{10\kappa^{-2}(\sigma^2/\widetilde{\sigma}_+^2)}\sqrt{16\log(10/\delta)/n}\\
    &\geq n^{-1/2}[C_1^4/16-126e^{6/\kappa^2}-144\kappa e^{10/\kappa^2}\sqrt{16\log(10/\delta)}],
\end{align*}
as long as $\log(en)\geq L$. Therefore, as long as, additionally, $C_1^4/16>126e^{6/\kappa^2}+144\kappa e^{10/\kappa^2}\sqrt{16\log(10/\delta)}$, we will have $S(t_1)>0$, and thus $\mu=\theta+r\notin\CIhat$.

\paragraph{\textbf{Case II.}} When $\varepsilon\in[n^{-1/2},\epsmax]$ and $\abs{r/\sigma}\in[C_2\varepsilon^{1/4}/\sqrt{\log(en\varepsilon^2)},2M/\sqrt{\log(en)}]$. Consider $t_2=(\kappa \widetilde{\sigma}_+)^{-1}\floor{\sqrt{\log(en\varepsilon^2)}}$. We have $t_2\in\widetilde{\cT}$. Also, since $\abs{rt_2}\leq 2M/\kappa$, we have $\abs{rt_2}\leq \pi$ as long as $\kappa\geq 2M/\pi$. Thus,
\begin{align*}
    S(t_2)&\geq C_2^4\floor{\sqrt{\log(en\varepsilon^2)}}/\sqrt{\log(en\varepsilon^2)}(\sigma^2/\widetilde{\sigma}_+^2)^2\varepsilon-126\varepsilon e^{6L\kappa^{-2}(\sigma^2/\widetilde{\sigma}_+^2)\log(en\varepsilon^2)/\log(en)}\nonumber\\
    &\ -144\kappa e^{10\kappa^{-2}(\sigma^2/\widetilde{\sigma}_+^2)\log(en\varepsilon^2)}\sqrt{16\log(10/\delta)/n}\\
    &\geq \varepsilon[C_2^4/32-126e^{6L/\kappa^2}]-144\kappa e^{10/\kappa^2}\sqrt{16\log(10/\delta)}n^{10/\kappa^2-1/2}\varepsilon^{20/\kappa^2}.
\end{align*}
If, in addition, $\kappa^2\geq 20$ and thus $10/\kappa^2-1/2\leq 0$, we have $n^{10/\kappa^2-1/2}\leq (1/\varepsilon^2)^{10/\kappa^2-1/2}$, which leads to
\begin{align*}
    S(t_2)&\geq \varepsilon[C_2^4/32-126e^{6L/\kappa^2}-144\kappa e^{10/\kappa^2}\sqrt{8\log(10/\delta)}].
\end{align*}
Therefore, we can pick $C_2^4/32>126e^{6L/\kappa^2}+144\kappa e^{10/\kappa^2}\sqrt{16\log(10/\delta)}$ to make $S(t_2)>0$ and thus $\mu=\theta+r\notin \CIhat$.

Aggregating the above conditions, and also taking into account the conditions needed for Proposition~\ref{prop:unknown_var_coverage}, we need
\begin{align*}
    \kappa\geq \sqrt{5}\vee (1/\sqrt{16\log(10)})\vee (2M/\pi)=\sqrt{5}\vee (2M/\pi),\quad N\geq e^{L-1}\vee (2N_0),
\end{align*}
which can be simultaneously satisfied by taking
\begin{align*}
    \kappa =\sqrt{5}\vee (2M/\pi),\quad N=\ceil{e^{L-1}}\vee (2N_0).
\end{align*}
With the above choices, we accordingly set
\begin{align*}
    C=C_1=C_2=4\Big(126e^{6L/\kappa^2}+144\kappa e^{10/\kappa^2}\sqrt{16\log(10/\delta)}\Big)^{1/4}
\end{align*}
to conclude the proof.
\end{proof}

\begin{proof}[Proof of Theorem~\ref{thm:unknown_var}]
The conclusions in both Proposition~\ref{prop:unknown_var_coverage} and Proposition~\ref{prop:unknown_var_length} hold on the same joint event $\widetilde{\cE}\medcap \cE$ that happens with at least $1-\delta$ probability. The desired Theorem~\ref{thm:unknown_var} is thus proved.
\end{proof}

\section{Deferred Proofs of Lower Bounds}\label{sec:proof_lb}
\subsection{Proof of the CI-to-Test Reduction}\label{sec:proof_CI_to_test}
\begin{Proposition}\label{prop:CI_to_test_known_var}
Given $\epsmax\in(0,1/2)$ and $\sigma^2>0$, if a confidence interval $\CIhat$ satisfies
\begin{align*}
    \inf_{\theta}\inf_{\varepsilon\in[0,\epsmax],\, Q}\ P_{\theta,\sigma,\varepsilon,Q}^n(\theta\in\CIhat)\geq 1-\delta
\end{align*}
and
\begin{align*}
    \inf_{\theta}\inf_{\varepsilon\in[0,\epsmax],\, Q}\ P_{\theta,\sigma,\varepsilon,Q}^n(\abs{\CIhat}\leq r/2)\geq 1-\delta,
\end{align*}
for some function $r=r(n,\sigma,\varepsilon)>0$, then its induced test $T(X_{1:n})=\indi\{0\notin \CIhat\}$ satisfies
\begin{align*}
    \sup_{P\in H_0(\sigma,\epsmax)} P^n(T)+\sup_{P\in H_1(r,\sigma,\varepsilon)}P^n(1-T)\leq 3\delta.
\end{align*}
\end{Proposition}
\begin{proof}
Let $\CIhat$ be such a confidence interval satisfying
\begin{align*}
    \inf_{\theta}\,\inf_{\varepsilon\in[0,\epsmax],\, Q}\ P_{\theta,\sigma,\varepsilon,Q}^n(\theta\in\CIhat)\geq 1-\delta
\end{align*}
and
\begin{align*}
    \inf_{\theta}\,\inf_{\varepsilon\in[0,\epsmax],\, Q}\ P_{\theta,\sigma,\varepsilon,Q}^n(\abs{\CIhat}\leq r/2)\geq 1-\delta.
\end{align*}
Then, by the above assumptions, we conclude for $T(X_{1:n})=\indi\{0\notin \CIhat\}$ that
\begin{align*}
    \sup_{P\in H_0(\sigma,\varepsilon)} P^n(T)&=\sup_{\varepsilon\in[0,\epsmax],\, Q}\ P_{0,\sigma,\varepsilon,Q}^n(0\notin\CIhat)\\
    &\leq \sup_{\theta}\,\sup_{\varepsilon\in[0,\epsmax],\, Q}\ P_{\theta,\sigma,\varepsilon,Q}^n(\theta\notin\CIhat)\leq \delta.
\end{align*}
Meanwhile, we also have that
\begin{align*}
    \sup_{P\in H_1(r,\sigma,\varepsilon)}P^n(1-T)
    &=\sup_{\varepsilon\in[0,\epsmax],\, Q}\ P_{r,\sigma,\varepsilon,Q}^n(0\in\CIhat)\\
    &=\sup_{\varepsilon\in[0,\epsmax],\, Q}\ \big\{P_{r,\sigma,\varepsilon,Q}^n(0\in\CIhat,\ r\in \CIhat)+P_{r,\sigma,\varepsilon,Q}^n(0\in\CIhat,\ r\notin \CIhat)\Big\}.
\end{align*}
Since $\{0\in\CIhat,\, r\in\CIhat\}$ implies $\abs{\CIhat}> r/2$, we continue to write
\begin{align*}
    \sup_{P\in H_1(r,\sigma,\varepsilon)}P^n(1-T)
    &\leq \sup_{\varepsilon\in[0,\epsmax],\, Q}\ \big\{P_{r,\sigma,\varepsilon,Q}^n(\abs{\CIhat}>r/2)+P_{r,\sigma,\varepsilon,Q}^n(r\notin \CIhat)\Big\}\\
    &\leq  \sup_{\theta}\,\sup_{\varepsilon\in[0,\epsmax],\, Q}\ P_{\theta,\sigma,\varepsilon,Q}^n(\abs{\CIhat}> r/2)+\sup_{\theta}\,\sup_{\varepsilon\in[0,\epsmax],\, Q}\ P_{\theta,\sigma,\varepsilon,Q}^n(\theta\notin\CIhat)\\
    &\leq 2\delta.
\end{align*}
Thus, the sum of type-I and type-II is no larger than $3\delta$, which concludes the proof.
\end{proof}

\begin{Proposition}\label{prop:CI_to_test_unknown_var}
Given $\epsmax\in(0,1/2)$, if a confidence interval $\CIhat$ satisfies
\begin{align*}
    \inf_{\theta,\sigma^2}\,\inf_{\varepsilon\in[0,\epsmax],\, Q}\ P_{\theta,\sigma,\varepsilon,Q}^n(\theta\in\CIhat)\geq 1-\delta
\end{align*}
and
\begin{align*}
    \inf_{\theta,\sigma^2}\,\inf_{\varepsilon\in[0,\epsmax],\, Q}\ P_{\theta,\sigma,\varepsilon,Q}^n(\abs{\CIhat}\leq r/2)\geq 1-\delta,
\end{align*}
for some function $r=r(n,\sigma,\varepsilon)>0$, then its induced test $T(X_{1:n})=\indi\{0\notin \CIhat\}$ satisfies
\begin{align*}
    \sup_{P\in H_0(\epsmax)} P^n(T)+\sup_{P\in H_1(r,\sigma,\varepsilon)}P^n(1-T)\leq 3\delta.
\end{align*}
\end{Proposition}
\begin{proof}
The proof of Proposition~\ref{prop:CI_to_test_unknown_var} is almost a verbatim copy of the previous one for Proposition~\ref{prop:CI_to_test_known_var}. We only need to adapt the translation between the coverage of $\CIhat$ and the type-I of its induced test $T=\indi\{0\notin\CIhat\}$ to the unknown-variance setting:
\begin{align*}
    \sup_{P\in H_0(\varepsilon)} P^n(T)&=\sup_{\sigma^2}\sup_{\varepsilon\in[0,\epsmax],\, Q}\ P_{0,\sigma,\varepsilon,Q}^n(0\notin\CIhat).
\end{align*}
\end{proof}
Note that we do not force $\epsmax\leq 1/3$ in the above proposition as we did in the upper bound proof.

\subsection{Proof of Proposition~\ref{prop:known_var_matching}}\label{sec:proof_known_var_matching}
Let $K$ be a positive multiple of 4, and $\tau>0$. Consider
\begin{align*}
    M_0(t)&=\litb{e^{\im t/(2\tau)}-e^{-\im t/(2\tau)}}^K=\sum_{k=0}^K(-1)^{K-k}\binom{K}{k}e^{\im t(k-K/2)/\tau},
\end{align*}
which is $O(t^K)$ as $t\to 0$. For $a,b\in(0,1)$, we iteratively define
\begin{align*}
    M_1(t)=e^{\im at/\tau}M_0(t)=O(t^K).
\end{align*}
and
\begin{align*}
    M_2(t)&=\int_0^t M_1(s)\dif s=O(t^{K+1})\\
    &=\sum_{k=0}^K(-1)^{K-k}\binom{K}{k}\frac{\tau}{k-K/2+a}e^{\im t(k-K/2+a)/\tau}-\sum_{k=0}^K(-1)^{K-k}\binom{K}{k}\frac{\tau}{k-K/2+a},
\end{align*}
and
\begin{align*}
    M_3(t)=e^{\im bt/\tau}M_2(t)=O(t^{K+1}),
\end{align*}
and
\begin{align*}
    M_4(t)&=\int_0^t M_3(s)\dif s=O(t^{K+2})\\
    &=\sum_{k=0}^K(-1)^{K-k}\binom{K}{k}\frac{\tau^2}{(k-K/2+a)(k-K/2+a+b)}e^{\im t(k-K/2+a+b)/\tau}\nonumber\\
    &\quad -\sum_{k=0}^K(-1)^{K-k}\binom{K}{k}\frac{\tau^2}{(k-K/2+a)b}e^{\im t b/\tau}\nonumber\\
    &\quad -\sum_{k=0}^K(-1)^{K-k}\binom{K}{k}\frac{\tau^2}{(k-K/2+a)(k-K/2+a+b)}\nonumber\\
    &\quad +\sum_{k=0}^K(-1)^{K-k}\binom{K}{k}\frac{\tau^2}{(k-K/2+a)b}\\
    &=\sum_{k=0}^K(-1)^{K-k}\binom{K}{k}\frac{\tau^2}{(k-K/2+a)(k-K/2+a+b)}e^{\im t(k-K/2+a+b)/\tau}\nonumber\\
    &\quad -\sum_{k=0}^K(-1)^{K-k}\binom{K}{k}\frac{\tau^2}{(k-K/2+a)b}e^{\im t b/\tau}\nonumber\\
    &\quad +\sum_{k=0}^K(-1)^{K-k}\binom{K}{k}\frac{\tau^2}{(k-K/2+a+b)b}.
\end{align*}
From $M_0(t)=O(t^K)$, $M_0(0)=0$ and since
\begin{align*}
    M_4(t)=\int_0^te^{\im bs/\tau}\Big(\int_0^s e^{\im a s'/\tau} M_0(s')\dif s'\Big)\dif s,
\end{align*}
it is straightforward that:
\begin{Claim}
We have
\begin{align*}
    M_4(0)=0,\quad M_4(t)=O(t^{K+2}).
\end{align*}
\end{Claim}
\begin{proof}
This is straightforward from the integral definition.
\end{proof}
Rewrite
\begin{align*}
    M_4(t)&=\sum_{2\mid k, k\neq K/2}\alpha_k e^{\im t(k-K/2+a+b)/\tau}+\alpha_{K/2}e^{\im t(a+b)/\tau}+\beta_+\nonumber\\
    &\quad - \Big(\sum_{2\nmid k}\alpha_k e^{\im t(k-K/2+a+b)/\tau}+\beta_-e^{\im tb/\tau}\Big),
\end{align*}
where
\begin{align*}
    \alpha_k=\binom{K}{k}\frac{\tau^2}{(k-K/2+a)(k-K/2+a+b)},
\end{align*}
and
\begin{align*}
    \beta_+=\sum_{k=0}^K(-1)^{K-k}\binom{K}{k}\frac{\tau^2}{(k-K/2+a+b)b},\quad \beta_-=\sum_{k=0}^K(-1)^{K-k}\binom{K}{k}\frac{\tau^2}{(k-K/2+a)b}.
\end{align*}
We claim that:
\begin{Claim}\label{clm:relative_magnitude}
When $a\in(0,1/8]$, $b\in(0,1/8]$, we have that $\beta_\pm$ and all $\alpha_{0:K}$ are positive, and
\begin{align*}
    \frac{\beta_+}{\alpha_{K/2}}\geq \frac{\pi a}{4b},\quad \frac{\beta_+}{\sum_{2\mid k,\, k\neq K/2}\alpha_k}\geq \frac{3}{8\pi b(a+b)},\quad \frac{\alpha_{K/2}}{\sum_{2\mid k,\, k\neq K/2}\alpha_k}\geq \frac{3}{2\pi^2a(a+b)},\quad \frac{\beta_-}{\sum_{2\nmid k}\alpha_k}\geq \frac{3}{8\pi ab}.
\end{align*}
\end{Claim}
\begin{proof}
The positivity of $\alpha_{0:K}$ is because $(k-K/2+a)$ and $(k-K/2+a+b)$ always take the same sign. 
We also observe that $\alpha_{K/2}=\binom{K}{K/2}\frac{\tau^2}{a(a+b)}$.
By Lemma~\ref{lem:control_alternation}, we have
\begin{align*}
    \beta_+\geq \frac{\pi}{4}\binom{K}{K/2}\frac{\tau^2}{b(a+b)},\quad \beta_-\geq \frac{\pi}{4}\binom{K}{K/2}\frac{\tau^2}{ab},
\end{align*}
and thus the positivity of $\beta_\pm$ and the first inequality are proved. Meanwhile, since 
\begin{align*}
    \frac{1}{(j+a)(j+a+b)}\leq \frac{2}{j^2},\quad j\in\bbZ\backslash\{0\},
\end{align*}
we have
\begin{align*}
    \sum_{2\nmid k}\alpha_k\vee \sum_{2\mid k,\,k\neq K/2}\alpha_k\leq\sum_{k\neq K/2}\alpha_k &\leq \binom{K}{K/2}\sum_{k\neq K/2}\frac{\tau^2}{(k-K/2+a)(k-K/2+a+b)}\\
    &\leq \binom{K}{K/2}\sum_{j\in\bbZ\backslash\{ 0\}}\frac{\tau^2}{(j+a)(j+a+b)}\\
    &\leq \binom{K}{K/2}\sum_{j\in\bbZ\backslash\{ 0\}}\frac{2\tau^2}{j^2}\\
    &=\frac{2\pi^2}{3}\binom{K}{K/2}\tau^2.
\end{align*}
Combining this with the lower bounds on $\beta_+$ and $\beta_-$ completes the proof.
\end{proof}
Let 
\begin{align*}
    T=\beta_++\sum_{2\mid k}\alpha_k=\beta_-+\sum_{2\nmid k}\alpha_k>0,
\end{align*}
where the second identity is due to $M_4(0)=0$. Finally, we define
\begin{align*}
    \overline{M}_4(t)=\frac{1}{T}M_4(t).
\end{align*}
Then, $\overline{M}_4(t)$ is the difference between the characteristic functions of two discrete probability distributions:
\begin{align*}
    \nu_0=\frac{\beta_+}{T}\delta_0+\sum_{2\mid k}\alpha_k \delta_{(k-K/2+a+b)/\tau},\quad \nu_1=\frac{\beta_-}{T}\delta_{b/\tau}+\sum_{2\nmid k}\alpha_k \delta_{(k-K/2+a+b)/\tau}.
\end{align*}
From now on, we set $a=\varepsilon^{1/2}/C_1$ and $b=a/C_2$ for some constants $C_1$ and $C_2$ that depend only on $\epsmax$. We shall show that by setting $C_1$ and $C_2$ appropriately large, the resulting $\nu_0$ and $\nu_1$ will satisfy Proposition~\ref{prop:known_var_matching}.
\begin{proof}[Proof of Proposition~\ref{prop:known_var_matching}] Setting $a=\varepsilon^{1/2}/C_1$ and $b=a/C_2$. The second part of Proposition~\ref{prop:known_var_matching} follows from the fact that
\begin{align*}
    \overline{M}_4(t)=\frac{1}{T}M_4(t)=O(t^{K+2}).
\end{align*}
To show the first conclusion, we apply Lemma~\ref{clm:relative_magnitude}. We have
\begin{align*}
    1-\nu_1(\{b/\tau\})=1-\frac{\beta_-}{T}=\frac{\sum_{2\nmid k}\alpha_k}{\beta_-+\sum_{2\nmid k}\alpha_k}\leq \frac{1}{3/(8\pi ab)+1}=\frac{\varepsilon}{3C_1^2C_2/(8\pi)+\varepsilon}.
\end{align*}
Thus, it suffices to let $C_1^2 C_2\geq 8\pi/3$. Similarly,
\begin{align*}
    1-\nu_0(\{0\})=\frac{\sum_{2\mid k}\alpha_k}{\beta_++\sum_{2\mid k}\alpha_k}&\leq \frac{4b/(\pi a)+8\pi b(a+b)/3}{1+4b/(\pi a)+8\pi b(a+b)/3}\leq \frac{4b/(\pi a)}{1+4b/(\pi a)}\leq \frac{4b}{\pi a}=\frac{4}{\pi C_2}.
\end{align*}
Thus, it suffices to let $C_2\geq 4/(\pi\epsmax)$. Aggregating the above conditions, we need to ensure that
\begin{align*}
    C_1^2C_2\geq 8\pi/3,\quad C_2\geq 4/(\pi\epsmax),
\end{align*}
and also for making $a,b\leq 1/8$: 
\begin{align*}
    2^{-1/2}/C_1\leq 1/8,\quad 2^{-1/2}/(C_1C_2)\leq 1/8,
\end{align*}
all of which can be simultaneously satisfied by setting
\begin{align*}
    C_2=2/\epsmax,\quad C_1=8.
\end{align*}
Thus, the first part of Proposition~\ref{prop:known_var_matching} is now proved. Fixing these choices of $a$ and $b$, we now prove the third part of Proposition~\ref{prop:known_var_matching}. Note that both $\nu_0$ and $\nu_1$ are supported on $[-K/\tau,K/\tau]$, and that the preceding proof has shown that the masses of $\nu_0$ and $\nu_1$ are highly concentrated on $0$ and $b/\tau$, respectively. Therefore, denoting $c=1/(C_1C_2)=\epsmax/16<1/32$, we have $b=c\varepsilon^{1/2}$, and
\begin{align*}
    \abs{\E_{\nu_1}[X^k]}&\leq \nu_1(\{b/\tau\})\cdot\abs{b/\tau}^k+(1-\nu_1(\{b/\tau\}))\cdot\abs{K/\tau}^k\\
    &\leq c^k \varepsilon^{k/2}/\tau^k+ \varepsilon (K/\tau)^k\\
    &\leq 2\varepsilon (K/\tau)^k
\end{align*}
as long as $k\geq 2$ and $K/\tau\geq 1$. To control the moments of $\nu_0$, we need one more step. Note that by Claim~\ref{clm:relative_magnitude},
\begin{align*}
    1-\nu_0(\{0\})-\nu_0(\{b/\tau\})=\frac{\sum_{2\mid k,\, k\neq K/2}\alpha_k}{T}&\leq \frac{1}{1+3/(8\pi b(a+b))}\\
    &= \frac{8\pi}{3}\varepsilon c(c+1/C_1)\\
    &=\frac{8\pi}{3}\cdot \frac{\epsmax}{16}(\frac{\epsmax}{16}+1/8)\varepsilon\leq \varepsilon.
\end{align*}
Therefore,
\begin{align*}
    \abs{\E_{\nu_0}[X^k]}&\leq \nu_0(\{b/\tau\})\cdot 0^k+\nu_0(\{b/\tau\}) \abs{b/\tau}^k+ (1-\nu_0(\{0\})-\nu_0(\{b/\tau\}))\cdot\abs{K/\tau}^k\\
    &\leq c^k\varepsilon^{k/2}/\tau^k+\varepsilon\abs{K/\tau}^k\\
    &\leq 2\varepsilon (K/\tau)^k
\end{align*}
whenever $k\geq 2$ and $K/\tau\geq 1$. The desired conclusions are then proved.  
\end{proof}

\subsection{Proof of Theorem~\ref{thm:unknown_var_lower}}\label{sec:proof_unknown_var_lower}
\begin{proof}[Proof of Theorem~\ref{thm:unknown_var_lower}]
For $\epsmax\in(0,1/3]$, we first construct
\begin{align*}
    \mu_0=(1-\epsmax)\delta_0+\epsmax\big(p\delta_{y_1}+(1-p)\delta_{y_2}\big),\quad \mu_1=\cN(1,s^2),
\end{align*}
where
\begin{align*}
     y_1=-\frac{1}{\epsmax}-\frac{A(\epsmax)}{\epsmax},\quad y_2=-\frac{1}{\epsmax}+\frac{A(\epsmax)}{\epsmax},\quad A(\epsmax)=\sqrt{\frac{2\epsmax^2+9\epsmax+5}{3\epsmax+1}},
\end{align*}
and
\begin{align*}
    s^2=\frac{2-\epsmax-\epsmax^2}{\epsmax(1+3\epsmax)},\quad p=\frac{1/\epsmax-y_2}{y_1-y_2}=\frac{1}{2}-\sqrt{\frac{3\epsmax+1}{2\epsmax^2+9\epsmax+5}}.
\end{align*}
It can be verified that $\mu_0$ and $\mu_1$ match the first three moments. Moreover, 
\begin{align*}
    s^2\leq 2/\epsmax,\quad \max\{\abs{y_1},\abs{y_2}\}\leq 4/\epsmax.
\end{align*}
Now we define $\nu_0$ and $\nu_1$ as the scaled version of $\mu_0$ and $\mu_1$, that is,
\begin{align*}
    \nu_0=(1-\epsmax)\delta_0+\epsmax\big(p\delta_{ry_1}+(1-p)\delta_{ry_2}\big),\quad \nu_1=\cN(r,r^2s^2)
\end{align*}
for $r=c'n^{-1/8}$ for some absolute constant $c'$ that depends on $\delta$ and $\epsmax$. Still, $\nu_0$ and $\nu_1$ match their first three moments. $\nu_0$ is supported within $[-1,1]\cdot 4c'n^{-1/8}/\epsmax$ and is thus $(32c'^2n^{-1/4}/\epsmax^2)$-sub-Gaussian; $r^2s^2\leq 4c'^2 n^{-1/4}/\epsmax$, and thus $\nu_1$ is also $(32c'^2n^{-1/4}/\epsmax^2)$-sub-Gaussian. Therefore, setting $P_i=\nu_i*\cN(0,1)$ for $i=1,2$, we conclude by Lemma~\ref{lem:chi2_subgaussian} that
\begin{align*}
    \chi^2(P_1,P_0)\leq \frac{16 (32c'^2n^{-1/4}/\epsmax^2)^4}{\sqrt{3}(1-32c'^2n^{-1/4}/\epsmax^2)}\leq 2^{24}c'^2/\epsmax^2\cdot n^{-1},
\end{align*}
as long as we set $c'^2\leq \epsmax^2/128$. By further setting $c'$ small with respect to $\delta$, we obtain a lower bound of
\begin{align*}
    r/\sigma\geq \frac{c'n^{-1/8}}{\sqrt{1+r^2s^2}}\geq c_{\delta,\epsmax} n^{-1/8},
\end{align*}
which proves the desired conclusion.
\end{proof}

\subsection{Proof of Theorem~\ref{thm:unknown_var_lower_above_third}}\label{sec:proof_unknown_var_lower_above_third}
\begin{proof}[Proof of Theorem~\ref{thm:unknown_var_lower_above_third}]
We start by considering when we can find $s^2>0$ and a probability distribution $Q$ such that
\begin{align*}
    \mu_0=(1-\epsmax)\delta_0+\epsmax Q\quad \textnormal{and}\quad \mu_1=N(1,s^2)
\end{align*}
match the first $K$ moments. This is equivalent to determining the existence of $Q$ such that
\begin{align*}
    \E_Q[X^k]=(1/\epsmax)\E_{\cN(1,s^2)}[X^k]\define m_k,\quad k=0,1,2,\cdots,K,
\end{align*}
which is precisely a truncated Hamburger moment problem. We choose to investigate only odd $K$ since for odd $K$, the existence of $K$ is easier to determine via the positive semidefiniteness of the Hankel matrices. Define $H_k=[m_{i+j}]_{i,j=0}^k\in\bbR^{(k+1)\times(k+1)}$ as the order-$k$ Hankel matrix that any candidate $Q$ must obey. We have
\begin{align*}
    \det H_0&=1,\\
    \det H_1&=\epsmax^{-2}(\epsmax s^2+\epsmax-1),\\
    \det H_2&=\epsmax^{-3}s^2\big[(3\epsmax-1)s^4+\epsmax-1\big],\\
    \det H_3&=2\epsmax^{-4}s^6\big[(9\epsmax-3)s^6+(9\epsmax-9)s^4+(3-3\epsmax)s^2+\epsmax-1\big],\\
    \det H_4&=12\epsmax^{-5}s^{12}\big[(45\epsmax-21)s^8+(30\epsmax-30)s^4+(8-8\epsmax)s^2+\epsmax-1\big],\\
    \det H_5&=288\epsmax^{-6}s^{20}\big[(225\epsmax-105)s^{10}+(225\epsmax-225)s^8+(150-150\epsmax)s^6\nonumber\\
    &\ +(90\epsmax-90)s^4+(15-15\epsmax)s^2+\epsmax-1\big],\\
    \det H_6&=34560\epsmax^{-7}s^{30}\big[(1575\epsmax-855)s^{12}+(1575\epsmax-1575)s^8+(840-840\epsmax)s^6\nonumber\\
    &\ +(225\epsmax-225)s^4+(24-24\epsmax)s^2+\epsmax-1\big].
\end{align*}
We observe that each determinant is a polynomial of $s^2$ with $\epsmax$-dependent coefficients. When $\epsmax\in(1/3,1/2)$, the leading coefficients of $\{\det H_k\}_{k=0}^3$ are all positive. Therefore, by setting $s^2=s^2(\epsmax)$ large enough, we can make $\det H_k>0$ for all $k=0,1,\cdots,3$, indicating $H_3\succ 0$ and thus the existence of $Q$ such that $\mu_0$ and $\mu_1$ match up to the $2\times 3+1=7$-th moment. Furthermore, there exists such a $Q$ supported on at most $8$ atoms. Pick this $Q$, and let $B=B(\epsmax)$ be the maximum absolute value of its 8 supporting atoms. Following the proof of Theorem~\ref{thm:unknown_var_lower}, we set $r=c'n^{-1/16}$ for $c'$ depending on $\delta$ and $\epsmax$, and define $\nu_0$ (resp. $\nu_1$) as $\mu_0$ (resp. $\mu_1$) scaled by $r$. Then, $\nu_0$ is $c'^2 B(\epsmax)^2n^{-1/8}$-sub-Gaussian, and $\nu_1$ is $c'^2s^2(\epsmax)n^{-1/8}$-sub-Gaussian. By setting $c'$ small enough with respect to $\epsmax$, we can make both sub-Gaussian coefficients smaller than $c''n^{-1/8}<1$. Define $P_i=\nu_i*\cN(0,1)$ for $i=0,1$. By Lemma~\ref{lem:chi2_subgaussian}, we have
\begin{align*}
    \chi^2(P_1,P_0)\leq \frac{16(c'' n^{-1/8})^8}{\sqrt{7}(1-c'' n^{-1/8})}\lesssim_{\delta,\epsmax} n^{-1}.
\end{align*}
By further making $c'$ small enough with respect to $\delta$, we obtain a lower bound of
\begin{align*}
    r/\sigma\geq \frac{c'n^{-1/16}}{\sqrt{1+s^2(\epsmax)}}\gtrsim_{\delta,\epsmax} n^{-1/16}
\end{align*}
for all $\epsmax\in(1/3,1/2)$. Similarly, since the leading coefficients of $\{\det H_k\}_{k=0}^5$ are all positive when $\epsmax\in(7/15,1/2)$, we can find a finitely supported $Q$ that matches from the $0$-th to the $2\times 5+1=11$-th moments, leading to a lower bound of
\begin{align*}
    r/\sigma\gtrsim_{\delta,\epsmax} n^{-1/(2\times 11+2)}\asymp_{\delta,\epsmax} n^{-1/24},\quad \epsmax\in(7/15,1/2).
\end{align*}
We cannot push this argument to higher moments, since making the leading coefficient of $\det H_6$ positive already requires $\epsmax>19/35$, which is above $1/2$.
\end{proof}

\section{Supplementary Proofs}\label{sec:other_proof}

\subsection{Remaining Proofs for Section~\ref{sec:overview}}\label{sec:other_proof_sec2}
\subsubsection{Proofs of Lemma~\ref{lem:true_candidate_known_var} and \ref{lem:true_candidate_unknown_var}}\label{sec:proof_chf}
\begin{proof}[Proof of Lemma~\ref{lem:true_candidate_known_var}]
Recall the notations $\phi(t)=\E_P[\exp(\im tX)]$ and $\xi(t)=\E_Q[\exp(\im tX)]$. We have 
\begin{align*}
    \Upsilon(t;\mu)=2e^{-\im \mu t+\sigma^2 t^2/2}\phi(t)-1&=2e^{-\im \mu t+\sigma^2 t^2/2}[(1-\varepsilon)e^{\im \theta t-\sigma^2t^2/2}+\varepsilon e^{-\sigma^2t^2/2}\xi(t)]-1\\
    &=2(1-\varepsilon)e^{\im (\theta-\mu)t}+2\varepsilon \xi(t)-1.
\end{align*}
Under the condition $\mu=\theta$, this becomes $\Upsilon(t;\mu=\theta)=1-2\varepsilon+2\varepsilon \xi(t)$, which is the characteristic function of $(1-2\varepsilon)\delta_0+2\varepsilon Q$. 

On the other hand, under the condition that $\Upsilon(t;\mu)$ is the characteristic function of a probability distribution, if $\mu\neq \theta$, we have
\begin{align*}
    1\geq \sup_{t\in\bbR}\abs{\Upsilon(t;\mu)}\geq \abs{\Upsilon(\pi/(\theta-\mu);\mu)}=\abs{2\varepsilon+2\varepsilon \xi(t)-3}\geq 3-4\varepsilon>1.
\end{align*}
This leads to a contradiction, and thus we must have $\mu=\theta$ under the given condition.
\end{proof}

\begin{proof}[Proof of Lemma~\ref{lem:true_candidate_unknown_var}]
We have 
\begin{align*}
    \Upsilon(t;\mu,\lambda)=3e^{-\im \mu t+\lambda t^2/2}\phi(t)-2&=3e^{-\im \mu t+\lambda t^2/2}[(1-\varepsilon)e^{\im \theta t-\sigma^2t^2/2}+\varepsilon e^{-\sigma^2t^2/2}\xi(t)]-2\\
    &=3(1-\varepsilon)e^{\im (\theta-\mu)t+(\lambda-\sigma^2)t^2/2}+3\varepsilon e^{(\lambda-\sigma^2) t^2/2} \xi(t)-2.
\end{align*}
Under the condition that $\mu=\theta$ and $\lambda=\sigma^2$, this becomes $\Upsilon(t;\mu,\lambda)=(1-3\varepsilon)+3\varepsilon \xi(t)$, which is the characteristic function of $(1-3\varepsilon)\delta_0+3\varepsilon Q$.

On the other hand, under the condition that $\Upsilon(t;\mu,\lambda)$ is the characteristic function of a true probability distribution, we must have
\begin{align*}
    1\geq \sup_{t\in\bbR}\ \abs{\Upsilon(t;\mu,\lambda)}.
\end{align*}
If $\lambda<\sigma^2$, then
\begin{align*}
    \sup_{t\in\bbR}\ \abs{\Upsilon(t;\mu,\lambda)}\geq \limsup_{t\to \infty}\ \abs{\Upsilon(t;\mu,\lambda)}=2>1,
\end{align*}
which leads to a contradiction. If $\lambda>\sigma^2$, then
\begin{align*}
    \sup_{t\in\bbR}\ \abs{\Upsilon(t;\mu,\lambda)}\geq \sup_{t\in\bbR}\ [3(1-\varepsilon)-3\varepsilon] e^{(\lambda-\sigma^2)t^2/2}-2=+\infty,
\end{align*}
which also leads to a contradiction. Therefore, we must have $\lambda=\sigma^2$. Given this, we have
\begin{align*}
    1\geq \abs{\Upsilon(t;\mu,\lambda=\sigma^2)}= \abs{3(1-\varepsilon)e^{\im (\theta-\mu)t}+3\varepsilon \xi(t)-2},\quad \forall\, t\in\bbR.
\end{align*}
If $\mu\neq \theta$, we pick $t=\pi/(\theta-\mu)$ to see that
\begin{align*}
    \abs{\Upsilon(t=\pi/(\mu-\theta);\mu,\lambda=\sigma^2)}=\abs{3\varepsilon+3\varepsilon\xi(t)-5}\geq5-6\varepsilon\geq 3
\end{align*}
since $\varepsilon\leq 1/3$, which again leads to a contradiction. Therefore, under the given condition, we must have $\mu=\theta$ and $\lambda=\sigma^2$.
\end{proof}

\subsubsection{Proof of Lemma~\ref{lem:pilot_unknown_var} Part I: Pilot Variance Estimator}
We first prove the existence of such a pilot variance estimator.
\begin{Lemma}\label{lem:pilot_unknown_var_varpart}
There exists a variance estimator $\widetilde{\sigma}^2$ such that: for any fixed $\delta\in(0,1)$ and $\epsmax\in(0,1/2)$, there are absolute constants $L$ and $N_1$ that depend only on $\delta$ and $\epsmax$, such that
\begin{align*}
    \inf_{\theta,\sigma^2}\, \inf_{ Q}\,P_{\theta,\sigma,\epsmax,Q}^n\Big(\frac{\abs{\widetilde{\sigma}^2-\sigma^2}}{\sigma^2}\leq \frac{L}{\log(en)}\Big)\geq 1-\delta
\end{align*}
whenever $n\geq N_1$.
\end{Lemma}
To prove Lemma~\ref{lem:pilot_unknown_var_varpart}, we need a preliminary result about a crude finite-sample approximation to $\sigma^2$ up to leading constants. The idea is to inspect the blockwise variation, as shown in \cite{kotekal2025sparsity} and \cite{kotekal2025optimal}.
\begin{Lemma}[\cite{kotekal2025optimal} Proposition 3.1]\label{lem:constant_var_est}
Consider the following variance estimator:
\begin{align}\label{eq:constant_var_est}
    \breve{\sigma}^2(X_{1:n})\define \min_{1\leq r\leq m}\frac{1}{\ell-1}\sum_{j\in E_r}(X_j-\overline{X}_{E_r})^2,
\end{align}
where $\{E_1,E_2,\cdots,E_m\}$ are independent random size-$\ell$ subsets of $[n]$ sampled uniformly at random, and $\overline{X}_{E_r}\define \sum_{j\in E_r}X_j/\ell$ denotes the group mean. Fix any $\delta\in(0,1)$, there exist absolute constants $C_1$, $C_2$, and $C_3$ that depend only on $\delta$, such that for $m=\ceil{n^{C_1}}$ and $2<\ell=\ceil{C_2\log(n)}$, the estimator \eqref{eq:constant_var_est} satisfies
\begin{align*}
      \inf_{\theta,\sigma^2}\,\inf_{\varepsilon\in[0,1/2),Q}\, P_{\theta,\sigma,\varepsilon,Q}^n\Big(C_3^{-1}\leq \frac{\breve{\sigma}^2}{\sigma^2}\leq C_3\Big)\geq 1-\delta.
\end{align*}   
\end{Lemma}
\begin{proof}[Proof of Lemma~\ref{lem:pilot_unknown_var_varpart}]
We perform a sample-splitting strategy in which we use the first half of the data to obtain the crude estimate $\breve{\sigma}^2$ defined in Lemma~\ref{lem:constant_var_est}. We then combine this estimate with the second half of the samples in a Fourier-based pipeline. Recall that for any distribution in $\{P_{\theta,\sigma,\epsmax,Q}:\, Q\}$, its characteristic function is given by
\begin{align*}
    \phi(t)=\Big((1-\epsmax) e^{\im \theta t}+\epsmax \xi(t)\Big)\cdot e^{-\sigma^2t^2/2},
\end{align*}
where $\xi(t)=\E_Q[\exp(\im t X)]$ denotes the characteristic function of the adversary. Thus, we have
\begin{align*}
    -\frac{2\log\abs{\phi(t)}}{t^2}=\sigma^2-\frac{2\log\abs{(1-\epsmax)e^{\im \theta t}/3+\epsmax\xi(t)}}{t^2}\in \Big[\sigma^2,\sigma^2+2\log\Big(\frac{1}{1-2\epsmax}\Big)/t^2\Big],\quad \forall\, t\neq 0,
\end{align*}
where we have used $\abs{\xi(t)}\leq 1$ so that
\begin{align*}
    1\geq \abs{(1-\epsmax)e^{\im \theta t}+\epsmax\xi(t)}\geq 1-2\epsmax\Longrightarrow \abs{\phi(t)}\in e^{-\sigma^2t^2/2}\cdot [1-2\epsmax,\,1]
\end{align*}
With this intuition, we define the following estimator:
\begin{align*}
    \widetilde{\sigma}^2=\frac{-2\log\abs{\phi_n(t_*)}}{t_*^2},\quad t_*=c_0 \frac{\sqrt{\log(en)}}{\breve{\sigma}},
\end{align*}
where $\phi_n$ is the empirical characteristic function evaluated on the second half of the data, $c_0$ is an absolute $(\delta,\epsmax)$-dependent constant to be specified later. We shall show that the above estimator $\widetilde{\sigma}^2$ can achieve the claimed performance. To begin with, we apply the triangle inequality to see that
we have
\begin{align*}
    \abs{\widetilde{\sigma}^2-\sigma^2}&\leq \abs*{-\frac{2\log\abs{\phi_n(t_*)}}{t_*^2}+\frac{2\log\abs{\phi(t_*)}}{t_*^2}}+\abs*{-\frac{2\log\abs{\phi(t_*)}}{t_*^2}-\sigma^2}\\
    &\leq \frac{2}{t_*^2}\frac{\abs{\phi_n(t_*)-\phi(t_*)}}{\abs{\phi_n(t_*)}\wedge \abs{\phi(t)}}+\frac{2\log(3)}{t_*^2},
\end{align*}
where we have used $\abs{\log(x)-\log(y)}\leq \abs{x-y}/(x\wedge y)$.
By Lemma~\ref{lem:constant_var_est}, when $n\geq N_0$, we have $\breve{\sigma}^2/\sigma^2\in[1/C_0, C_0]$ with probability at least $1-\delta/2$ (w.r.t. the holdout data), where both $N_0$ and $C_0$ are $\delta$-dependent constants. Regarding the second (training) half of the samples, we apply Hoeffding's inequality to see that
\begin{align*}
    \abs{\phi_n(t_*)-\phi(t_*)}\leq A \frac{1}{\sqrt{n}},
\end{align*}
with probability at least $1-\delta/2$, where $A$ is an absolute constant depending only on $\delta$. Therefore, on the union event that holds with probability at least $1-\delta$, we have
\begin{align*}
    \frac{\abs{\widetilde{\sigma}^2-\sigma^2}}{\sigma^2}&\leq \frac{2\breve{\sigma}^2}{c_0^2\sigma^2\log(en)}\cdot \frac{A n^{-1/2}}{((1-2\epsmax)e^{-\sigma^2 t_*^2/2}-An^{-1/2})_+}+  \frac{2\log(3)\breve{\sigma}^2}{c_0^2\sigma^2\log(en)}\\
    &\leq \frac{2C_0^2}{c_0^2\log(en)}\Big\{\frac{An^{-1/2}}{((1-2\epsmax)(en)^{-c_0^2/C_0^2}-An^{-1/2})_+}+\log(3)\Big\}.
\end{align*}
Thus, we may choose $c_0\leq C_0/2$, and then $(en)^{-c_0^2/C_0^2}/3\geq 2An^{-1/2}$ holds whenever $n\geq N_1$, where $N_1$ is an absolute constant that depends only on $\delta$ and $\epsmax$. With this choice, the term inside the bracket is at most $1+\log(3)$, which proves the desired conclusion.
\end{proof}

\subsubsection{Proof of Lemma~\ref{lem:pilot_unknown_var} Part II: Pilot Mean Estimator}
We then prove the existence of such a pilot mean estimator claimed in Lemma~\ref{lem:pilot_unknown_var}. 
\begin{Lemma}\label{lem:pilot_unknown_var_meanpart}
    There exists a mean estimator $\widetilde{\theta}$ such that: for any fixed $\delta\in(0,1)$ and $\epsmax\in(0,1/2)$, there are absolute constants $M$ and $N_2$ that depend only on $\delta$ and $\epsmax$, such that
\begin{align*}
    \inf_{\theta,\sigma^2}\, \inf_{ Q}\,P_{\theta,\sigma,\epsmax,Q}^n\Big(\abs{\widetilde{\theta}-\theta}\leq \frac{M}{\sqrt{\log(en)}}\Big)\geq 1-\delta
\end{align*}
whenever $n\geq N_2$.
\end{Lemma}
We need the following standard result that serves as a crude (constant-rate) approximation to the true mean.
\begin{Lemma}\label{lem:median}
For any fixed $\delta\in(0,1)$ and $\epsmax\in(0,1/2)$, the sample median $X_\textnormal{med}=\textnormal{Median}\{X_1,X_2,\cdots,X_n\}$ satisfies
\begin{align*}
    \inf_{\theta,\sigma^2}\,\inf_{Q}\, P_{\theta,\sigma,\epsmax,Q}^n(\abs{X_\textnormal{med}-\theta}\leq D\sigma)\geq 1-\delta,
\end{align*}
whenever $n\geq 8(1-2\epsmax)^{-2}\log(2/\delta)$, where
\begin{align*}
    D=\Phi^{-1}\Big(\frac{1/2+\sqrt{\log(2/\delta)/(2n)}}{1-\epsmax}\Big)
\end{align*}
is an absolute $(\delta,\epsmax)$-dependent constant and $\Phi^{-1}$ is the inverse CDF of $\cN(0,1)$.
\end{Lemma}

\begin{proof}
    By Dvoretzky-Kiefer-Wolfowitz (DKW) inequality, with probability exceeding $1-\delta$, we have
\begin{align*}
    \abs{F_n(t)-F(t)}\leq \sqrt{\frac{\log(2/\delta)}{2n}},\quad \forall\, t\in\bbR,
\end{align*}
regardless of the choice of the underlying distribution, where $F$ and $F_n$ are the population and empirical cumulative distribution functions, respectively. On this uniform concentration event,
\begin{align*}
    \sqrt{\frac{\log(2/\delta)}{2n}}\geq \abs{F_n(X_\mathrm{med})-F(X_\mathrm{med})}=\abs{F(X_\mathrm{med})-1/2}.
\end{align*}
Meanwhile, for Efron's model, it holds that
\begin{align*}
    (1-\epsmax)\Phi((X_\mathrm{med}-\theta)/\sigma)\leq F(X_\mathrm{med})\leq (1-\epsmax)\Phi((X_\mathrm{med}-\theta)/\sigma)+\epsmax.
\end{align*}
Combining the two chains of inequalities, we obtain
\begin{align*}
    \frac{1/2-\epsmax-\sqrt{\log(2/\delta)/(2n)}}{1-\epsmax}\leq \Phi((X_\mathrm{med}-\theta)/\sigma)\leq \frac{1/2+\sqrt{\log(2/\delta)/(2n)}}{1-\epsmax},
\end{align*}
which is
\begin{align*}
    \Phi(\abs{X_\mathrm{med}-\theta}/\sigma)\leq \frac{1/2+\sqrt{\log(2/\delta)/(2n)}}{1-\epsmax}.
\end{align*}
When $n\geq 8\log(2/\delta)/(1-2\epsmax)^2$, the right-hand side is upper bounded by
\begin{align*}
    \frac{1/2+(1-2\epsmax)/4}{1-\epsmax}=\frac{3/4-\epsmax/2}{(1-\epsmax)}<1,
\end{align*}
and thus, on the uniform concentration event, we will have $\abs{X_\mathrm{med}-\theta}\leq D\sigma$. The desired conclusion then follows.
\end{proof}

With Lemma~\ref{lem:constant_var_est} and Lemma~\ref{lem:median}, we construct the estimator $\widetilde{\theta}$ required by Lemma~\ref{lem:pilot_unknown_var_meanpart} as follows. We first obtain the crude estimators $\breve{\sigma}^2$ and $X_\textnormal{med}$ from a holdout set of size $n/2$. Then, we define
\begin{align*}
    \cL(\mu,\lambda)\define \max_{t\in \cT}\Big(\abs*{\frac{1}{\epsmax}e^{-\im \mu t+\lambda t^2/2}\phi(t)-\frac{1-\epsmax}{\epsmax}}-1\Big)_+,
\end{align*}
where $\cT$ is a finite frequency grid to be specified later, and use the second half of the data to compute its empirical version
\begin{align*}
    \cL_n(\mu,\lambda)\define\max_{t\in \cT}\Big(\abs*{\frac{1}{\epsmax}e^{-\im \mu t+\lambda t^2/2}\phi_n(t)-\frac{1-\epsmax}{\epsmax}}-1\Big)_+.
\end{align*}
We define the estimator $\widetilde{\theta}$ by:
\begin{align}
    (\widetilde{\theta},\widetilde{\lambda})\in\argmin_{(\mu,\lambda)\in \cM\times \cV} \, \cL_n(\mu,\lambda),
\end{align}
where $\cM$ and $\cV$ are also finite grids to be specified later. We shall show that by choosing $\cT$, $\cM$ and $\cV$ appropriately, the resulting $\widetilde{\theta}$ satisfies the claimed property. Note that we can equivalently write
\begin{align*}
    \cL_n(\mu,\lambda)=\frac{1}{\epsmax}\max_{t\in\cT}\,\sup_{\xi\in\bbC:\, \abs{\xi}\leq 1}\,\abs*{e^{-\im \mu t+\lambda t^2/2}\phi_n(t)-(1-\epsmax)-\epsmax\xi}.
\end{align*}
Therefore, the estimator $\widetilde{\theta}$ is the same as the one in \cite{kotekal2025optimal} up to scaling, except that the candidate sets of $(t,\mu,\lambda)$ are all discretized.

According to Lemma~\ref{lem:constant_var_est} and Lemma~\ref{lem:median}, there are $(\delta,\epsmax)$-dependent absolute constants $C$, $D$, and $N_0$, such that the event
\begin{align*}
    \breve{\cE}=\{1/C\leq \breve{\sigma}^2/\sigma^2\leq C,\quad \abs{X_\textnormal{med}-\theta}\leq D\sigma\}
\end{align*}
holds with probability at least $1-\delta/2$ whenever $n\geq N_0$, with respect to the randomness of the holdout data ($P_\textnormal{holdout}^{n/2}$).
Concretely, we set:
\begin{enumerate}
    \item Three $\epsmax$-dependent constants
    \begin{align*}
        \alpha=\frac{\epsmax}{1-\epsmax}\in(0,1),\quad \beta=\sqrt{\frac{1+\alpha^2}{2}}\in(0,1),\quad \gamma=\frac{\beta}{\alpha}-1>0.
    \end{align*}
    \item $\cT_0 =[t_*,\, 2t_*]$ with $t_*=\sqrt{\log(en)/(8C)}/\breve{\sigma}$, and $\cT$ is a gird of $\cT_0$ with spacing at most
    \begin{align*}
        h_{\cT}=\frac{\arccos(\beta)}{2D\sqrt{C}\breve{\sigma}}.
    \end{align*}
    \item $\cM_0=[X_\textnormal{med}-2D\sqrt{C}\breve{\sigma},X_\textnormal{med}+2D\sqrt{C}\breve{\sigma}]$, and $\cM$ is an equally-spaced grid of $\cM_0$ with spacing at most
    \begin{align*}
        h_{\cM}=\frac{\gamma\epsmax}{16t_*(en)^{1/4}}.
    \end{align*}
    \item $\cV_0=[\breve{\sigma}^2/C,C\breve{\sigma}^2]$, and $\cV$ is an equally-spaced grid of $\cV_0$ with spacing at most
    \begin{align*}
        h_\cV=\frac{\gamma\epsmax}{16t_*^2(en)^{1/4}}.
    \end{align*}
\end{enumerate}

\begin{Claim}\label{clm:chf_union_bound}
With the above setup, for the event
\begin{align*}
    \widetilde{\cE}=\Big\{\max_{t\in\cT}\frac{e^{C\breve{\sigma}^2t^2/2}}{\epsmax}\abs{\phi_n(t)-\phi(t)}\leq  \frac{\gamma}{4}\Big\},
\end{align*}
we have
\begin{align*}
    P^{n/2}_\textnormal{train}(\widetilde{\cE})\geq 1 - 4\abs{\cT} e^{-\frac{1}{256}(\epsmax\gamma)^2\sqrt{n}}.
\end{align*}
\end{Claim}
\begin{proof}
On the training set, we use Hoeffding's inequality to see that
\begin{align*}
    \abs{\phi_n(t)-\phi(t)}\leq x,\quad \textnormal{w.p.}\ 1-4\exp(-nx^2/8),
\end{align*}
for each $t\in \cT$. Now, for each $t\in\cT$, we set $x_t=\frac{1}{4}\epsmax\cdot \gamma e^{-C\breve{\sigma}^2t^2/2}$ to see that
\begin{align*}
     P^{n/2}_\textnormal{train}\Big(\frac{1}{\epsmax}e^{C\breve{\sigma}^2t^2/2}\abs{\phi_n(t)-\phi(t)}> \frac{\gamma}{4}\Big)&\leq 4e^{-\frac{1}{128}n(\epsmax\gamma)^2\exp(-C\breve{\sigma}^2t^2)}\\
     &\leq 4e^{-\frac{1}{128}n\epsmax^2\gamma^2\exp(-4C\breve{\sigma}^2t_*^2)}\\
     &\leq  4e^{-\frac{1}{256}(\epsmax\gamma)^2\sqrt{n}},
\end{align*}
where we have used $\max \cT=2t_*=2\sqrt{\log(en)/(8C)}/\breve{\sigma}$. Taking the union bound over $t\in \cT$, we obtain the desired result.
\end{proof}

\begin{Claim}[Small $\cL_n$ at the truth]\label{clm:truth_small_loss}
With the above setup, we have
\begin{align*}
    \cL(\theta,\sigma^2)\leq 0.
\end{align*}
Thus, on $\widetilde{\cE}$, we have
\begin{align*}
    \cL_n(\theta,\sigma^2)\leq \frac{\gamma}{4}.
\end{align*}
\end{Claim}
\begin{proof}
Since
\begin{align*}
    \phi(t)=[(1-\epsmax)e^{\im \theta t}+\epsmax\xi(t)]\cdot e^{-\sigma^2t^2/2},
\end{align*}
we have
\begin{align*}
    \cL(\theta,\sigma^2)=(\abs{e^{-\im \theta t}\xi(t)}-1)_+\leq (1-1)_+=0.
\end{align*}
On the other hand, since $\abs{(\abs{z}-c)_+-(\abs{w}-c)_+}\leq \abs{z-w}$ for any $c\in \bbR$ and $z,w\in\bbC$, we have, on $\widetilde{\cE}$, that
\begin{align*}
    \abs{\cL_n(\theta,\sigma^2)-\cL(\theta,\sigma^2)}&\leq \max_{t\in \cT,\, \lambda\in\cV}\, \frac{1}{\epsmax}e^{\lambda t^2/2}\abs{\phi_n(t)-\phi(t)}\\
    &\leq  \max_{t\in \cT}\, \frac{1}{\epsmax}e^{C\breve{\sigma}^2 t^2/2}\abs{\phi_n(t)-\phi(t)}\\
    &\leq \frac{\gamma}{4},
\end{align*}
where we have used Claim~\ref{clm:chf_union_bound} in the last line.
\end{proof}

\begin{Claim}[Uniformly large $\cL_n$ at bad candidates]\label{clm:bad_large_loss}
    With the above setup, let $(\mu,\lambda)\in \cM_0\times \cV_0$ and write $r=\mu-\theta$. If $\breve{\cE}$ holds and
    \begin{align*}
        \abs{r}\geq \frac{2\pi}{t_*},
    \end{align*}
    then
    \begin{align*}
        \cL(\mu,\lambda)\geq \gamma.
    \end{align*}
    Thus, if $\widetilde{\cE}$ also holds, we have
    \begin{align*}
        \inf_{(\mu,\lambda)\in\cM_0\times \cV_0:\ \abs{\mu-\theta}\geq 2\pi/t_*}\, \cL_n(\mu,\lambda)\geq \frac{3\gamma}{4}.
    \end{align*}
\end{Claim}
\begin{proof}
Because $\abs{r}\geq 2\pi/t_*$, the interval $[\abs{r}t_*, 2\abs{r}t_*]$ has a length of at least $2\pi$, and therefore must contain an odd multiple of $\pi$. That is, there exist $k\in \bbZ$ and $t_1\in [t_*, 2t_*]$, such that
\begin{align*}
    (2k+1)\pi= \abs{r}t_1.
\end{align*}
By our construction of $\cT$, there $t_2\in\cT$ such that $\abs{t_1-t_2}\leq h_\cT$. Therefore,
\begin{align*}
    \abs{(2k+1)\pi-\abs{r}t_2}=\abs{r}\abs{t_1-t_2}\leq h_\cT \abs{r}\leq \frac{\arccos(\beta)}{2D\sqrt{C}\breve{\sigma}}\cdot 2D\sqrt{C}\breve{\sigma}\leq \arccos(\beta),
\end{align*}
which implies
\begin{align*}
    \cos(\abs{r}t_2)\leq -\cos(\arccos(\beta))=-\beta.
\end{align*}
Let us move to the structure of the characteristic function. We have
    \begin{align*}
        \cL(\mu,\lambda)&\geq \abs*{\frac{1}{\epsmax}e^{-\im \mu t_2+\lambda t_2^2/2}\phi(t_2)-\frac{1-\epsmax}{\epsmax}}-1.
    \end{align*}
Using $\phi(t)=[(1-\epsmax)e^{\im \theta t}+\epsmax \xi(t)]\cdot e^{-\sigma^2t^2/2}$, we can write
\begin{align*}
    \abs*{\frac{1}{\epsmax}e^{-\im \mu t_2+\lambda t_2^2/2}\phi(t_2)-\frac{1-\epsmax}{\epsmax}}&=\abs*{e^{(\lambda-\sigma^2)t_2^2/2}\Big(\frac{(1-\epsmax)e^{-\im rt_2}}{\epsmax}+e^{-\im \mu t}\xi(t)\Big)-\frac{1-\epsmax}{\epsmax}}\\
    &\geq \frac{1-\epsmax}{\epsmax}\abs*{1-e^{-\im rt_2+(\lambda-\sigma^2)t_2^2/2}}-e^{(\lambda-\sigma^2)t_2^2/2}\\
    &=\frac{1-\epsmax}{\epsmax}\sqrt{1+e^{(\lambda-\sigma^2)t_2^2}-2e^{(\lambda-\sigma^2)t_2^2/2}\cos(\abs{r}t_2)}-e^{(\lambda-\sigma^2)t_2^2/2}\\
    &\geq \frac{1-\epsmax}{\epsmax}\sqrt{1+e^{(\lambda-\sigma^2)t_2^2}+2e^{(\lambda-\sigma^2)t_2^2/2}\beta}-e^{(\lambda-\sigma^2)t_2^2/2}\\
    &\geq \frac{1-\epsmax}{\epsmax}\sqrt{[1+e^{(\lambda-\sigma^2)t_2^2}+2e^{(\lambda-\sigma^2)t_2^2/2}]\beta}-e^{(\lambda-\sigma^2)t_2^2/2}\\
    &=\frac{1-\epsmax}{\epsmax}\beta(1+e^{(\lambda-\sigma)t_2^2/2})-e^{(\lambda-\sigma^2)t_2^2/2},
\end{align*}
where we have used the previous conclusion $\cos(\abs{r}t_2)\leq-\beta$ in the third-to-last line. Therefore, 
\begin{align*}
     \cL(\mu,\lambda)&\geq \frac{1-\epsmax}{\epsmax}\beta(1+e^{(\lambda-\sigma)t_2^2/2})-e^{(\lambda-\sigma^2)t_2^2/2}-1\\
     &= (\frac{1-\epsmax}{\epsmax}\beta-1)(1+e^{(\lambda-\sigma)t_2^2/2})\\
     &\geq  \frac{1-\epsmax}{\epsmax}\beta-1\\
     &=\frac{\beta}{\alpha}-1\\
     &=\gamma.
\end{align*}
Finally, we apply Claim~\ref{clm:chf_union_bound} and proceed as in Claim~\ref{clm:truth_small_loss} to see that
\begin{align*}
        &\ \inf_{(\mu,\lambda)\in\cM_0\times \cV_0:\ \abs{\mu-\theta}\geq 2\pi/t_*}\, \cL_n(\mu,\lambda)\nonumber\\
        &\geq \inf_{(\mu,\lambda)\in\cM_0\times \cV_0:\ \abs{\mu-\theta}\geq 2\pi/t_*}\, \cL(\mu,\lambda)-\sup_{(\mu,\lambda)\in\cM_0\times \cV_0}\abs{\cL_n(\mu,\lambda)-\cL(\mu,\lambda)}\\
        &\geq \gamma-\max_{t\in\cT}\frac{e^{C\breve{\sigma}^2t^2/2}}{\epsmax}\abs{\phi_n(t)-\phi(t)}\\
        &\geq \gamma-\frac{\gamma}{4}=\frac{3\gamma}{4},
    \end{align*}
which concludes the proof.
\end{proof}
Combining Claim~\ref{clm:truth_small_loss} and Claim~\ref{clm:bad_large_loss} already shows that the truth $(\theta,\sigma^2)$ can beat all bad candidates $(\mu,\lambda):\abs{\mu-\theta}\gtrsim t_*^{-1}\asymp \sigma/\sqrt{\log(n)}$ if we minimize $\cL_n(\mu,\lambda)$ over continuous pilot intervals $\cM_0\times \cV_0$. That being said, when minimizing over $\cM_0\times \cV_0$, bad candidates will not become the minimizer on a high probability event. However, due to computational issues, we have chosen to minimize $\cL_n(\mu,\lambda)$ over the discrete set $\cM\times \cV$ instead, where the true values $(\theta,\sigma^2)$ may not be included. Therefore, we need to show further that a close-to-truth pair $(\mu,\lambda)\approx(\theta,\sigma^2)$ can also produce a small loss compared to all bad candidates. This can be done through a Lipschitz argument that relates $\cL_n(\mu,\lambda)$ to $\cL_n(\theta,\sigma^2)$.

\begin{Claim}[Lipschitzness of $\cL_n$]\label{clm:lipschitz}
    With the above setup, for any $(\mu,\lambda)$ and $(\mu',\lambda')$ in $\cM_0\times \cV_0$, we have
    \begin{align*}
        \abs{\cL_n(\mu,\lambda)-\cL_n(\mu',\lambda')}\leq \frac{2t_* (en)^{1/4}}{\epsmax}\abs{\mu-\mu'}+\frac{2t_*^2 (en)^{1/4}}{\epsmax}\abs{\lambda-\lambda'}.
    \end{align*}
\end{Claim}
\begin{proof}
Since $f(z)=(\abs{z}-c)_+$ is 1-Lipschitz on $\bbC$ for any $c\in\bbR$, we have
\begin{align*}
    \abs{\cL_n(\mu,\lambda)-\cL_n(\mu',\lambda')}&\leq \sup_{t\in\cT}\frac{\abs{\phi_n(t)}}{\epsmax}\abs{e^{-\im \mu t+\lambda t^2/2}-e^{-\im \mu't+\lambda't^2/2}}\\
    &\leq  \frac{1}{\epsmax}\cdot \sup_{t\in\cT}\, \abs{e^{-\im \mu t+\lambda t^2/2}-e^{-\im \mu't+\lambda't^2/2}}\\
    & =\frac{1}{\epsmax}\cdot \sup_{t\in\cT}\, \abs{F_t(\mu,\lambda)-F_t(\mu',\lambda')},
\end{align*}
where $F_t(u,v)\define e^{-\im ut+vt^2/2}$. Note that
\begin{align*}
    \partial_u F_t(u,v)=-\im t F_t(u,v),\quad \partial_v F_t(u,v)=\frac{t^2}{2}F_t(u,v).
\end{align*}
By the mean-value theorem, for each $t\in \cT$, there exist $\mu''_t$ between $\mu$ and $\mu'$ (and thus in $\cM_0$), and $\lambda''_t$ between $\lambda$ and $\lambda'$ (and thus in $\cV_0$), such that
\begin{align*}
    \abs{F_t(\mu,\lambda)-F_t(\mu',\lambda')}&=\abs*{-\im t F_t(\mu''_t,\lambda''_t)(\mu-\mu')+\frac{t^2}{2}F_t(\mu''_t,\lambda''_t)(\lambda-\lambda')}\\
    &\leq \abs{e^{-\im\mu''_t t+\lambda''_t t^2/2}}\cdot ( 2t_*\abs{\mu-\mu'}+2t_*^2\abs{\lambda-\lambda'})\\
    &\leq e^{C\breve{\sigma}^2 (2t_*)^2/2 }\cdot ( 2t_*\abs{\mu-\mu'}+2t_*^2\abs{\lambda-\lambda'})\\
    &=e^{2C\log(en)/(8C)}\cdot ( 2t_*\abs{\mu-\mu'}+2t_*^2\abs{\lambda-\lambda'})\\
    &=(en)^{1/4}\cdot ( 2t_*\abs{\mu-\mu'}+2t_*^2\abs{\lambda-\lambda'}).
\end{align*}
The desired conclusion is thus proved.
\end{proof}
We now derive the final claim, which asserts that a small loss $\cL_n(\mu,\lambda)$ can be achieved somewhere within the finite grid $\cM\times \cV$.
\begin{Claim}[Small $\cL_n$ within the grid]\label{clm:grid_small_loss}
With the above setup, there exists $(\mu_*,\lambda_*)\in\cM\times \cV$, such that
\begin{align*}
    \cL_n(\mu_*,\lambda_*)\leq \frac{\gamma}{2}
\end{align*}
on the joint event $\breve{\cE}\medcap \tilde{\cE}$.
\end{Claim}
\begin{proof}
On $\breve{\cE}$, we have $(\theta,\sigma^2)\in \cM_0\times \cV_0$. Since $\cM\times \cV$ is a discretization of $\cM_0\times \cV_0$ with step-sizes no larger than $(h_\cM,h_\cV)$, there exists $(\mu_*,\lambda_*)\in\cM\times \cV$ such that
\begin{align*}
    \abs{\mu_*-\theta}\leq h_\cM,\quad \abs{\lambda_*-\sigma^2}\leq h_\cV.
\end{align*}
By Claim~\ref{clm:lipschitz}, we have
\begin{align*}
    \cL_n(\mu_*,\lambda_*)&\leq \cL_n(\theta,\sigma^2)+\frac{2t_*(en)^{1/4}}{\epsmax}\abs{\mu_*-\theta}+\frac{2t_*^2(en)^{1/4}}{\epsmax}\abs{\lambda_*-\sigma^2}\\
    &\leq \cL_n(\theta,\sigma^2)+\frac{2t_*(en)^{1/4}}{\epsmax} h_\cM+\frac{2t_*^2(en)^{1/4}}{\epsmax}h_\cV.
\end{align*}
When $\tilde{\cE}$ also holds, we conclude by Claim~\ref{clm:truth_small_loss} that $\cL_n(\theta,\sigma^2)\leq \gamma/4$. Plugging in this bound and our preset magnitudes of $(h_\cM,h_\cV)$, we obtain that, on $\breve{\cE}\medcap \tilde{\cE}$,
\begin{align*}
    \cL_n(\mu_*,\lambda_*)&\leq \frac{\gamma}{4}+\frac{2t_*(en)^{1/4}}{\epsmax}\cdot \frac{\gamma\epsmax}{16t_*(en)^{1/4}}+\frac{2t_*^2(en)^{1/4}}{\epsmax}\cdot \frac{\gamma \epsmax}{16t_*^2(en)^{1/4}}\\
    &\leq \frac{\gamma}{2}.
\end{align*}
\end{proof}
We are now able to prove Lemma~\ref{lem:pilot_unknown_var_meanpart}.
\begin{proof}[Proof of Lemma~\ref{lem:pilot_unknown_var_meanpart}] Consider the joint event $\breve{\cE}\medcap \tilde{\cE}$. On this event, we combine Claim~\ref{clm:bad_large_loss} and Claim~\ref{clm:grid_small_loss} to see that
\begin{align*}
    \cL_n(\mu_*,\lambda_*)\leq \frac{\gamma}{2}<\frac{3\gamma}{4}&\leq  \inf_{(\mu,\lambda)\in\cM_0\times \cV_0:\ \abs{\mu-\theta}\geq 2\pi/t_*}\, \cL_n(\mu,\lambda)\\
    &\leq \inf_{(\mu,\lambda)\in\cM\times \cV:\ \abs{\mu-\theta}\geq 2\pi/t_*}\, \cL_n(\mu,\lambda)
\end{align*}
for some $(\mu_*,\lambda_*)\in\cM\times\cV$. Therefore, on this joint event, the solution to
\begin{align*}
    (\widetilde{\theta},\widetilde{\lambda})\in\argmin_{(\mu,\lambda)\in \cM\times\cV }\cL_n(\mu,\lambda)
\end{align*}
must satisfy
\begin{align*}
    \abs{\widetilde{\theta}-\theta}\leq \frac{2\pi}{t_*}=\frac{2\pi \breve{\sigma}}{\sqrt{\log(en)/(8C)}}\leq \frac{4\sqrt{2}\pi C\sigma}{\sqrt{\log(en)}},
\end{align*}
where the last inequality follows from $\breve{\sigma}^2\leq C\sigma^2$ on $\breve{\cE}$. Meanwhile, this joint event $\breve{\cE}\medcap \tilde{\cE}$ holds with probability at least
\begin{align*}
    P^n(\breve{\cE}\medcap \tilde{\cE})&\geq 1-\delta/2-\abs{\cT}e^{-\frac{1}{256}(\epsmax\gamma)^2\sqrt{n}}\\
    &\geq 1-\delta/2-(\ceil{\frac{t_*}{h_\cT}}+1)e^{-\frac{1}{256}(\epsmax\gamma)^2\sqrt{n}}\\
    &=1-\delta/2-\ceil*{\frac{D\sqrt{\log(en)}}{\sqrt{2}\arccos(\beta)}+1}e^{-\frac{1}{256}(\epsmax\gamma)^2\sqrt{n}}.
\end{align*}
For any fixed $\delta\in(0,1)$ and $\epsmax\in(0,1/2)$, the last term on the right-hand side tends to $0$ as $n$ tends to infinity. Therefore, we can select a $(\delta,\epsmax)$-dependent absolute constant $N$ to make this term smaller than $\delta/2$ whenever $n\geq N$.
    
\end{proof}

\subsection{Remaining Proofs for Section~\ref{sec:upper_bound_known_var}}
\subsubsection{Details about Remark~\ref{rmk:whether_interval_known_var}}\label{sec:explain_remark}
As we have remarked in Remark~\ref{rmk:whether_interval_known_var},  $\CIhat$ will become a true interval as long as we set $\kappa$ as a large enough $(\delta,\epsmax)$-dependent constant compared to $M$. This can be seen from the following arguments:
    \begin{itemize}
        \item For each $t\in \cT$, $\Upsilon_n(t;\mu)$ has period $2\pi/t$. From \eqref{eq:upsilon_to_cosine}, we know that the set $\{\mu\in\bbR:\, \abs{\Upsilon_n(t;\mu)}\leq 1+\Delta(t)\}$ is the union of infinitely many intervals; the centers of two consecutive intervals are separated by a distance of $2\pi/t$.
        \item By our choice of $\cT$ \eqref{eq:known_var_grid}, we have $2\pi/t\geq 2\pi\kappa \sigma/(1+\sqrt{\log(en)})$. When $\kappa$ is large enough compared to $M$, this gap $2\pi/t$ is larger than the length of the pilot interval $\widetilde{\CI}$ -- which is $M\sigma/\sqrt{\log(en)}$. When this holds, the set $\CIhat_t$ defined by \eqref{eq:CI_known_var_piece}, which is exactly the intersection between the union of intervals $\{\mu\in\bbR:\, \abs{\Upsilon_n(t;\mu)}\leq 1+\Delta(t)\}$ and the pilot confidence interval $\widetilde{\CI}$, will be one interval (including the special case $\emptyset$).
        \item Finally, recall that $\CIhat=\bigcap_{t\in\cT}\CIhat_t$, and note that the intersection of intervals is still an interval (or $\emptyset$ as its most special case).
    \end{itemize}

\subsection{Remaining Proofs for Section~\ref{sec:upper_bound_unknown_var}}
\subsubsection{Algorithmic Description in the Unknown-Variance Setting}\label{sec:alg_description}
In the following Algorithm~\ref{alg:unknown_var}, we present the algorithmic description of the confidence interval \eqref{eq:CI_unknown_var}.
\begin{algorithm}[ht]
  \caption{Certification-based construction of $\CIhat$ (unknown-variance)}
  \label{alg:unknown_var}
  \begin{algorithmic}[2]
    \Require Max contamination level $\epsmax\in(0,1/3]$; Data $X_1,\dots,X_n$; Pilot sets $\widetilde{\CI}$, $\widetilde{\cT}$, and $\widetilde{\cV}$
    \Ensure Confidence set $\CIhat$ for $\theta$

    \State $\CIhat \gets \varnothing$
    \For{$\text{each candidate }\mu \in \widetilde{\CI}$}\Comment{crude pilot interval}
        \For{$\text{each} $ $\lambda\in \widetilde{\cV}$}
            \State $\textsc{Cet}\gets \textsc{Pass}$
            \For{$\text{each }t\in\widetilde{\cT}$ }
                \State Compute $\Upsilon_n(t;\mu,\lambda)=3e^{-\im\mu t+\lambda t^2/2}\phi_n(t)-2$
                \If{$\abs{\Upsilon_n(t;\mu,\lambda)}>1+\Delta_1(t)$ }
                    \State $\textsc{Cet}\gets \textsc{Fail}$
                    \State \textbf{break}\Comment{end the loop over $t\in\widetilde{\cT}$}
                \ElsIf{$\abs{\Upsilon_n^2(t;\mu,\lambda)-\Upsilon_n(2t;\mu,\lambda)}+\abs{\Upsilon_n(t;\mu,\lambda)}^2>1+\Delta_2(t)$}
                    \State $\textsc{Cet}\gets \textsc{Fail}$
                    \State \textbf{break}\Comment{end the loop over $t\in\widetilde{\cT}$}
                \EndIf
            \EndFor
            \If{$\textsc{Cet}=\textsc{Pass}$}\Comment{all certificates are passed}
                \State $\CIhat \gets \CIhat \bigcup \{\mu\}$
                \State \textbf{break}\Comment{end the loop over $\lambda\in\widetilde{\cV}$}
            \EndIf
        \EndFor
    \EndFor
    \State $\CIhat\gets [\inf\CIhat,\ \sup \CIhat]$\Comment{(optional) make $\CIhat$ a true interval}
  \end{algorithmic}
\end{algorithm} 
\begin{Remark}
In analogy to Remark~\ref{rmk:known_var_discretize} following Algorithm~\ref{alg:known_var}, in Algorithm~\ref{alg:unknown_var}, we can also discretize $\widetilde{\CI}$ into pivots equally spaced by a gap of $\Theta(\widetilde{\sigma}n^{-1/8})$ to allow fast scanning over $\widetilde{\CI}$ without changing the order of the length.
\end{Remark}

\subsection{Remaining Proofs for Section~\ref{sec:discussion}}
\subsubsection{Proof of Theorem~\ref{thm:no_adaptation_between}}\label{sec:proof_no_adaptation_between}
Recall the conversion from confidence intervals to tests in Section~\ref{sec:lower_bound}. We define a half-Huber-half-Efron version of testing problems:
\begin{align}\label{eq:known_var_test_hybrid}
    H_0^{(\textnormal{Huber})}(\sigma,\epsmax)=\{(1-\epsmax)\cN(0,\sigma^2)+\varepsilon Q:\,Q\}\quad \textnormal{v.s.}\quad H_1^{(\textnormal{Efron})}(r,\sigma,\varepsilon)
\end{align}
and
\begin{align}\label{eq:unknown_var_test_hybrid}
    H_0^{(\textnormal{Huber})}(\epsmax)=\{(1-\epsmax)\cN(0,\sigma^2)+\varepsilon Q:\,\sigma,Q\}\quad \textnormal{v.s.}\quad H_1^{(\textnormal{Efron})}(r,\sigma,\varepsilon)
\end{align}
The following results are direct analogies to Proposition~\ref{prop:CI_to_test_known_var} and Proposition~\ref{prop:CI_to_test_unknown_var}.
\begin{Proposition}
Let $\delta\in(0,1)$ and $\epsmax\in(0,1/2)$ be fixed constants.
\begin{itemize}
    \item If $\sigma^2>0$ is known and a confidence interval $\CIhat$ satisfies
\begin{align*}
    \inf_\theta \,\inf_{\varepsilon\in[0,\epsmax],\,Q} (P_{\theta,\sigma,\varepsilon,Q}^\textnormal{(Huber)})^n(\theta\in\CIhat)\geq 1-\delta
\end{align*}
and
\begin{align*}
    \inf_\theta \,\inf_{\varepsilon\in[0,\epsmax],\,Q} (P_{\theta,\sigma,\varepsilon,Q}^\textnormal{(Efron)})^n(\abs{\CIhat}\leq r/2)\geq 1-\delta
\end{align*}
for some $r=r(n,\sigma,\varepsilon)>0$, then its induced test $T(X_{1:n})=\indi\{0\notin \CIhat\}$ satisfies
\begin{align*}
    \sup_{P\in H_0^{(\textnormal{Huber})}(\sigma,\epsmax)} P^n(T)+\sup_{P\in H_1^{(\textnormal{Efron})}(r,\sigma,\varepsilon)} P^n(1-T)\leq 3\delta.
\end{align*}
\item If $\sigma^2>0$ is unknown and a confidence interval $\CIhat$ satisfies
\begin{align*}
    \inf_{\theta,\sigma^2} \,\inf_{\varepsilon\in[0,\epsmax],\,Q} (P_{\theta,\sigma,\varepsilon,Q}^\textnormal{(Huber)})^n(\theta\in\CIhat)\geq 1-\delta
\end{align*}
and
\begin{align*}
    \inf_{\theta,\sigma^2} \,\inf_{\varepsilon\in[0,\epsmax],\,Q} (P_{\theta,\sigma,\varepsilon,Q}^\textnormal{(Efron)})^n(\abs{\CIhat}\leq r/2)\geq 1-\delta
\end{align*}
for some $r=r(n,\sigma,\varepsilon)>0$, then its induced test $T(X_{1:n})=\indi\{0\notin \CIhat\}$ satisfies
\begin{align*}
    \sup_{P\in H_0^{(\textnormal{Huber})}(\epsmax)} P^n(T)+\sup_{P\in H_1^{(\textnormal{Efron})}(r,\sigma,\varepsilon)} P^n(1-T)\leq 3\delta.
\end{align*}
\end{itemize}
\end{Proposition}
With this reduction approach, it remains to study the lower bound of the hybrid testing problems \eqref{eq:known_var_test_hybrid} and \eqref{eq:unknown_var_test_hybrid}. Although the exact detection thresholds seem complicated, we note that when $\varepsilon=0$, Efron's models and Huber's models are identical, and both \eqref{eq:known_var_test_hybrid} and \eqref{eq:unknown_var_test_hybrid} become Huber-versus-Huber problems that have been studied in \cite{luo2024adaptive}. The special $\varepsilon$ cases can serve as a lower bound for general $\varepsilon$, which are enough for our derivation of Theorem~\ref{thm:no_adaptation_between}.
\begin{proof}[Proof of Theorem~\ref{thm:no_adaptation_between}]
First, we consider the known-variance testing problem \eqref{eq:known_var_test_hybrid}. For the special case of $\varepsilon=0$, Lemma 4 of \cite{luo2024adaptive} shows that
\begin{align}
    r\gtrsim_{\delta,\epsmax}\frac{\sigma}{\sqrt{\log(n)}}
\end{align}
is required for detecting within a total error of $4\delta$. The first claim of Theorem~\ref{thm:no_adaptation_between} is thus proved. As for the second claim, we consider the unknown-variance counterpart \eqref{eq:unknown_var_test_hybrid}. Lemma 10 of \cite{luo2024adaptive} shows a lower bound of 
\begin{align*}
    r\gtrsim_{\delta,\epsmax}\sigma
\end{align*}
in this setting, which completes the proof of the second part.
\end{proof}

\section{Auxiliary Results}\label{sec:aux}
\begin{Lemma}\label{lem:cosine_quadratic}
    We have
    \begin{align}
        \frac{4}{\pi^2}x^2 \leq 1-\cos(x) \leq \frac{1}{2}x^2,\quad \forall\, x\in[-\pi/2,\pi/2],
    \end{align}
    and
    \begin{align}
        1-\cos(x)\geq \frac{2}{\pi^2}x^2,\quad \forall\, x\in[-\pi,\pi].
    \end{align}
\end{Lemma}

\begin{Lemma}\label{lem:chf_concentration}
    Consider an arbitrary probability distribution $P$ on $\bbR$. For any $\delta\in(0,1)$ and any $t_1>0$, we have
    \begin{align*}
        P^n\litb{\sup_{t\in t_1\cdot \sqrt{\bbZ_+}}t^{-1}\abs*{\frac{1}{n}\sum_{j=1}^ne^{itX_j}-\E_P[e^{itX}]}\leq t_1^{-1}\sqrt{\frac{8\log(5/\delta)}{n}}}\geq 1-\delta.
    \end{align*}
\end{Lemma}
\begin{proof}
    Since $\abs{\cos(tX)}\leq 1$, by Hoeffding's inequality, we have that for any $u\geq 0$
    \begin{align*}
        P^n\litb{\abs*{\frac{1}{n}\sum_{j=1}^n\cos(tX_j)-\E_P[\cos(tX)]}>ut}\leq 2e^{-n^2(ut)^2/2}=2e^{-n^2u^2t^2/2}.
    \end{align*}
    Let $u=t_1^{-1}\sqrt{2\log(5/\delta)/n}$, we have
    \begin{align*}
        \sum_{t\in t_1\cdot \sqrt{\bbZ_+}} 2e^{-n^2u^2t^2/2}&=2\sum_{k\in \bbZ_+}(\frac{\delta}{5})^{k}\\
        &\leq 2\cdot \frac{\delta/5}{1-\delta/5}\\
        &\leq \frac{2\delta/5}{1-1/5}\\
        &\leq \delta/2.
    \end{align*}
    This means that
    \begin{align*}
        \abs*{\frac{1}{n}\sum_{j=1}^n\cos(tX_j)-\E_P[\cos(tX)]}\leq \frac{t}{t_1}\sqrt{\frac{2\log(5/\delta)}{n}},\quad \forall\, t\in t_1\cdot \sqrt{\bbZ_+},
    \end{align*}
    with probability at least $1-\delta/2$. The same result holds for $\sin(tX)$. Finally, we use
    \begin{align*}
        \abs*{\frac{1}{n}\sum_{j=1}^ne^{itX_j}-\E_P[e^{itX}]}&\leq \abs*{\frac{1}{n}\sum_{j=1}^n\cos(tX)-\E_P[\cos(tX)]}+\abs*{\frac{1}{n}\sum_{j=1}^n\sin(tX)-\E_P[\sin(tX)]}
    \end{align*}
    to conclude the proof.
\end{proof}

\begin{Lemma}[\cite{vershynin2018high} Exercise 2.3.3]\label{lem:poisson_tail}
For $k> \lambda>0$, we have
\begin{align*}
    \Prob\litb{\mathrm{Pois}(\lambda)\geq k}\leq (e\lambda/k)^ke^{-\lambda}. 
\end{align*}
\end{Lemma}

\begin{Lemma}[\cite{stein2010complex} Chapter 5.3.2]\label{lem:Mittag-Leffler}
The meromorphic function $\pi\cot(\pi z)$ admits its Mittag-Leffler expansion:
\begin{align*}
    \pi\cot(\pi z)=\frac{1}{z}+\sum_{j=1}^{+\infty}\frac{2z}{z^2-j^2},\quad z\in\bbC\backslash \bbZ.
\end{align*}
\end{Lemma}

\begin{Lemma}\label{lem:cot}
Let $f(x)=1-\pi x\cot(\pi x)$. Then, $f(x)$ is increasing on $(0,1)$ and $\lim_{x\to 0^+}f(x)=0$.
\end{Lemma}

\begin{Lemma}\label{lem:control_alternation}
For $K$ being a positive multiple of $4$ and $x\in(0,1)$, we have
\begin{align*}
    \abs*{\sum_{k=0}^L(-1)^{K-k}\binom{K}{k}\frac{1}{(k-L/2+x)}-\frac{1}{x}\binom{K}{K/2}}\leq \frac{1-\pi x\cot(\pi x)}{x}\binom{K}{K/2}.
\end{align*}
Specifically, if $x\in(0,1/4)$, we have
\begin{align*}
    \frac{\pi/4}{x}\binom{K}{K/2}\leq \sum_{k=0}^L(-1)^{K-k}\binom{K}{k}\frac{1}{(k-L/2+x)}\leq \frac{2-\pi/4}{x}\binom{K}{K/2}.
\end{align*}
\end{Lemma}
\begin{proof}
We have
\begin{align*}
    &\ \abs*{\sum_{k=0}^K(-1)^{K-k}\binom{K}{k}\frac{1}{(k-K/2+x)}-\frac{1}{x}\binom{K}{K/2}}\\
    &=\abs*{\sum_{k=0}^{K/2-1}(-1)^{K-k}\binom{K}{k}\litb{\frac{1}{(k-K/2+x)}+\frac{1}{(K/2-k+x)}}}\\
    &=\abs*{\sum_{k=0}^{K/2-1}(-1)^{K-k}\binom{K}{k}\frac{2x}{x^2-(k-K/2)^2}}\\
    &\leq \binom{K}{K/2}\sum_{j=1}^{+\infty}\frac{2x}{j^2-x^2}\\
    &=\frac{1-\pi x\cot(\pi x)}{x}\binom{K}{K/2},
\end{align*}
where the last line comes from the Mittag-Leffler expansion in Lemma~\ref{lem:Mittag-Leffler}. The second conclusion follows by applying the triangular inequality and noting from Lemma~\ref{lem:cot} that $1-\pi x\cot(\pi x)\in (0,1-\pi/4)$ when $x\in(0,1/4)$.
\end{proof}

\begin{Lemma}\label{lem:finite_support}
    Let $\{m_0=1,m_1,m_2,\cdots\,m_k\}$ be a truncated moment sequence of some probability measure on $\bbR$. Then, we can find a probability measure $P$ supported on at most $k+1$ points such that $\E_P[X^j]=m_j$ for $j=0,1,2,\cdots,k$.
\end{Lemma}
\begin{proof}
By definition, since $m_{0:k}$ is a truncated moment sequence, there exists a probability measure $Q$ such that
\begin{align*}
    \E_Q[X^j]=m_j,\quad j=0,1,2,\cdots,k.
\end{align*}
Define the set of moment curves
\begin{align*}
    \cC=\{(x,x^2,\cdots,x^k):\ x\in\bbR\}\subset \bbR^k.
\end{align*}
Then, we have
\begin{align*}
    (m_1,m_2,\cdots,m_k)=\E_{X\sim Q}[(X,X^2,\cdots,X^k)]\in\mathrm{conv}(\cC).
\end{align*}
By Carath\'{e}odory's convex hull theorem (\cite{caratheodory1907variabilitatsbereich}, see also Proposition A.35 of \cite{schmudgen2017moment}), for any subset $T\subseteq \bbR^k$, any point in $\mathrm{conv}(T)$ can be written as the convex combination of at most $k+1$ points in $T$. This implies the existence of atoms $y_{1:(k+1)}$ and weights $w_{1:(k+1)}$ such that
\begin{align*}
    (m_1,m_2,\cdots,m_k)=\sum_{i=1}^{k+1} w_i (y_i, y_i^2,\cdots y_i^k).
\end{align*}
The probability measure $P=\sum_{i=1}^{k+1} w_i\delta_{y_i}$ satisfies the desired properties.
\end{proof}

\end{sloppypar}

\end{document}